\documentclass{sn-jnl}

\usepackage{graphicx}%
\usepackage{multirow}%
\usepackage{amsmath,amssymb,amsfonts}%
\usepackage{amsthm}%
\usepackage{mathrsfs}%
\usepackage[title]{appendix}%
\usepackage{xcolor}%
\usepackage{textcomp}%
\usepackage{manyfoot}%
\usepackage{booktabs}%
\usepackage{subfig}
\usepackage{algorithm}%
\usepackage{algorithmicx}%
\usepackage{algpseudocode}%
\usepackage{listings}%
\usepackage{bm}
\usepackage{tikz-cd}
\usepackage{extarrows}

\newcommand{\m}{\textbf{-}}
\newcommand{\nm}[1]{\| #1 \|}
\newcommand{\bnm}[1]{\bigl\| #1 \bigr\|}
\newcommand{\Bnm}[1]{\Bigl\| #1 \Bigr\|}
\newcommand{\dt}{\,\mathrm{d}t}

\numberwithin{equation}{section}
\theoremstyle{thmstyleone}%
\newtheorem{theorem}{Theorem}
\newtheorem{proposition}[theorem]{Proposition}%
\newtheorem{lemma}{Lemma}[section]%
\newtheorem{assumption}{Assumption}

\theoremstyle{thmstyletwo}%
\newtheorem{remark}{Remark}%

\theoremstyle{thmstylethree}%
\newtheorem{definition}{Definition}%

\begin{document}

\title[Article Title]{Asymptotic Recovery in Fourier Spectral Methods for the Schr\"odinger Equation with Point Singularities}


\author[1]{\fnm{Yanjie} \sur{Li}}

\author[2]{\fnm{Sihong} \sur{Shao}}

\affil[1]{School of Mathematical Sciences, Peking University, Beijing 100871, China. Email: {\tt 2401110048@stu.pku.edu.cn}}
\affil[2]{ CAPT, LMAM and School of Mathematical Sciences, Peking University, Beijing 100871, China. Email: {\tt sihong@math.pku.edu.cn}}


\abstract{
\unboldmath
This paper studies the Fourier spectral method (FSM) for the Schr\"odinger equation with singular potentials $V \in H^{s}$, where $s > \max\{d/2-2,-1\}$ and $d$ denotes the spatial dimension. 
This setting includes a broad class of singular potentials, such as the 3D Coulomb potential and the 1D Dirac-delta potential.
First, we combine the Feshbach-Schur map with a refined perturbation argument to derive sharp convergence orders for FSM, yielding order $2s+2$ for eigenvalues and order $s+1$ for eigenfunctions in the $H^1$ norm. 
More importantly, the $H^1$ error with respect to the projected eigenfunction converges with a higher order $s+1+b$, where
$
b=\min\{s+2-d/2-\varepsilon,\; s+1,\; 2\}>0
$
for arbitrarily small $\varepsilon>0$, revealing a super-convergence phenomenon.
Second, in the presence of potentials with isolated point singularities, we develop an asymptotic-recovery (AR) technique to post-process the FSM solutions. The resulting method, dubbed AR-FSM, fully exploits the super-convergence property and achieves convergence orders $2s+2+2b$ for eigenvalues and $s+1+b$ for eigenfunctions in the $H^1$ norm, while the AR post-processing requires only a computational cost that is linear in the number of FSM degrees of freedom.
The analysis introduces a rigorous definition of point singularities and develops a foundational framework for their study. It further establishes an asymptotic expansion of eigenfunctions consisting of a regular component in $H^{s+4}$ together with $d+1$ asymptotic functions associated with each singular point. 
Numerical experiments confirm the sharpness of these theoretical bounds.
}

\keywords{singular potentials,
asymptotic recovery,
eigenvalue problem,
Fourier spectrum methods,
super-convergence}


\pacs[MSC Classification]{65N15, 65N35, 65T40, 35J10, 81Q05}

\maketitle

\section{Introduction}
Accurate calculation of eigenpairs of the Schr\"odinger operator is of broad scientific and engineering importance, motivating the development of numerous numerical schemes. In practice, however, physically relevant potentials often contain point singularities, such as the Coulomb potential in electronic systems or Dirac‑delta interactions that model point defects and impurities \cite{dean2021delta,holmer2007delta}. The presence of such singularities can cause standard numerical methods to break down or suffer from severely degraded performance. This limitation creates a strong demand for robust, high‑order accurate schemes capable of handling singular potentials efficiently.

Several approaches have been proposed to address the difficulty caused by singular potentials. 
Pseudo-potential methods \cite{troullier1991PP,kresse1994PP,ramer1999PP,hamann2013PP} seek to replace the singular potential with an effective one, thus lifting the regularity of the solution. 
Other strategies include the use of linear combinations of atomic orbitals \cite{bachmayr2014LCAO,scholz2017LCAO,li2017invpow}, as well as explicit correlation methods \cite{shiozaki2009EC,ten2012EC,kong2012EC,li2012EC,hattig2012EC}, which incorporate electron–electron correlations directly into the approximation space.

From the perspective of Fourier spectral methods (FSM), also known as plane-wave methods, the global Fourier basis offers several attractive features, including natural periodicity, exponential convergence for smooth problems, and efficient implementation via the Fast Fourier Transform (FFT). 
However, the presence of point singularities in the potential severely degrades the performance of FSM. 
The underlying difficulty stems from the cusp-like local behavior of the eigenfunction near each singular point. 
These cusps lead to slow decay of the Fourier coefficients, which limits the overall convergence order. To address this issue within the plane-wave framework, several hybrid approaches have been developed. 
Among them, the linearized augmented plane-wave method \cite{madsen2001LAPW,blugel2006LAPW} and the projector augmented-wave method \cite{blochl1994PAW,kresse1999PAW} are among the most widely used techniques in electronic structure calculations.

From a theoretical standpoint, it has been established in \cite{cances2010THM} that FSM achieves convergence of order $2s+2$ for eigenvalues and order $s+1$ for eigenfunctions in the $H^1$ norm when the potential satisfies $V\in H^{s}$ with $s>d/2$, where $d$ denotes the spatial dimension and $H^{s}$ is the periodic Sobolev space. The same convergence orders were shown for piecewise smooth potentials $V\in H^{1/2-\varepsilon}$ in two dimensions \cite{norton2010FSM-Theorem}. For singular potentials with $s\le d/2$, however, rigorous convergence analysis of FSM remains limited. To the best of our knowledge, \cite{dusson2023FSM-Theorem} provides one of the few available results in this direction, showing that FSM achieves convergence  of order $2r$ for eigenvalues and eigenfunctions in the $L^2$ norm, provided that $V$ acts as a bounded operator from $H^{1-r}$ to $H^{-1+r}$. However, these orders appear not to be sharp in general.

In this paper, we impose the following assumption on the potential. This assumption includes a wide range of singular potential including the 1D Dirac delta potential ($s=-0.5^{\m}$) as well as the 3D Coulomb potential ($s=0.5^{\m}$), where the notation $a^{\m}$ denotes $a-\varepsilon$ for arbitrarily small $\varepsilon>0$.
\begin{assumption}\label{Asp::FSM}
The potential satisfies $V \in H^s$ with $s > \max\{d/2-2,\,-1\}$.
\end{assumption}

Under this assumption, we adapt the Feshbach-Schur map  developed in \cite{dusson2023FSM-Theorem} to obtain sharp convergence orders for FSM. The main improvement in the analysis is a perturbation argument performed directly on the first $n$ eigenvalues, which allows full exploitation of the regularity of the corresponding eigenfunctions. As a result, FSM achieves convergence orders $2s+2$ for eigenvalues and $s+1$ for eigenfunctions in the $H^1$ norm. Moreover, the $H^1$ error with respect to the projected eigenfunction converges with order $s+1+b$, where $b=\min\{(s+2-d/2)^{\m},s+1,2\}$, revealing a super-convergence phenomenon. This result extends the analysis for regular potentials in \cite{cances2010THM,cances2018THM} to the singular setting, constituting the first contribution of this paper.

Motivated by the super-convergence property, we develop an asymptotic-recovery (AR) technique for potentials with point singularities, providing a post-processing procedure to recover the coefficients in the asymptotic expansion of eigenfunctions and thereby reconstruct the high-frequency content of the FSM solution. The resulting method, termed AR-FSM, fully exploits the super-convergence property and achieves convergence orders $2s+2+2b$ for eigenvalues and $s+1+b$ for eigenfunctions in the $H^1$ norm.
The analysis introduces a rigorous definition of point singularities and develops a foundational framework for their study. Within this framework, we establish an asymptotic expansion of eigenfunctions consisting of a regular component in $H^{s+4}$ together with $d+1$ asymptotic functions associated with each singular point. These results provide a precise description of the high-frequency behavior of the Fourier coefficients and are of independent interest beyond the AR-FSM framework. This constitutes the second contribution of this paper.

\subsection{Problem set up}
Consider the $d$-dimensional Schr\"odinger operator
\begin{equation}
\label{Eq-1.1-01}
\mathcal{H} := \mathcal{T}+\mathcal{V} =-\nabla^2_{\bm{x}} + V(\bm{x}),
\end{equation}
with a periodic potential $V$ on $\Omega = [0, 1]^d$, acting on the Hilbert space
\[L_{\text{per}}^2(\Omega):=\{ \psi(\bm{x})\in L_{\text{loc}}^2(\mathbb{R}^d) | \text{$\psi(\bm{x})$ is periodic on $\Omega$}\},\]
equipped with the inner product $\langle f, g \rangle := \int_{\Omega} \overline{f(\bm x)}{g(\bm x)}  \, dx$.
Under Assumption~\ref{Asp::FSM}, the operator $\mathcal{H}$ admits a non-decreasing sequence of eigenpairs $(\lambda_n, \psi_n) \in \mathbb{R} \times H^1$, $n \in \mathbb{N}^+$, with $\lambda_1 \leq \lambda_2 \leq \cdots \leq \lambda_n \leq\cdots  $, satisfying
\begin{equation}
\label{Eq-1.1-02}
\langle \phi, \mathcal{H} \psi_n \rangle = \lambda_n \langle \phi, \psi_n \rangle \quad \forall\, \phi \in H^1.
\end{equation}
Here, $H^s$ denotes the periodic Sobolev space of order $s \in \mathbb{R}$ with the norm
\begin{equation}
\label{Eq-1.1-03}
\| \psi \|_{H^s} :=\left(\sum_{\bm{k} \in \mathbb{Z}^d} (4\pi^2|\bm{k}|^2 + 1)^{s} |\hat{\psi}_{\bm{k}}|^2\right)^{1/2},
\end{equation}
where $\hat{\psi}_{\bm k}  = \langle e_{\bm k},\psi\rangle$ is the Fourier coefficients of $\psi$.  and $e_{\bm k}(\bm x) = e^{2\pi i \bm k\cdot \bm x}$
is the Fourier basis. The goal is to develop numerical methods for computing the eigenpairs $(\lambda_n, \psi_n)$ defined by \eqref{Eq-1.1-02}.  

Throughout this paper,  we write $A \lesssim B$ to denote that there exists a constant $C$ independent of $N$ and $\psi$ such that $A \le C B$.  
For simplicity, we also omit the domain $\Omega$ when referring to periodic function spaces (for instance, $H^s$ and $W^{t,\infty}$ in Lemma~\ref{Lem5.5}), whereas the domain is written explicitly for non-periodic function spaces (for instance, $H^s(\mathbb{R}^d)$ and $H^s(B_R)$ in Lemma~\ref{Lem6.3}).

\subsection{Fourier spectral method}
\label{subsec:FM}
FSM solves \eqref{Eq-1.1-02} by expanding the eigenfunction $\psi$ in a finite Fourier basis. 
For truncation parameter $N \in \mathbb{N}^+$, define the truncated index set and the corresponding finite-dimensional function space by
\[\mathscr{K}_N := \left\{ \bm k\in \mathbb{Z}^d\ | \ \| \bm k \|_{\infty}\leq N\right\},\quad \mathscr{X}_N:=\text{span}\{ e_{\bm k}|\ \bm k\in \mathscr{K}_N\}.\] 
Applying the Galerkin method to \eqref{Eq-1.1-02} over $\mathscr{X}_N$
leads to the algebraic eigenvalue problem: Find 
$(\lambda_n^N,\psi_n^N)\in \mathbb{R}\times \mathscr{X}_N$ such that
\begin{equation}
\label{Eq-1.2-01}
    \langle \phi, \mathcal{H}\psi_n^N\rangle
    = \lambda_n^N \langle \phi,\psi_n^N\rangle,
    \qquad \forall\, \phi \in \mathscr{X}_N.
\end{equation}
Let $\mathcal{P}_N$ denote the orthogonal projection onto $\mathscr{X}_N$, defined by
\[
\mathcal{P}_N \psi
:= \sum_{\bm k \in \mathscr{K}_N}
\langle e_{\bm k},\psi\rangle e_{\bm k}.
\]
Then \eqref{Eq-1.2-01} can equivalently be written as
$
\mathcal{P}_N \mathcal{H} \mathcal{P}_N \psi_n^N
=
\lambda_n^N \psi_n^N,
$
Our convergence analysis for FSM follows from the key lemma below.
\begin{lemma}\label{Lem1.1}
Under Assumption~\ref{Asp::FSM}, let
$
b = \min\{(s+2-d/2)^{\m},\, s+1,\, 2\}.
$
Then for fixed $n \in \mathbb{N}^+$ and sufficiently large $N$, we have 
\begin{equation}
\|\psi_n^N - \mathcal{P}_N \psi\|_{H^1}
\lesssim N^{-(s+1+b)}
\|\psi_n^N\|_{L^2},
\end{equation}
where $\psi \in \operatorname{ker}(\lambda_n - \mathcal{H})$ is an exact eigenfunction associated with $\lambda_n$.
\end{lemma}
Based on this lemma, we prove the optimal convergence order of FSM.
\begin{theorem}[\textbf{Convergence order of FSM}] \label{Thm01}
Under Assumption~\ref{Asp::FSM}, for fixed $n \in \mathbb{N}^+$ and sufficiently large $N$, we have
\begin{equation}
\|\psi_n^N - \psi\|_{H^1} \lesssim N^{-(s+1)}\|\psi\|_{L^2},\qquad |\lambda_n^N - \lambda_n| \lesssim N^{-(2s+2)}.
\end{equation}
\end{theorem}
The proofs of Lemma~\ref{Lem1.1} and Theorem~\ref{Thm01} are deferred to Section~\ref{sec:proofFM}.

The estimate in Lemma~\ref{Lem1.1} implies a super-convergence of order $b$ for the eigenfunctions in the $H^1$ norm. 
This observation naturally suggests the use of post-processing techniques to exploit this super-convergence. 
In particular, two-grid methods \cite{cances2016post,cances2021post,dusson2021post2} have been proposed for this purpose. 
These approaches solve the problem on a coarse grid and then lift the solution to a finer grid, thereby reducing the computational cost of the eigensolver.  However, the overall computational complexity includes the cost of FFT performed on the finer Fourier grid. 
Consequently, while these methods improve efficiency, they do not fundamentally enhance the convergence order.

\subsection{An illustration of AR} %
\begin{figure}[t]
\centering
\subfloat[$\psi_{\mathrm{num}}=\mathcal{P}_N\psi$]{\label{fig:1a}\includegraphics[width=0.45\linewidth]{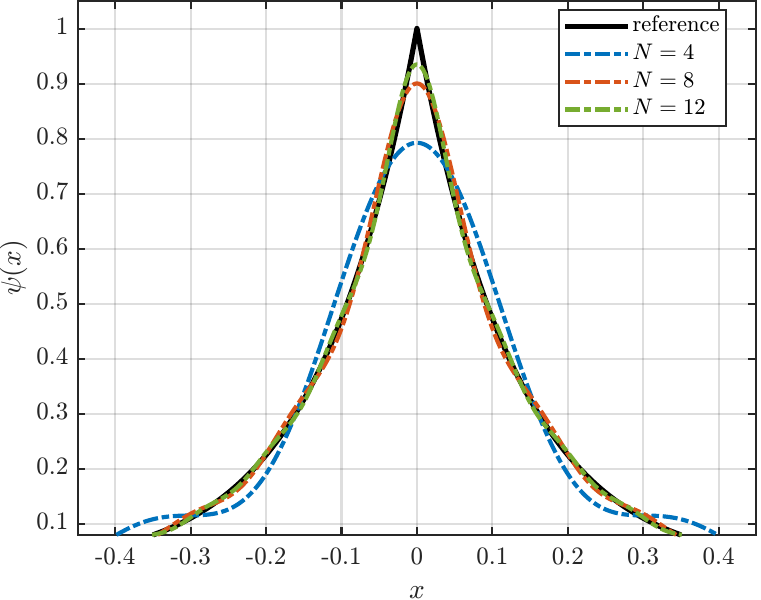}}\quad
\subfloat[$\psi_{\mathrm{num}}=\mathcal{A}_N\mathcal{P}_N\psi$]{\label{fig:1b}\includegraphics[width=0.45\linewidth]{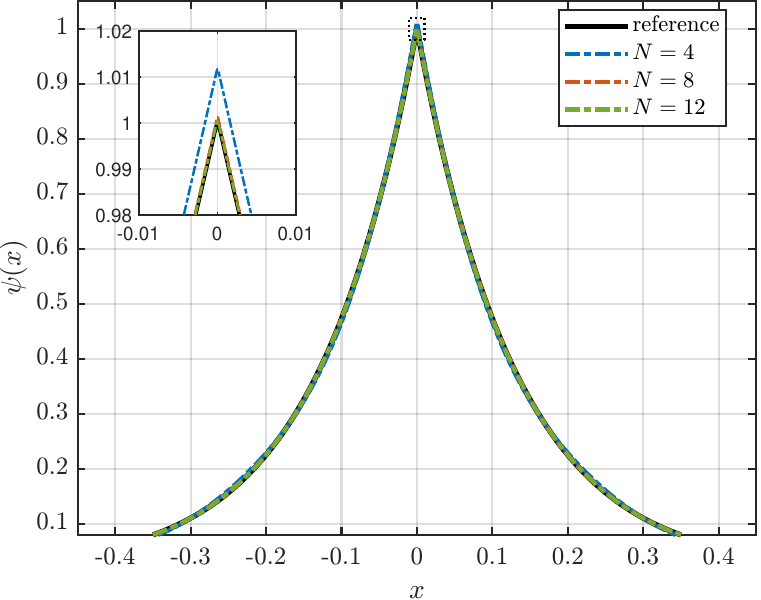}}
\caption{
Numerical eigenfunctions of the one-dimensional Schr\"odinger equation with a Dirac-delta potential $V(x) = -15\delta(x)$.
Figure (a) shows the Fourier projection of $\psi$, while figure (b) displays the result after applying the AR post-processing \eqref{Eq-1.3-04} to the Fourier projection.
}
\label{fig:01}
\end{figure}

The proposed AR post-processing is based on decomposing the eigenfunction into the low and high-frequency components, in the spirit of two-grid methods. However, instead of representing the high-frequency part on a finer Fourier grid, we approximate it using the intrinsic asymptotic behavior of the eigenfunctions. To illustrate this idea, we provide a simple yet representative example. Consider the one‑dimensional Schr\"odinger operator with Dirac-delta potential $V(x) = \gamma\delta(x)$. The weak form of the eigenvalue problem \eqref{Eq-1.1-02} then reads
\begin{equation}\label{Eq-1.3-01}
    \langle \phi,\mathcal{T}\psi\rangle +\gamma\overline{\phi(0)}\psi(0)  = \lambda \langle \phi,\psi\rangle , \quad \forall\, \phi \in H^{1}.
\end{equation}

Substituting the test function $\phi(x) = e_{k}(x)$ into the weak formulation \eqref{Eq-1.3-01} gives
\begin{equation}\label{Eq-1.3-02}
    \hat{\psi}_k = \frac{-\gamma\psi(0)}{4\pi^2 k^2-\lambda}.
\end{equation}
This relation reveals the leading asymptotic decay $\hat{\psi}_k \sim-\gamma\psi(0)/4\pi^2k^2$ of the eigenfunction $\psi(x)$. To match this decay, we introduce the asymptotic function
\begin{equation}\label{Eq-1.3-03}
\Phi(x) = 1+\gamma\frac{\sin(\pi|x|)}{2\pi} 
\quad\Longleftrightarrow\quad
\hat{\Phi}_k = \delta_{0k}+\frac{-\gamma}{4\pi^2k^2-\pi^2},
\end{equation}
whose Fourier coefficients reproduce the same $k^{-2}$ asymptotic decay. 
Let $\psi^N \in \mathscr{X}_N$ be a numerical solution. 
We define the AR operator $\mathcal{A}_N$ by introducing a high-frequency correction
\begin{equation}\label{Eq-1.3-04}
\mathcal{A}_N \psi^N := \psi^N + \beta \mathcal{P}_N^\perp \Phi = \underbrace{\psi^N-\beta\mathcal{P}_N\Phi}_{\in\mathscr{X}_N} +\beta\Phi,
\end{equation}
where $\mathcal{P}_N^{\perp} = 1-\mathcal{P}_N$ is the projection operator onto high-frequency space. The coefficient $\beta$ is then determined by substituting $\psi = \mathcal{A}_N \psi^N$ into \eqref{Eq-1.3-02} and matching coefficient of the $k^{-2}$ decay,
\begin{equation}\label{Eq-1.3-05}
\beta = [\mathcal{A}_N \psi](0)
= \left[\psi^N - \beta \mathcal{P}_N \Phi\right](0) + \beta
\quad \Longrightarrow \quad
\left[\psi^N - \beta \mathcal{P}_N \Phi\right](0) = 0.
\end{equation}
 This forms a linear equation for $\beta$, where the coefficients can be evaluated within the complexity $\mathcal{O}(N)$. As presented in Figure~\ref{fig:01}, this simple correction yields an evident improvement in the post-processed eigenfunction. 

\begin{table}[t]
\centering
\footnotesize
\caption{Theoretical orders of convergence for eigenvalues and eigenfunctions in the $H^1$ norm for FSM and AR-FSM. 
The notation $a^{\m}$ denotes $a-\varepsilon$ for arbitrarily small $\varepsilon>0$. 
}
\label{tab:01}
\begin{tabular}{c cc cc}
\toprule
& \multicolumn{2}{c}{FSM} & \multicolumn{2}{c}{AR-FSM} \\
\cmidrule(lr){2-3} \cmidrule(lr){4-5}
Potential  
& $H^1$ & Energy 
& $H^1$ & Energy \\
\midrule
$V\in H^{s}$ & $s+1$ & $2s+2$ & $s+1+b$ & $2s+2+b$ \\
\addlinespace
$d=1,\ V_{\mathrm{lc}}=\gamma\delta(x)$ & $0.5^{\m}$ & $1.0^{\m}$ & $1.0^{\m}$ & $2.0^{\m}$ \\
\addlinespace
$d=3,\ V_{\mathrm{lc}}\sim \gamma/|\bm x|$ & $1.5^{\m}$ & $3.0^{\m}$ & $2.5^{\m}$ & $5.0^{\m}$ \\
\bottomrule
\end{tabular}
\end{table}
This example illustrates how the AR technique reconstruct the high-frequency components and improve the accuracy of the FSM solution. 
This observation extends to a general class of potentials with isolated point singularities. 
In particular, the AR post-processing technique fully exploits the super-convergence of FSM implied by Lemma~\ref{Lem1.1}. 
We summarize the resulting convergence properties as follows.
\begin{theorem}[\bf Convergence order of AR-FSM]\label{Thm02}
  Let $V$ satisfy Assumption~\ref{Asp::ARFSM}. Then the estimates for the post-processed eigenfunction $\tilde{\psi}_n^N:=\mathcal{A}_N\psi^N_n$ and the eigenvalue $\tilde{\lambda}_n^N:=\langle\tilde{\psi}_n^N,\mathcal{H}\tilde{\psi}_n^N\rangle/\langle\tilde{\psi}_n^N,\tilde{\psi}_n^N\rangle$ hold
\begin{equation}
\|\tilde{\psi}_n^N - \psi\|_{H^1} \lesssim N^{-(s+1+b)}\|\psi\|_{L_2},\qquad |\tilde{\lambda}_n^N - \lambda_n| \lesssim N^{-(2s+2+2b)}.
\end{equation}
\end{theorem}

While FSM achieves convergence orders of $s+1$  
eigenfunctions in the $H_1$ norm and $2s+2$ for eigenvalues, AR-FSM enhances these to $s+1+b$ and $2s+2+2b$, as stated in Theorem~\ref{Thm02}. The additional computational cost in the post-processing is dominated by the assembly of a linear system for recovering the asymptotic coefficients, and scales linearly with the number of FSM degrees of freedom.
A summary of the convergence orders for representative singular potentials is presented in Table~\ref{tab:01}.

The remainder of this paper is organized as follows. 
In Section~\ref{sec:expansion}, we introduce a formal definition of point singularities and derive asymptotic expansions of eigenfunctions. 
These expansions provide the foundation for the AR post-processing, which is detailed in Section~\ref{sec:postproc}. 
Numerical experiments are presented in Section~\ref{sec:experiments}. 
The proofs of the theoretical results are deferred to Sections~\ref{sec:proofFM} and \ref{sec:proofAR}. 
Finally, Section~\ref{sec:conclusions} concludes the paper and discusses some future directions.

\section{Expansion of eigenfunctions in point-singular potential}
\label{sec:expansion}
This section is devoted to deriving an expansion of eigenfunction $\psi$ of the form
\begin{equation}\label{Eq-2.0-01}
\psi =
\underbrace{\psi^{(0)}}_{\in H^{s+2}}
+
\underbrace{\psi^{(1)}}_{\in H^{s+3}}
+
\underbrace{\psi^{\geq 2}}_{\in H^{s+4}}.
\end{equation}
We retain only the first two orders of the expansion, as this is sufficient for the post-processing of FSM, for which the super-convergence gain $b \leq 2$ according to Lemma~\ref{Lem1.1}.

We begin by introducing a formal definition of point singularities, followed by the asymptotic expansion of eigenfunction.

\subsection{Potential with point singularities}
Let $V$ be a potential with a point singularity at $\bm x=\bm 0$.  Intuitively, the asymptotic decay in Fourier coefficients of $\mathcal{{V}}\psi$ is governed by the local behavior
of $\psi$ near $\bm x=\bm 0$. More specifically, for a function $\psi$ with sufficient regularity to admits a Taylor expansion, the action of $V$ on monomials determines the high-frequency asymptotic decay of $V\psi$. 
We therefore adopt the following definition.
\begin{definition}[Periodic point‑singular Sobolev space]\label{Def01}
   Let $s\in\mathbb R$ and $r\in\mathbb N$. 
    a function $V$ is said to belong to $PSH^{s,r}$, if and only if, there exists a compactly supported function $V_{\mathrm{lc}} $ such that
    \begin{equation}\label{Eq-defPSH-01}
        V(\bm{x})
        =
        \sum_{\bm n \in \mathbb{Z}^d}
        V_{\mathrm{lc}}(\bm{x}+\bm n),
    \end{equation}
    and  for every multi-index  $\bm{\alpha} \in \mathbb{N}^{d}$ with $|\bm{\alpha}| \le r$, 
    \begin{equation}\label{Eq-defPSH-02}
    \bm x^{\bm \alpha}V_{\mathrm{lc}}(\bm x) \in H^{s+|\bm \alpha|}(\mathbb{R}^d),
    \end{equation}
    where $\bm x^{\bm \alpha} := \prod_{i=1}^d x_i^{\alpha_i}$ and $|\bm \alpha| := \sum_{i=1}^{d}\alpha_i$. 
\end{definition}

\begin{remark}
        The requirement that $V_{\mathrm{lc}}$ be compactly supported is sufficient to ensure the convergence of the periodic sum in \eqref{Eq-defPSH-01}. Long-range potentials, such as the logarithmic potentials $\log |\bm x|$ and Coulomb potentials $1/|\bm x|$, also fall within this framework in the sense that their short-range components in the Ewald decomposition \cite{toukmaji1996ewald} satisfy this compact-support condition.
\end{remark}

\begin{remark}
    In Fourier representation, the multiplication by $\bm x^{\bm \alpha}$ translates into the differential operator $(2\pi i)^{-|\bm \alpha|} \partial_{\bm k}^{\bm \alpha}$. Equivalently, \eqref{Eq-defPSH-02} implies that the Fourier coefficients of  $V_{\mathrm{lc}}$ exhibit accelerated decay after applying the partial derivative in $\bm k$-space. Consequently, the space  $PSH^{s,r}$ includes functions in $H^{s}$ whose Fourier coefficients behave like $|\bm k|^{-\gamma}$ for some $\gamma \in \mathbb{R}$. Typical singular potentials, such as the Coulomb potential, the Dirac‑delta potential, and the logarithmic potential, all belong to this class.
\end{remark}

In general, we allow the potential to possess isolated point singularities, which are specified in the following assumption.
\begin{assumption}\label{Asp::ARFSM}
Let $s > \max\left\{ d/2 - 2, -1 \right\}$. The potential $V(x)$ admits a decomposition of the form
\begin{equation}
V(\bm{x}) = V_{0}(\bm{x}) + \sum_{p=1}^{n_p} V_{p}(\bm{x} - \bm{x}_p),
\end{equation}
where $V_{0}\in H^{s+2}$, $V_{p}(\cdot-\bm x_p)\in PSH^{s,2}$,
and $r_m=\min_{p\neq q}|\bm x_p-\bm x_q|>0$.
Accordingly, we decompose potential operator as $ \mathcal{V}=\mathcal{V}_{0}+\sum_{p=1}^{n_p}\mathcal{V}_p.$
\end{assumption}

\subsection{Asymptotic expansion} 
Assumption~\ref{Asp::ARFSM} enables an expansion of $\mathcal{V}\psi$ based on the Taylor expansion of $\psi$ around each singular point $\bm x_p$. However, the eigenfunction $\psi$ may not possess sufficient regularity to admit such expansion. Therefore, we assume that the $r$-th order term of the expansion admits
\begin{equation}\label{Eq-2.2-01}
    \psi^{(r)} := \sum_{p=1}^{n_p}\sum_{|\bm \alpha|=r} \beta_{p,\bm \alpha}\Phi_{p,\bm \alpha},\quad r\in\{0,1\},
\end{equation}
where the functions $\Phi_{p,\bm \alpha}$ are periodic, low-regularity analogues of $(\bm x - \bm x_p)^{\bm \alpha}$ associated with $V_p$, and $\beta_{p,\bm \alpha}$ are the expansion coefficients. We require each $\Phi_{p,\bm \alpha}$ approximately solves single potential problem and satisfies
\begin{equation}
\label{Eq-2.2-02}
    \bigl(\mathcal{T}+\mathcal{V}_p\bigr)\Phi_{p,\bm\alpha} :=
    R_{p,\bm \alpha}\in H^{s+2},\quad
    \mathcal{D}^{\bm \alpha'}_{\bm x_p} [\Phi_{p,\bm \alpha}] = \delta_{\bm \alpha',\bm \alpha},\quad
    \forall\, |\bm \alpha'|\leq |\bm \alpha|\leq 1,
\end{equation}
where $\mathcal{D}^{\bm \alpha}_{\bm x_p}$ denote the linear functional that extracts the derivative of order $\bm \alpha$ at $\bm x_p$.

To further determine the coefficients $\beta_{p,\bm \alpha}$, we substitute the expansion \eqref{Eq-2.0-01} and \eqref{Eq-2.2-01} into the eigenvalue equation $(\mathcal{H}-\lambda)\psi = 0$,  which yields
\begin{equation}
\label{Eq-2.2-03}
    \mathcal{T}\psi^{\geq 2} + \sum_{p=1}^{n_p}\mathcal{V}_p\bigl(\psi - \sum_{|\bm \alpha|\leq 1}\beta_{p,\bm \alpha}\Phi_{p,\bm \alpha} \bigr) +R=0,
\end{equation}
where the total residual is given by
\begin{equation}\label{Eq-2.2-04}
R = \sum_{p=1}^{n_p}\sum_{|\bm \alpha|\leq 1}\beta_{p,\bm \alpha}R_{p,\bm \alpha}+\bigl(\mathcal{V}_{0}-\lambda\bigr)\psi \in H^{s+2}.     
\end{equation}
 The $H^{s}$ and $H^{s+1}$ components in left-hand-side of \eqref{Eq-2.2-03} vanish. In view of condition \eqref{Eq-defPSH-02}, each singular potential $V_p(\bm x)$, and hence the operator $\mathcal{V}_p$, acts on a function whose value and first-order derivatives vanish at $\bm x_p$, which implies
\begin{equation}\label{Eq-2.2-05}
    \mathcal{D}^{\bm \alpha'}_{\bm x_p} \Bigl(\psi - \sum_{|\bm \alpha|\leq 1}\beta_{p,\bm \alpha}\Phi_{p,\bm \alpha} \Bigr) =0,\quad  \forall\, |\bm \alpha'|\leq 1.
\end{equation}
By \eqref{Eq-2.2-02},  \eqref{Eq-2.2-05} is equivalent to
\begin{equation}\label{Eq-2.2-06}
    \beta_{p,\bm 0} = \mathcal{D}^{\bm 0}_{\bm x_p}[\psi],\qquad
    \beta_{p,\bm \alpha}=\mathcal{D}^{\bm \alpha}_{\bm x_p}\bigl[\psi-\beta_{p,\bm 0}\Phi_{p,\bm 0}\big],\quad \forall\, |\bm \alpha|=1.
\end{equation}
This implies that the coefficients are invariant  with respect to the expansion order. 

The proceeding argument is formalized in Lemma~\ref{Lem2.1}, whose proof is deferred to Section~\ref{subsec:AR3}.
\begin{lemma}
\label{Lem2.1}
    Under Assumption~\ref{Asp::ARFSM}, the eigenfunction $\psi$ of $\mathcal{H}$ admits an expansion of the form \eqref{Eq-2.0-01}. Specifically, there exists $\Phi_{p,\bm \alpha}$ satisfying \eqref{Eq-2.2-02}, and each expansion term  admits the representation \eqref{Eq-2.2-01}, where the coefficients $\beta_{p,\bm \alpha}$ are determined by \eqref{Eq-2.2-06}. Moreover,
     \begin{equation}
     \|\psi^{(0)}\|_{H^{s+2}}+ 
     \|\psi^{(1)}\|_{H^{s+3}}+ 
     \|\psi^{\geq 2}\|_{H^{s+4}} \lesssim \|\psi\|_{L^2}.
     \end{equation}
\end{lemma}

\begin{remark}
We refer to \eqref{Eq-2.0-01} as an asymptotic expansion in the sense that the Fourier coefficients  $\hat{\psi}_{\bm k}^{(r)}$  exhibit faster decay as $|\bm k|\to \infty$ for larger values of $r$.
\end{remark}
\begin{remark}
    One can generalize \eqref{Eq-2.0-01} to higher order term $\psi^{(r)}$ with $r\geq 2$, while determining the coefficients $\beta_{p,\bm \alpha}$ by matching the $H^{s+|\bm \alpha|}$ components in \eqref{Eq-2.2-03}.
    In this case, the residual $R$ contains a $H^{s+2}$ contribution from the term $\lambda\psi$, and thus the higher-order coefficients $\beta_{p,\bm \alpha}$ (for $|\bm \alpha| \geq 2 $) will depend on $\lambda$.
\end{remark}

\section{Asymptotic recovery}
\label{sec:postproc}
Let $\psi^N  \in \mathscr{X}_N$ be an approximation to $\mathcal{P}_N\psi$, for instance, obtained by FSM. The AR operator recovers the coefficients $\beta_{p,\bm \alpha}$ in $\psi^{(0)}$ and $\psi^{(1)}$ 
from $\psi^N$, and then adds these terms in high-frequency coefficients 
\begin{equation}
\label{Eq-3.0-01}
    \mathcal{A}_N\psi^N := \psi^N + \mathcal{P}^\perp_N\sum_{p=1}^{n_p}\sum_{|\bm \alpha|\leq 1}\beta_{p,\bm \alpha}\Phi_{p,\bm \alpha}.
\end{equation}
To determine the coefficients $\beta_{p,\bm \alpha}$, substituting $\psi=\mathcal{A}_N\psi^N$ into equation \eqref{Eq-2.2-05} yields,  for each $|\bm \alpha'|\leq 1,\ 1\leq p\leq n_p$, the linear system for 
\begin{equation}
\label{Eq-3.0-02}
     {D}^{\bm \alpha'}_{\bm x_p} \Bigl(\psi^N - \sum_{|\bm \alpha|\leq 1}\beta_{p,\bm \alpha}\mathcal{P}_N\Phi_{p,\bm \alpha} +\sum_{q\neq p}\sum_{|\bm \alpha|\leq 1}\beta_{q,\bm\alpha}\mathcal{P}_N^{\perp}\Phi_{q,\bm \alpha}\Bigr) =0.
\end{equation}
A detailed analysis presented in Section~\ref{subsec:AR4} shows that this linear system is uniformly well conditioned for all sufficiently large $N$; therefore, it uniquely determines $\beta_{p,\bm \alpha}$. Furthermore, constructing the coefficients of the linear system requires only $O(N^{d})$ operations. Indeed, evaluating the derivatives of $\psi^N$  and $\mathcal{P}_N\Phi_{p,\bm \alpha}$ at a given point requires only $O(N^d)$ operations using their Fourier coefficients, while the $\mathcal{P}_N^\perp\Phi_{p,\bm \alpha}$ term can be evaluated in the same way using the identity $\mathcal{P}_N^\perp\Phi_{p,\bm \alpha} = \Phi_{p,\bm \alpha}-\mathcal{P}_N\Phi_{p,\bm \alpha}$.

For the rest of this section, we establish the convergence order of AR and AR-FSM, and then derive explicit expressions for the asymptotic functions $\Phi_{\bm \alpha}$ for the 1D Dirac-delta potential and the 3D Coulomb potential.

\subsection{Convergence order of AR and AR-FSM}
By the triangle inequality, the error of AR post-processing can be decomposed as
\begin{equation}
\|\mathcal{A}_N\psi^N-\psi\|_{H^1} \leq \|\mathcal{A}_N(\psi^N-\mathcal{P}_N\psi)\|_{H^{1} } + \|\mathcal{A}_N\mathcal{P}_N\psi - \psi\|_{H^1}.  
\end{equation}
Bounding the first term requires a stability estimate for $\mathcal{A}_N$, while the second term quantifies its approximation accuracy. A detailed analysis of these properties is presented in Section~\ref{subsec:AR4}, leading to the following theorem.
\begin{theorem}[\bf Convergence order of AR]\label{Thm03}  
     Under  Assumption~\ref{Asp::ARFSM}, let $\Phi_{p,\bm \alpha}$ satisfy \eqref{Eq-2.2-02}, $\psi$ be an eigenfunction of $\mathcal{H}$ and  $\psi^N\in \mathscr{X}_N$. 
     Then the estimate holds for AR operator $\mathcal{A}_N$ given by \eqref{Eq-3.0-01} and \eqref{Eq-3.0-02}, for all sufficiently large $N$ 
    \begin{equation}\label{Eq-thm03-01}
    \|\mathcal{A}_N \psi^N-\psi\|_{H^{1}}\lesssim \|\psi^N-\mathcal{P}_N \psi \|_{H^1}+N^{-(s+3)}\|\psi\|_{L^2}.
    \end{equation}
\end{theorem}
By this lemma, the overall convergence order of eigenfunctions is governed by the minimum of  $s+3$ and the order of projected error $\|\psi^N - \mathcal{P}_N \psi\|_{H^1}$. According to Lemma~\ref{Lem1.1}, FSM produces an approximation $\psi^N \in \mathscr{X}_N$ such that the projected error converges at order $s+1+b$. Provided Lemma~\ref{Lem1.1} and Theorem~\ref{Thm03}, we can establish the convergence order of AR-FSM in Theorem~\ref{Thm02}.
\begin{proof}[Proof of Theorem~\ref{Thm02}]
    Combining Lemma~\ref{Lem1.1} and Theorem~\ref{Thm03} together,
    \begin{equation}\label{Eq-thm01-02}
    \begin{aligned}
        \|\mathcal{A}_N\psi_n^N-\psi\|_{H^1} 
        &\lesssim \|\psi_n^N-\mathcal{P}_N\psi\|_{H^1} + N^{-(s+3)}\|\psi\|_{L^2}\\
        &\lesssim N^{-(s+1+b)}\|\psi\|_{L^2},
    \end{aligned}
    \end{equation}
    where $\psi \in \operatorname{ker}(\lambda_n-\mathcal{H})$ is an exact eigenfunction associated with $\lambda_n$. This proves the estimate for eigenfunction since $\tilde{\psi}_n^N=\mathcal{A}_N\psi_n^N$. The estimate for eigenvalue is obtained using the Rayleigh quotient and Lemma~\ref{Lem5.3}:
     \begin{equation}
    \begin{aligned}
        |\tilde\lambda_n^N-\lambda_n| \|\tilde\psi_n^N\|^2_{L^2}
        &= 
        \left|\langle \tilde\psi_n^N,(\mathcal{H}-\lambda_n)\tilde\psi_n^N\rangle\right|\\
        &=
        \left| \langle \tilde\psi_n^N-\psi,(\mathcal{H}-\lambda_n)(\tilde\psi_n^N-\psi)\rangle  \right|\\
        &\lesssim 
        \|\tilde\psi_n^N-\psi\|_{H^1}^2 + (\lambda_n+M)\|\tilde\psi_n^N-\psi\|_{L^2}^2\\
        &\lesssim N^{-(2s+2+2b)}\|\psi\|_{L^2}^2\\
        &\lesssim N^{-(2s+2+2b)}\|\tilde\psi_n^N\|_{L^2}^2,
    \end{aligned}
    \end{equation}
    where the last two inequalities are obtained using \eqref{Eq-thm01-02}. Thus, dividing both sides $\|\tilde\psi_n^N\|_{L^2}^2$ yields the desired estimate for eigenvalue.
\end{proof}

\subsection{Explicit construction of asymptotic functions}
\label{subsec:3.2}

The general existence of functions $\Phi_{p,\bm \alpha}$ satisfying \eqref{Eq-2.2-02} is guaranteed by Lemma~\ref{Lem6.5}. For practical purposes, we derive explicit expressions for these functions for certain classes of potentials in this section. Since the construction is invariant under translations, we shift $\bm x_p$ to the origin and restrict our attention to the case $\bm x_p=\bm 0$. Our objective is to construct $\Phi_{\bm \alpha} $ for $ |\bm \alpha|\leq 1$ such that
 \begin{equation}\label{Eq-3.2-01}
\bigl(\mathcal{T}+\mathcal{V}\bigr)\Phi_{\bm \alpha}\in H^{s+2},\qquad          \mathcal{D}^{\bm \alpha'}_{\bm 0}[\Phi_{\bm \alpha}] = \delta_{\bm \alpha',\bm \alpha},\quad \forall\, |\bm\alpha'|\leq |\bm \alpha|,
 \end{equation}
for the 1D Dirac delta potential and the 3D Coulomb potential.
\subsubsection*{1D Dirac-delta potential} The potential $V(x) = \gamma\delta(x)$ where $d=1$ and $s=0.5^{\m}$. 
The asymptotic function for $\alpha=0$ is defined by \eqref{Eq-1.3-03}. 
A direct calculation shows that
\begin{equation}
    (\mathcal{T}+\mathcal{V})\Phi_{0}(x) =-\frac{\pi\gamma}{2}\sin(\pi |x|)\in H^{s+2},
\end{equation}
which verifies the required conditions \eqref{Eq-3.2-01} for the case $\alpha=0$.

For  $\alpha=1$, note that $\Phi_1(x)V(x)=0$  whenever $\Phi_1(0)=0$, we may simply choose a smooth function
\begin{equation}
    \Phi_{1} (x) = \frac{1}{2\pi}\sin (2\pi x).
\end{equation}
 In this case, we have $\psi^{(1)}\in C^{\infty}$ in asymptotic expansion \eqref{Eq-2.0-01}. Consequently, $\psi^{(1)}$ can be absorbed into $\psi^{\geq 2}$ and is therefore not necessary to include in the expansion.

\subsubsection*{3D Coulomb potential} The potential $V_{\mathrm{lc}}(\bm x) \sim \gamma|\bm x|^{-1}$ where $d=3$ and $s=0.5^{\m}$. In the periodic setting, we define $V=\sum_{\bm n}V_{\mathrm{lc}}(\bm \cdot+\bm n)$ via its Fourier coefficients as 
\begin{equation}
\hat V_{\bm k}=
\begin{cases}
0, & \bm k=\bm 0,\\
\dfrac{\gamma}{\pi}\dfrac{1}{|\bm k|^{2}}, & \bm k\neq \bm 0 ,
\end{cases}
\qquad \bm k \in \mathbb{Z}^3.
\end{equation}
Define the auxiliary function (with $B>0$ an arbitrary parameter)
\begin{equation}\label{Eq-3.2-05}
    \Psi(\bm x)
    := \sum_{\bm n\in \mathbb{Z}^3} e^{-2\pi B|\bm x + \bm n|}
    \quad \Longleftrightarrow \quad
    \hat{\Psi}_{\bm k}
    = \frac{B}{\pi^2\bigl(|\bm k|^2+B^2\bigr)^2}.
\end{equation}
We then define
\begin{equation}\label{Eq-3.2-06}
    \Phi_{\bm 0}(\bm x)
    := -\frac{\gamma}{4\pi B}\bigl(\Psi(\bm x)-\Psi(\bm 0)\bigr) + 1 .
\end{equation}
To construct $\Phi_{\bm \alpha}$ for $|\bm \alpha|=1$, we multiply the auxiliary
function in \eqref{Eq-3.2-05} by $\bm x^{\bm \alpha}$,
\begin{equation}
    \Psi^{\bm \alpha}(\bm x)
    := \sum_{\bm n\in \mathbb{Z}^3}
    (\bm x+\bm n)^{\bm \alpha} e^{-2\pi B|\bm x+\bm n|}
    \quad \Longleftrightarrow \quad
    \hat{\Psi}^{\bm \alpha}_{\bm k}
    = \frac{-2 i B\, \bm k^{\bm \alpha}}
    {\pi^3\bigl(|\bm k|^2+B^2\bigr)^3}.
\end{equation}
Let $s_{\bm \alpha} = {(2\pi)}^{-1}\sin (2\pi \bm x^{\bm \alpha})$, then we define
\begin{equation}\label{Eq-3.2-08}
    \Phi_{\bm \alpha}(\bm x)
    := -\frac{\gamma}{2\pi B}
    \Bigl(
        (\Psi^{\bm \alpha}(\bm x)
        - s_{\bm \alpha}(\bm x)\mathcal{D}^{\bm \alpha}_{\bm 0}[ \Psi^{\bm \alpha}] 
    \Bigr)
    +s_{\bm \alpha}(\bm x),
    \qquad |\bm \alpha|=1.
\end{equation}
The asymptotic functions are explicitly constructed in \eqref{Eq-3.2-06} and \eqref{Eq-3.2-08}. 

Since the super-convergence order of FSM satisfies $s+1+b\leq s+2$ for the 3D Coulomb potential, recovering the $\Phi_{\bm 0}$ term alone is sufficient to fully exploit the super-convergence property. Indeed, for $|\bm \alpha|=1$,
\[
\|\mathcal{P}_{N}^{\perp}\Phi_{\bm \alpha}\|_{H^1}
\lesssim
N^{-(s+2)}
\|\Phi_{\bm \alpha}\|_{H^{s+3}},
\]
so the contribution from these terms is already of higher order than the target convergence order $s+1+b$. Below we only verify that $\Phi_{\bm 0}$ defined in \eqref{Eq-3.2-06} satisfies \eqref{Eq-3.2-01}.
The function $\Phi_{\bm 0}$ satisfies the normalization condition
$\mathcal{D}^{\bm 0}_{\bm 0}[\Phi_{\bm 0}] = 1$, and a direct calculation yields
\begin{equation}
    \bigl(\mathcal{T}+\mathcal{V}\bigr)\Phi_{\bm 0}
    = V -\frac{\gamma}{4\pi B}\,\mathcal{T}\Psi
      - \frac{\gamma}{4\pi B}\, V\bigl(\Psi-\Psi(\bm 0)\bigr).
\end{equation}
The Fourier coefficients of
$V - \frac{\gamma}{4\pi B}\mathcal{T}\Psi$ are given by
\[
    \frac{\gamma}{\pi}\frac{1}{|\bm k|^2} - \frac{\gamma}{4\pi B}\frac{4\pi^2 |\bm k|^2B}{\pi^2(|\bm k|^2+B^2)^2} = \frac{\gamma}{\pi}\left(\frac{1}{|\bm k|^2}-\frac{|\bm k|^2}{(|\bm k|^2+B^2)^2}\right),
\]
which decay as $\mathcal{O}(|\bm k|^{-4})$. Therefore this term belongs to
$H^{s+2}$. For $ V(\Psi-\Psi(\bm 0))$ term, using the real-space representation in
\eqref{Eq-3.2-05}, 
\begin{equation}\label{Eq-3.2-09}
    V\bigl(\Psi-\Psi(\bm 0)\bigr)
    = \sum_{\bm n \in \mathbb{Z}^3}
    V_{\mathrm{lc}}(\bm x+\bm n)
    \Bigl(
        e^{-2\pi B|\bm x+\bm n|}-1
        + \phi(\bm x+\bm n)
        - \phi(\bm 0)
    \Bigr),
\end{equation}
where $\phi(\bm x)
    := \sum_{\bm n\in \mathbb{Z}^3\setminus\{\bm 0\}}
    e^{-2\pi B|\bm x+\bm n|}$
collects the contributions from the periodic images. The function $\phi$ is smooth in a neighborhood of $\bm x=\bm 0$ and satisfies $\nabla\phi(\bm 0)=\bm 0$. Consequently,
$
    V_{\mathrm{lc}}(\bm x)\bigl(\phi(\bm x)-\phi(\bm 0)\bigr)
    \in H^{s+2}(\mathbb{R}^3).
$ On the other hand, the function $|\bm x|^{-1}\bigl(e^{-2\pi B|\bm x|}-1\bigr)$ also belongs to $H^{s+2}(\mathbb{R}^3)$ for $s=0.5^{\m}$.
It therefore follows from \eqref{Eq-3.2-09} that $ \mathcal V\bigl(\Psi-\Psi(\bm 0)\bigr)\in H^{s+2} $. 

\begin{remark}
    In both the Dirac-delta and Coulomb cases, the high-frequency component of $\Phi_{p, \bm \alpha}$ can be written as $\gamma$ times a fixed function, and $\gamma$ denotes the strength of the potential $V$. It is not clear at present whether this property persists for general potential with point singularities.
\end{remark}

\section{Numerical experiments}
\label{sec:experiments}
In this section, we numerically validate the theoretical estimates and evaluate the performance of FSM and AR-FSM. We first benchmark FSM, and confirm the sharpness of our convergence bounds stated in Theorem~\ref{Thm01}. Then we apply AR-FSM to two representative singular potentials—the 1D Dirac‑delta potential and the 3D Coulomb potential—to demonstrate the improved convergence order.
The resulting algebraic eigenvalue problems are solved using the Locally Optimal Block Preconditioned Conjugate Gradient (LOBPCG) method \cite{knyazev2001LOBPCG}.

\subsection{FSM benchmark}
To present the numerical results, we first introduce several error indicators.
For the $n$-th numerical eigenpair $(\lambda_n^N, \psi_n^N) \in \mathbb{R} \times \mathscr{X}_N$, with $\lambda_n$ denoting the corresponding reference eigenvalue, we define the eigenvalue error as
\begin{equation}
\mathrm{EigErr}_n := |\lambda_n^N - \lambda_n^{\mathrm{ref}}|.
\end{equation}
The $H^1$ eigenfunction errors are quantified in two ways. First, $\mathrm{pH1Err}_n$ measures the relative $H^1$ error with respect to the projection of the reference eigenfunction onto $\mathscr{X}_N$:
\begin{equation}\label{eq:pH1Err}
\mathrm{pH1Err}_n := \min_{\psi^{\mathrm{ref}}_n} 
\frac{\|\psi_n^N - \mathcal{P}_N \psi^{\mathrm{ref}}_n\|_{H^1}}{\|\psi_n^N\|_{H^1}},
\end{equation}
while $\mathrm{H1Err}_n$ measures the relative $H^1$ error with respect to the reference eigenfunction
\begin{equation}
\mathrm{H1Err}_n := \min_{\psi_n^{\mathrm{ref}} } 
\frac{\|\psi_n^N - \psi_n^{\mathrm{ref}}\|_{H^1}}{\|\psi_n^N\|_{H^1}},
\end{equation}
where the minimization is taken over the eigenspace associated with the reference eigenvalue $\lambda_n^{\mathrm{ref}}$.

To evaluate the overall accuracy in the numerical experiments, we define the mean errors over the first $m=8$ eigenstates:
\[
\mathrm{EigErr} := \frac{1}{m}\sum_{n=1}^{m} \mathrm{EigErr}_n, \quad
\mathrm{pH1Err} := \frac{1}{m}\sum_{n=1}^{m} \mathrm{pH1Err}_n, \quad
\mathrm{H1Err} := \frac{1}{m}\sum_{n=1}^{m} \mathrm{H1Err}_n.
\]

\subsubsection*{1D model problem} \label{subsubsec:4.1.1}

\begin{figure}[t]
    \centering
    \captionof{table}{Theoretical convergence orders of $\mathrm{EigErr}$, $\mathrm{pH1Err}$, $\mathrm{H1Err}$ in 1D test case,  as predicted by Lemma~\ref{Lem1.1} and Theorem~\ref{Thm01} with different values of $t$ in \eqref{Eq-4.1-01}. The notation $a^{\m}$ denotes $a-\varepsilon$ for arbitrarily small $\varepsilon>0$.}
    \label{tab:02}
    \begin{tabular}{c c c  c c c}
    \toprule
    \multicolumn{3}{c}{parameters} & \multicolumn{3}{c}{Theoretical orders} \\
    \cmidrule(lr){1-3}
    \cmidrule(lr){4-6}
    $t$ & $s$ & $b$ 
    & EigErr & pH1Err & H1Err \\
    \midrule
    $0$        & $-0.5^\m$ & $0.5^\m$ & $1.0^\m$   & $1.0^\m$   & $0.5^\m$ \\
    $0.5$ & $0^\m$         & $1.0^\m$        & $2.0^\m$   & $2.0^\m$   & $1.0^\m$ \\
    $1.0$        & $0.5^\m$  & $1.5^\m$ & $3.0^\m$   & $3.0^\m$ & $1.5^\m$ \\
    \bottomrule
    \end{tabular}
    
    \includegraphics[width=0.95\linewidth]{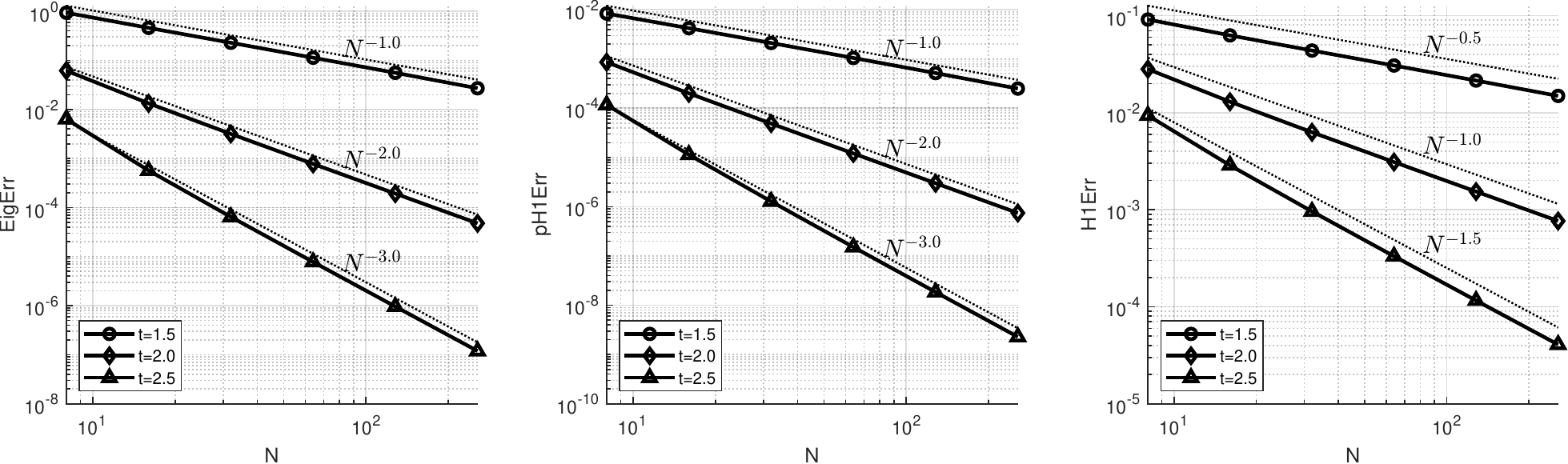}
    \captionof{figure}{FSM benchmark: Convergence of $\mathrm{EigErr}$ (left), $\mathrm{pH1Err}$ (middle), $\mathrm{H1Err}$ (right) against $N$ in 1D test case with different values of  in \eqref{Eq-4.1-01}. The black curves with markers represent the numerical errors,  where the markers correspond to $N=8,16,32,64,128, 256$. The dotted lines indicate the theoretical convergence slopes. The observed convergence orders are in excellent agreement with the theoretical predictions. } 
    \label{fig:02}

\end{figure}

Consider the 1D test problem with potential $V$ defined by the Fourier coefficients
\begin{equation}\label{Eq-4.1-01}
\hat{V}_k = \begin{cases}
    0 & k=0,\\
    -10|k|^{-t}&k\neq 0.
\end{cases}
\end{equation}
By the definition of $H^s$ norm \eqref{Eq-1.1-03}, we have $V \in H^{s}$ for $s=(t-0.5)^{\m}$. 

We investigate the convergence order of FSM for three representative values of $t=0, 0.5,1.0$. Table~\ref{tab:02} summarizes the theoretical convergence order predicted by Lemma~\ref{Lem1.1} and Theorem~\ref{Thm01}, while Figure~\ref{fig:02} illustrates the convergence of the FSM solutions with respect to the truncation parameter $N$ for these values of $t$. 

\subsubsection*{3D model problem}\label{subsubsec:4.1.2}

\begin{figure}[t]
\centering
    \captionof{table}{Theoretical convergence orders of $\mathrm{EigErr}$, $\mathrm{pH1Err}$ in 3D test case, as predicted by Lemma~\ref{Lem1.1} and Theorem~\ref{Thm01} with different values of $t$ in \eqref{Eq-4.1-02}. The notation $a^{\m}$ denotes $a-\varepsilon$ for arbitrarily small $\varepsilon>0$.}
    \label{tab:03}
    \begin{tabular}{c c c  c c c}
    \toprule
    \multicolumn{3}{c}{parameters} & \multicolumn{3}{c}{Theoretical orders} \\
    \cmidrule(lr){1-3}
    \cmidrule(lr){4-6}
    $t$ & $s$ & $b$ 
    & EigErr & pH1Err & H1Err \\
    \midrule
    $1.5$        & $0^\m$    & $0.5^\m$ & $2.0^\m$   & $1.5^\m$   & $1.0^\m$ \\
    $2.0$        & $0.5^\m$  & $1.0^\m$ & $3.0^\m$   & $2.5^\m$   & $1.5^\m$ \\
    $2.5$        & $1.0^\m$  & $1.5^\m$ & $4.0^\m$   & $3.5^\m$   & $2.0^\m$ \\
    \bottomrule
    \end{tabular}

    \includegraphics[width=0.95\linewidth]{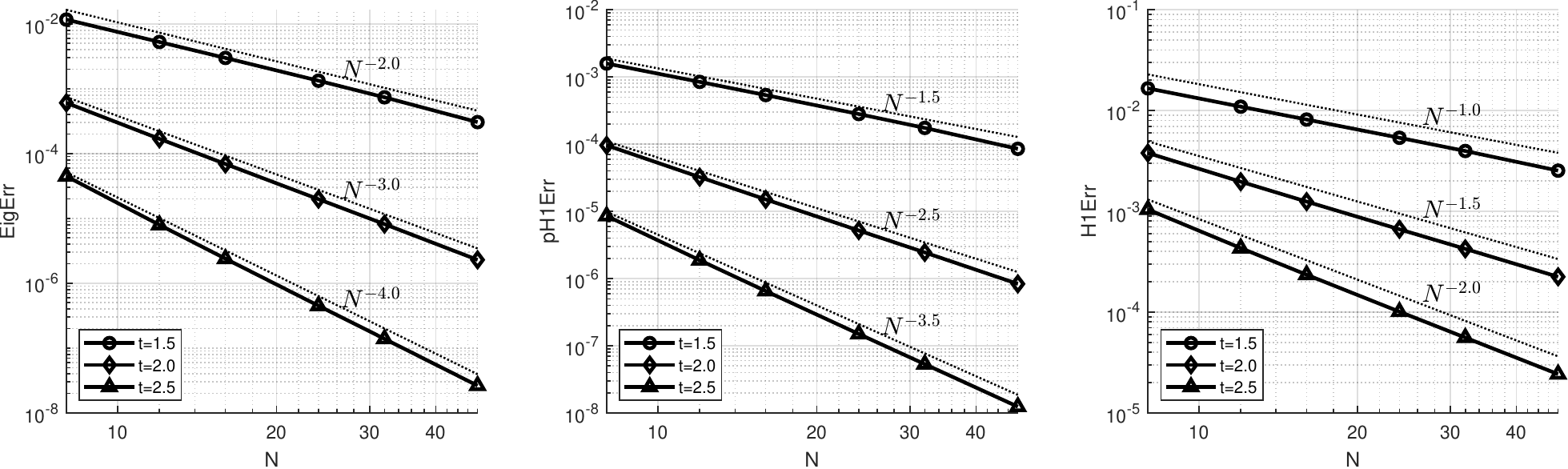}
    \captionof{figure}{FSM benchmark: Convergence of $\mathrm{EigErr}$ (left), $\mathrm{pH1Err}$ (middle), $\mathrm{H1Err}$ (right) against $N$ in 3D test case with different values of in \eqref{Eq-4.1-02}. The black curves with markers represent the numerical errors,  where the markers correspond to $N=8,12,16,24,32,48$. The dotted lines indicate the reference slopes corresponding to theoretical orders. The observed convergence orders are in excellent agreement with the theoretical predictions. }
    \label{fig:03}
\end{figure}

We next consider the 3D test problem to examine the convergence behavior of FSM in higher dimensions.
The potential $V$ is defined through its Fourier coefficients
\begin{equation}\label{Eq-4.1-02}
\hat{V}_{\bm k} =
\begin{cases}
0, & \bm k=\bm 0,\\
-|\bm k|^{-t}, & \bm k\neq \bm 0.
\end{cases}
\end{equation}

By the definition of $H^s$ norm \eqref{Eq-1.1-03}, we have $V \in H^{s}$ for $s=(t-1.5)^{\m}$.

We investigate the convergence order of FSM for three representative values of $t=1.5, 2.0,2.5$. Table~\ref{tab:03} summarizes the theoretical convergence order predicted by Lemma~\ref{Lem1.1} and Theorem~\ref{Thm01}, while Figure~\ref{fig:03} illustrates the convergence of the FSM solutions with respect to the truncation parameter $N$ for these values of $t$. 

 The observed convergence orders in both test cases are in excellent agreement with the theoretical predictions, thereby confirming the sharpness of the estimate of FSM in \ref{Lem1.1} and Theorem~\ref{Thm01}.

\subsection{Performance of AR-FSM}

Next, we compare the performance of FSM and AR-FSM  with point-singular potential.
We first consider a 1D test problem with a superposition of Dirac-delta singularities, and then a three-dimensional test problem with a superposition of Coulomb singularities. The corresponding asymptotic basis functions are provided in Section~\ref{subsec:3.2}.

To observe the convergence order of the AR post-processed eigenvalues in Theorem~\ref{Thm02}, it is necessary to evaluate the numerical integrals in the Rayleigh quotient with sufficient accuracy. We employ a re-expansion technique to eliminate terms involving more than two distinct singular points, while terms associated with a single singular point are computed analytically. The details of the implementation are deferred to  the forthcoming work \cite{lipreparationEFSM}.

\begin{figure}[t]
\centering
\begin{minipage}{0.48\linewidth}
\centering
\includegraphics[width=\linewidth]{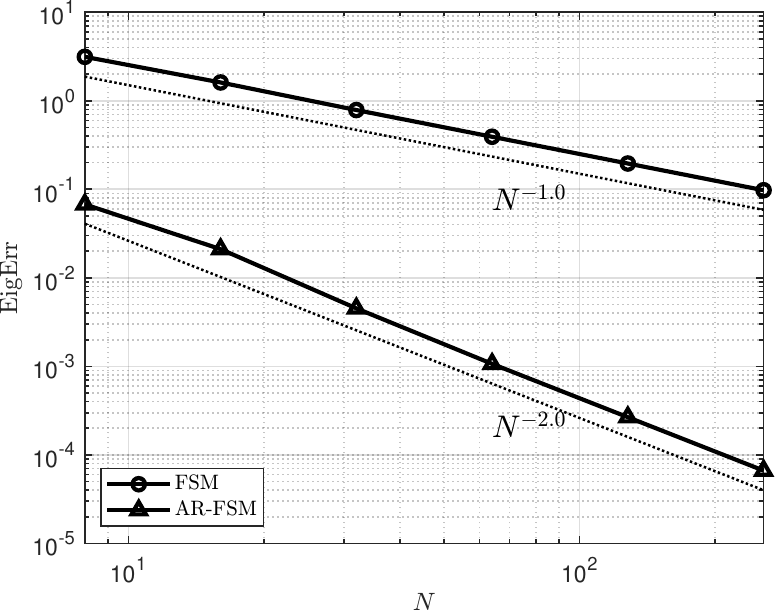}
\end{minipage}
\hfill
\begin{minipage}{0.48\linewidth}
\centering
\includegraphics[width=\linewidth]{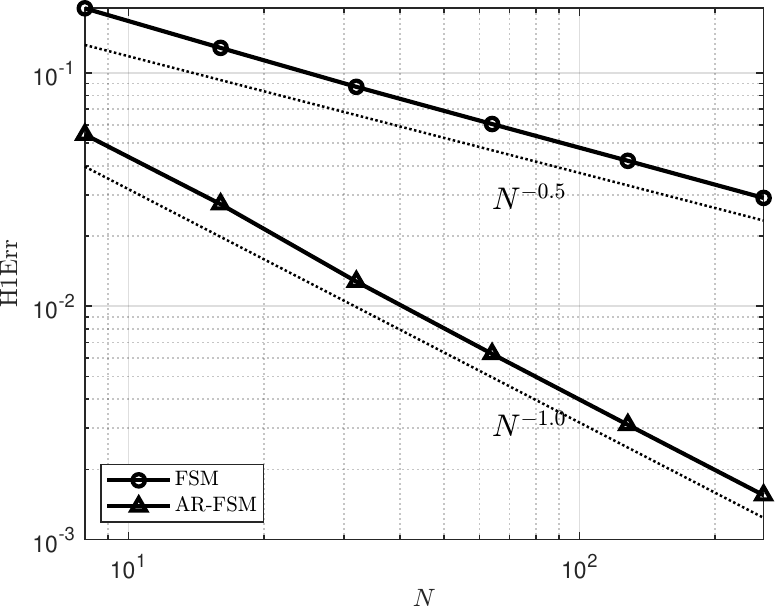}
\end{minipage}

\begin{minipage}{0.48\linewidth}
\centering
\includegraphics[width=\linewidth]{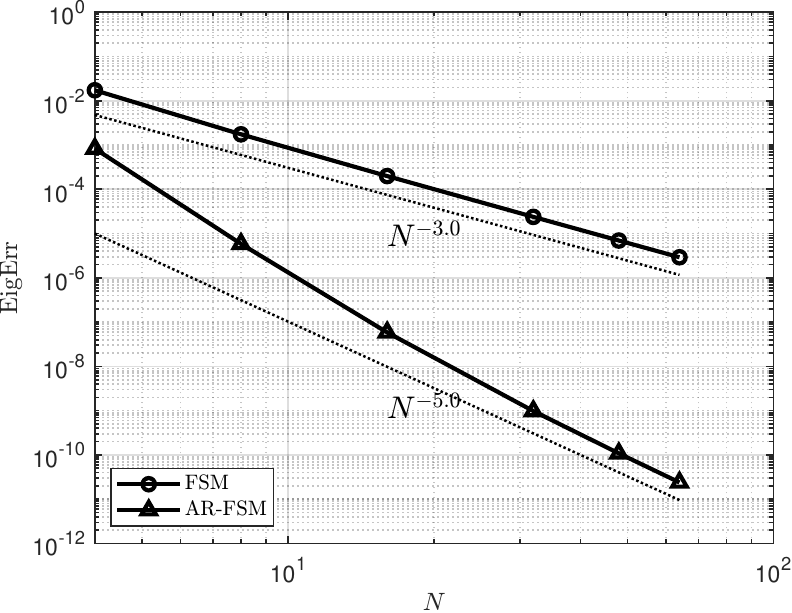}
\end{minipage}
\hfill
\begin{minipage}{0.48\linewidth}
\centering
\includegraphics[width=\linewidth]{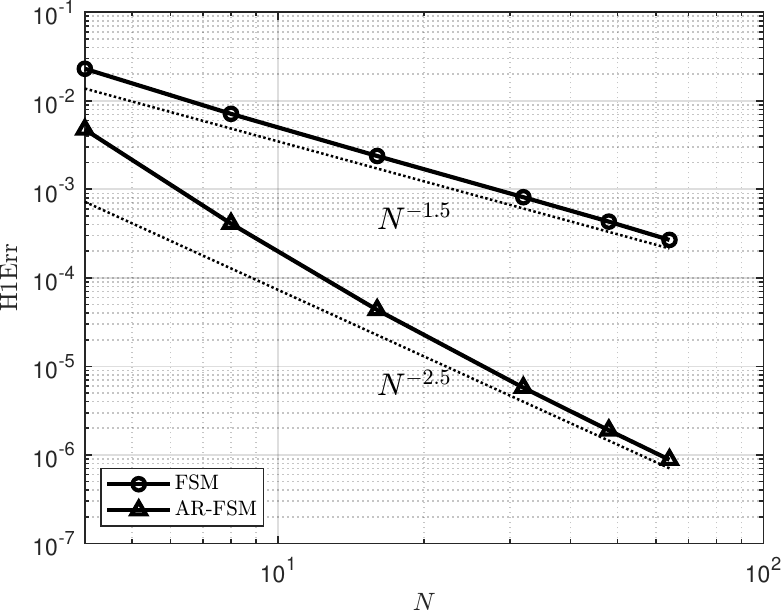}
\end{minipage}
\caption{ FSM v.s. AR-FSM: Convergence of EigErr (left) and H1Err (right) against $N$ for the 1D Dirac-delta potentials (top) and the 3D Coulomb potentials (bottom).
The bold curves with markers represent the numerical error, where the markers correspond to $N=8,16,32,64,128,256$ in the 1D case, and $N=4,8,16,32,48,64$ in the 3D case.
The results confirm the theoretical predictions and demonstrate the superior performance of AR-FSM over FSM.}
\label{fig:04}
\end{figure}

\subsubsection*{1D Dirac-delta singularities}
Consider the potential
\[
V(x) = \sin(2\pi x) + \sum_{p=1}^{n_p} \gamma_p \, \delta(x - x_p),
\]
with $n_p=6$. The positions $\hat{x}:=(x_p)$ and the strengths $ \hat{\gamma}:=(\gamma_p)$ are
\[
\hat{x} = (0.08,\ 0.21,\ 0.39,\ 0.57,\ 0.74,\ 0.91),
\qquad
\hat{\gamma} = (-12,7,\ -9,\  6,-7,\  9).
\]
The potential $V$ satisfies Assumption~\ref{Asp::ARFSM} with index $s=-0.5^{\m}$. 

\subsubsection*{3D Coulomb singularities}
Consider the potential
\[
V(\bm{x}) = \sum_{p=1}^{n_p} \gamma_p \, V_{\mathrm{Coulomb}}(\bm{x}-\bm{x}_p),
\]
where $ V_{\mathrm{Coulomb}} $ is defined via its Fourier coefficients
\[
\hat V_{\mathrm{Coulomb,}\bm k} =
\begin{cases}
\dfrac{\pi}{2}, & \bm k=\bm 0,\\[1mm]
\dfrac{1}{\pi |\bm k|^{2}}, & \bm k\neq \bm 0.
\end{cases}
\]
We set $ n_p=4 $, with 
\[
\begin{aligned}
\bm{x}_1 &= (0.00,\,0.00,\,0.00), &
\bm{x}_2 &= (0.70,\,0.40,\,0.40),\\
\bm{x}_3 &= (0.30,\,0.50,\,0.70), &
\bm{x}_4 &= (0.40,\,0.70,\,0.40),
\end{aligned}
\]
\[
\bm{\gamma}_1=\gamma_2=\gamma_3=\gamma_4=-\pi.
\]
The potential $V$ satisfies Assumption~\ref{Asp::ARFSM} with index $s=0.5^{\m}$.

According to Theorems \ref{Thm01} and \ref{Thm02}, in the 1D Dirac-delta case ($s=-0.5^{\m}$, $b=0.5^{\m}$), FSM achieves convergence orders of $1^{\m}$ for $\mathrm{EigErr}$ and $0.5^{\m}$ for $\mathrm{H1Err}$, while the AR-FSM improves these to $2^{\m}$ and $1^{\m}$, respectively. 
In the 3D Coulomb case ($s=0.5^{\m}$, $b=1^{\m}$),  FSM achieves convergence orders of $3^{\m}$ for $\mathrm{EigErr}$ and $1.5^{\m}$ for $\mathrm{H1Err}$, while the AR-FSM improves these to $5^{\m}$ and $2.5^{\m}$, respectively. 
The numerical results in Figure~\ref{fig:04} confirm the sharpness of these theoretical predictions.

\section{Analysis of FSM}\label{sec:proofFM}

This section is devoted to proving Lemma~\ref{Lem1.1} and Theorem~\ref{Thm01}.  
The proof proceeds in three main steps. First, in Section~\ref{subsec:proofFM-4}, we establish preliminary results, including the regularity properties of eigenfunctions and related estimates.  
Next, in Section~\ref{subsec:proofFM-2}, we connect the FSM solution to the exact eigenpair through the Feshbach-Schur map.
Finally, in Section~\ref{subsec:proofFM-3}, we derive a perturbation bound that controls the error in the subspace spanned by the first $n$ eigenfunctions.  Combining these steps leads to the main convergence results of FSM.

Throughout Sections \ref{sec:proofFM} and \ref{sec:proofAR}, the parameters $s, b$ and $l$ are fixed and retain the same meaning. In particular, $s>\max\{d/2-2,-1\}$ denotes the regularity index of the potential $V$, and we set $b = \min\{(s+2-d/2)^{\m}, s+1,2\}>0$ and $l=2-b=\max\{-(s-d/2)^{\m},1-s,0\}$. For any $-1 \leq t \leq s$, the following index conditions are satisfied
\begin{equation}\label{Eq-5.0-01}
    t \leq \min\{t+l, s\}, \qquad
    t + d/2 < t+l + s, \qquad
    t+l + s \geq 0.
\end{equation}
\subsection{Preliminary results}
\label{subsec:proofFM-4}
We first provide the boundedness property of potential $V$ satisfying  Assumption~\ref{Asp::FSM}.
\begin{lemma}\label{Lem5.1}(Theorem A.1 in \cite{behzadan2021multiplication})
    Let $(\alpha,\beta,\gamma) \in \mathbb{R}^3$ belong to the index set
    \[ 
    I:=\bigl\{(x,y,z)\in \mathbb{R}^3|\  z \leq \min\{x,y
    \},\ z+d/2< x+y,\ x+y \geq 0\bigr\}.
    \]
    Then the embedding $H^\alpha(\cdot) \times H^\beta(\cdot) \hookrightarrow H^\gamma(\cdot)$ holds, where $H^{s}(\cdot)$ denotes $H^{s}$ or $H^{s}(\mathbb{R}^d)$. That is,  for all $u \in H^\alpha(\cdot)$, $v \in H^\beta(\cdot)$,
    \begin{equation}
        \|uv\|_{H^\gamma(\cdot)} \lesssim \| u \|_{H^\alpha(\cdot)} \|v\|_{H^\beta(\cdot)}.
    \end{equation}
\end{lemma}
\begin{proof}
The only case not covered by Theorem A.1 in \cite{behzadan2021multiplication} occurs when $\gamma<0$ and $\min\{\alpha,\beta\}\geq 0$, where we require $\alpha+\beta\geq0$ rather than $\alpha+\beta > 0$. In this remaining case $\alpha=\beta=0$, we have 
\[
H^{0}(\cdot)\times H^{0}(\cdot) \hookrightarrow L^{1}(\cdot)\hookrightarrow\bigl(L^{\infty}(\cdot)\bigr)^{\star}\hookrightarrow \bigr(H^{-\gamma}(\cdot)\bigl)^{\star}=H^{\gamma}(\cdot),
\]
where the last embedding follows from  $H(\cdot)^{-\gamma}\hookrightarrow L^{\infty}(\cdot)$ using $-\gamma>d/2$.

\end{proof}

Combining this lemma with \eqref{Eq-5.0-01}, we obtain the following corollary.
\begin{lemma}\label{Lem5.2}
Let $V \in H^s$. Then for any $-1 \leq t \leq s$ and $\phi \in H^{t+l}$,
\begin{equation}\label{eq:Lem5.2-01}
\|\mathcal{V} \phi\|_{H^{t}} \lesssim \|\phi\|_{H^{t+l}}.
\end{equation}
\end{lemma}

Next, we establish boundedness and regularity results for the operator
$\mathcal{H}=\mathcal{T}+\mathcal{V}$. We note that Lemmas~\ref{Lem5.3} and \ref{Lem5.4} hold under the sole assumption that $\mathcal{V}$ satisfies \eqref{eq:Lem5.2-01} (without requiring $\mathcal{V}$ to be a multiplication operator).

\begin{lemma}\label{Lem5.3}
 Let $\mathcal{H}=\mathcal{T}+\mathcal{V}$ with $\mathcal{V}$ satisfying \eqref{eq:Lem5.2-01}. Then there exists a constant $M>0$ such that
\begin{equation}\label{Eq-Lem5.3-01}
\tfrac12\|\phi\|_{H^1}^2 \leq \langle \phi, (\mathcal{H}+M)\phi\rangle \lesssim \|\phi\|_{H^1}^2,\qquad \forall\, \phi \in H^1.
\end{equation}
\end{lemma}

\begin{proof}
By Lemma~\ref{Lem5.2}, there exists a constant $C_V$ such that
\begin{equation}
\langle \phi, \mathcal{V}\phi \rangle 
\leq \|\phi\|_{H^{l/2}} \|\mathcal{V}\phi\|_{H^{-l/2}} 
\leq C_V \|\phi\|_{H^{l/2}}^2.
\end{equation}
By the Sobolev interpolation inequality and Young's inequality, for any $\delta>0$,
\[
\|\phi\|_{H^{l/2}}^2 
\leq \delta \|\phi\|_{H^1}^2 + C_\delta \|\phi\|_{L^2}^2,
\]
where $C_\delta = \delta^{-l/(2-l)}$.
Choosing $\delta = (2C_V)^{-1}$, we deduce
\[
 \tfrac{1}{2}\|\phi\|_{H^1}^2 - (1 + C_V C_\delta)\|\phi\|_{L^2}^2
 \leq \langle \phi, \mathcal{H}\phi \rangle 
 \leq\tfrac{3}{2}\|\phi\|_{H^1}^2 + (-1 + C_V C_\delta)\|\phi\|_{L^2}^2.
\]
This implies \eqref{Eq-Lem5.3-01} by taking $M = 1 + C_V C_\delta$. 
\end{proof}     

\begin{lemma}\label{Lem5.4}
Let $\mathcal{H}=\mathcal{T}+\mathcal{V}$ with $\mathcal{V}$ satisfying \eqref{eq:Lem5.2-01}, and let $(\lambda,\psi)$ be an eigenpair of $\mathcal{H}$ with $\lambda\lesssim 1$. Then 
    \begin{equation}\label{Eq-Lem5.4-01}
    \|\psi\|_{H^{s+2}} \lesssim \|\psi\|_{L^2}.
    \end{equation}
\end{lemma}

\begin{proof}
By Lemma~\ref{Lem5.3}, there exists $M>0$ such that 
\[
\|\psi\|_{H^1}^2 \leq2 \langle \psi, (\mathcal{H}+M)\psi \rangle 
= 2(\lambda+M)\|\psi\|_{L^2}^2,
\]
which implies $\|\psi\|_{H^1}\lesssim \|\psi\|_{L^2}$. To obtain higher regularity, we employ a bootstrap argument. For $-1 \le t \le s$, we rewrite the eigenvalue equation as
$
(\mathcal{T}+1)\psi = (\lambda+1 - \mathcal{V})\psi.
$
Taking the $H^t$ norm gives
\begin{equation}
\|\psi\|_{H^{t+2}} \lesssim \|\psi\|_{H^t} + \|\mathcal{V}\psi\|_{H^t}
\lesssim \|\psi\|_{H^t} + \|\psi\|_{H^{t+l}}.
\end{equation}
Since $l<2$, the estimate implies the bootstrap implication
$
\psi \in H^{t+l} \;\Longrightarrow\; \psi \in H^{t+2}.
$
Starting from $\psi \in H^1\hookrightarrow H^{-1+l}$, a finite iteration of this argument yields \eqref{Eq-Lem5.4-01}.
\end{proof}
\begin{remark}
    Lemma~\ref{Lem5.4} also holds for discrete eigenpair $(\lambda^N,\psi^N)$ of $\mathcal{H}$ in $\mathscr{X}_N$ with $\lambda^N\lesssim 1$. 
    The only modification in the proof is to take the $H^t$ norm on the projected equation $(\mathcal{T}+1)\psi^N=\mathcal{P}_{N}(\lambda^N+1-\mathcal{V})\psi^N.$
\end{remark}

Finally, the following lemma summarizes the truncation error and inverse inequality with respect to $N$ in the Sobolev norm. The proof follows from the Cauchy-Schwarz inequality and is deferred to Appendix~\ref{proofLem5.5}.
\begin{lemma}\label{Lem5.5}
    Let  $\psi\in H^{m}$ for some $m\in \mathbb{R}$. Then the following estimates hold
    \begin{enumerate}

        \item [$(\mathrm{i})$] For any $t<m$
        \begin{equation}\label{eq-Lem5.5-01}
            \|\mathcal{P}_N^\perp \phi\|_{H^t} \leq (2\pi N)^{t-m}\|\phi\|_{H^{m}};
        \end{equation}
        
        \item [$(\mathrm{ii})$] For any $t\in \mathbb{N}$ with $t+d/2 < m$,
        \begin{equation}\label{eq-Lem5.5-02}
            \|\mathcal{P}_N^\perp \phi\|_{W^{t,\infty}} \lesssim N^{t-m+d/2}\|\phi\|_{H^{m}};
        \end{equation}

        \item [$(\mathrm{iii})$] For any $t\in \mathbb{N}$ ,
        \begin{equation}\label{eq-Lem5.5-03}
            \|\mathcal{P}_N \phi\|_{W^{t,\infty}} \lesssim N^{\max\{-0^{\m},t-m+d/2\}}\|\phi\|_{H^{m}}.
        \end{equation}
        Here, the $W^{t,\infty}$ norm is define as 
        \[
            \|\phi\|_{W^{t,\infty}} :=\max_{|\bm \alpha|\leq t}\| \partial^{\bm \alpha}\phi\|_{L^{\infty}}.
        \]
    \end{enumerate}
\end{lemma}

\subsection{The Feshbach-Schur map}
\label{subsec:proofFM-2}
This section introduces the Feshbach-Schur map, which maps the original problem \eqref{Eq-1.1-02} in $H^1$ to discrete problem in $\mathscr{X}_N$. 
Decompose $\psi \in H^1$ as $\psi = \psi^N + \psi^\perp$, where $\psi^N =\mathcal{P}_N\psi$ and $\psi^\perp =\mathcal{P}^{\perp}_N\psi$. Note that the operator $\mathcal{T}$ commutes with $\mathcal{P}_N$ and $\mathcal{P}_N^\perp$, the eigenvalue equation $\mathcal{H}\psi=\lambda\psi$ can be rewritten as: 
\begin{subequations}
\begin{align}
\label{Eq-5.2-01a}
    \langle \phi^N, (\mathcal{H}-\lambda)\psi^N\rangle  
&= -\langle \phi^N, \mathcal{V}\psi^\perp\rangle,
&&\forall\phi^N\in \mathscr{X}_N
\\
\label{Eq-5.2-01b}
\langle \phi^\perp, (\mathcal{H}-\lambda)\psi^\perp\rangle 
&= -\langle \phi^\perp, \mathcal{V}\psi^N\rangle , 
&&\forall\,\phi^\perp\in \mathscr{X}^\perp_N.
\end{align}
\end{subequations}
Here $\mathscr{X}_N^\perp$ denotes the orthonormal complement of $\mathscr{X}_N$ in $H^1$.
By Lemma~\ref{Lem5.6} (i) , \eqref{Eq-5.2-01b} is uni-solvent for $\psi^\perp$, obtaining $\psi^\perp =-\mathcal{G}_N(\lambda) \mathcal{V} \psi^N$, where 
\[\mathcal{G}_N(\lambda):= \mathcal{P}_N^\perp(\mathcal{H}-\lambda)^{-1}_{\mathscr{X}_N^\perp}\mathcal{P}_N^\perp\]
is the solution operator.
Substituting $\psi^\perp$ into \eqref{Eq-5.2-01a} yields a reduced nonlinear eigenvalue problem in finite-dimensional space $\mathscr{X}_N$,
\begin{equation} 
\label{Eq-5.2-02}
    \langle \phi^N, \left(\mathcal{H}- \mathcal{V}\mathcal{G}_N(\lambda)\mathcal{V}\right)\psi^N \rangle  = \lambda \langle \phi^N, \psi^N\rangle,\quad \forall\,\phi^N\in \mathscr{X}_N.
\end{equation}
Define the self-adjoint operator
\[
\mathcal{F}_N(\lambda^{\star} ) = \mathcal{V} \mathcal{G}_N(\lambda^{\star}) \mathcal{V}.
\] 
The mapping $\mathcal{H}\mapsto \mathcal{P}_N\bigl(\mathcal{H}-\mathcal{F}_N(\lambda)\bigr)\mathcal{P}_N$ is referred to as the Feshbach-Schur map. The above construction establishes a one-to-one correspondence between the eigenspaces,
\[
    \mathcal{H}\psi = \lambda\psi \quad 
    \Longleftrightarrow \quad
    \mathcal{P}_N\bigl(\mathcal{H}-\mathcal{F}_{N}(\lambda)\bigr)\mathcal{P}_N\psi^N = \lambda\psi^N.
\]
Now fix an value $\lambda^{\star} \in \mathbb{R}$ on the left-hand side of \eqref{Eq-5.2-02}, denote the algebraic eigenpairs of $\mathcal{P}_N\bigl[\mathcal{H}-\mathcal{F}_N(\lambda^\star)\bigr]\mathcal{P}_N$ in $\mathscr{X}_N$ by
$\big( \sigma_n^N(\lambda^{\star}), \psi^N_n(\lambda^{\star}) \bigr) \in \mathbb{R} \times \mathscr{X}_N$, which satisfy
\begin{equation}
\label{Eq-5.2-03}
       \langle \phi^N, (\mathcal{H}-\mathcal{F}_N(\lambda^{\star}))\psi^N_n(\lambda^{\star}) \rangle  = \sigma_n^N(\lambda^{\star}) \langle \phi^N, \psi^N_n(\lambda^{\star})\rangle,\quad \forall\,\phi^N\in \mathscr{X}_N.
\end{equation}

The following lemma establishes a connection between the eigenpairs of $\mathcal{H}$ and those of $\mathcal{H}-\mathcal{F}_N(\lambda^\star)$,
implying that \eqref{Eq-5.2-03} reproduces the exact $n$-th eigenpair when $\lambda^{\star} = \lambda_n$. In this sense, the operator $\mathcal{F}_N(\lambda_n)$ can be interpreted as a perturbation operator that connects the FSM solution with the exact solution.
\begin{lemma}\label{Lem5.6}
    Let 
    $\{(\lambda_n,\psi_n)\}_{n\geq 1}$
    and $\{(\sigma_n^N(\lambda^{\star}),\psi^N_n(\lambda^{\star})\}_{n\geq 1}$ denote the solutions of \eqref{Eq-1.1-02} and \eqref{Eq-5.2-03}, respectively, all listed in non-decreasing order. Then there exists $N_0$ such that for all $N\geq N_0$, the following hold:
    
    \begin{enumerate}
    \item[$\mathrm{(i)}$] For $\lambda^{\star}\leq \pi^2N^2$, \eqref{Eq-5.2-01b} is uni-solvent for $\psi^\perp =-\mathcal{G}_N(\lambda^\star) \mathcal{V} \psi^N\in \mathscr{X}_N^\perp$;
    
    \item[$\mathrm{(ii)}$] For $\lambda^{\star} \leq \pi^2 N^2$,  $\sigma_n^N(\lambda^{\star})$ is continuous and non-increasing in $\lambda^{\star}$;
    
    \item[$\mathrm{(iii)}$] If $\lambda^\star=\lambda_n \leq \pi^2N^2$, then $\sigma_n^N(\lambda^\star) = \lambda_n$ and there exists $\psi_n \in \ker(\mathcal{H}-\lambda_n)$ such that $\psi^N_n(\lambda^\star) = \mathcal{P}_N\psi_n$.
    \end{enumerate}
\end{lemma}

The proof is deferred to Appendix~\ref{proofLem5.6}.

Next, we show that $\mathcal{F}_N(\lambda^\star)$ also enjoys the  boundedness properties in \eqref{eq:Lem5.2-01}. Consequently, we can apply Lemmas~\ref{Lem5.3} and \ref{Lem5.4} to the operator $\mathcal{H} - \mathcal{F}_N(\lambda^\star)$.
\begin{lemma}\label{Lem5.7}
    Let $V\in H^{s}$. There exists $N_0$ such that for all $N\geq N_0$, the following estimates hold for $\lambda^\star < \pi ^2N^2$ and $-1\leq t\leq s$
    \begin{equation}
        \|\mathcal{G}_N(\lambda^\star)f\|_{H^{t+2}}\lesssim \|\mathcal{P}^\perp_Nf\|_{H^{t}},\qquad
        \|\mathcal{F}_N(\lambda^\star)\phi\|_{H^{t}}\lesssim \|\phi\|_{H^{t+l}}.
    \end{equation}
\end{lemma}
\begin{proof}
 Let $g^\perp = \mathcal{G}(\lambda^\star)f$ solves
 $ \mathcal{P}_N^\perp\bigl(\mathcal{H}-\lambda^\star\bigr)g^\perp=\mathcal{P}_N^\perp f$. Taking the $H^t$ norm on both sides, and using Lemmas~\ref{Lem5.2} and \ref{Lem5.5} (i) yields, $\exists\, C_V>0$,
\begin{equation}\label{eq-Lem5.7-02}
\begin{aligned}
    \|\mathcal{P}_N^\perp f\|_{H^{t}} 
    &= \|\mathcal{P}^\perp_N(\mathcal{H}-\lambda^\star)g^\perp\|_{H^{t}}\\
    &\geq \|(\mathcal{T}-\lambda^\star)g^\perp\|_{H^{t}}-\|\mathcal{V} g^\perp\|_{H^{t}}\\
    &\geq \|(\mathcal{T}-\lambda^\star)g^\perp\|_{H^{t}}-C_V\| g^\perp\|_{H^{t+l}}\\
     &\geq
    \left(1-(\pi^2N^2+1)(2\pi N)^{-2}-C_V(2\pi N)^{l-2}\right)\|g^\perp\|_{H^{t+2}}.\\
\end{aligned}
\end{equation}
For sufficiently large $N$ such that $C_V(2\pi N)^{l-2}<1/2$, \eqref{eq-Lem5.7-02} implies
\begin{equation}
    \|\mathcal{G}(\lambda^\star)f\|_{H^{t+2}}=\|g^\perp\|_{H^{t+2}}\leq 5\|\mathcal{P}_N^\perp f\|_{H^{t}}.
\end{equation}
Then it follows from Lemma~\ref{Lem5.2} that 
\begin{equation}
    \|\mathcal{F}_N(\lambda^\star)\phi\|_{H^{t}}\lesssim \|\mathcal{G}_N(\lambda^\star)\mathcal{V}\phi\|_{H^{t+l}}
    \leq \|\mathcal{G}_N(\lambda^\star)\mathcal{V}\phi\|_{H^{t+2}} \lesssim \|\mathcal{V}\phi\|_{H^{t}} \lesssim \|\phi\|_{H^{t+l}}.
\end{equation}
\end{proof}

\subsection{Perturbation bound of the first \texorpdfstring{$n$}{n} eigenpairs}
\label{subsec:proofFM-3}
The perturbation theory for eigenvalues of self-adjoint operators (e.g., \cite{kato2013perturbation}) provides uniform estimates that valid for all eigenvalues. To obtain sharper bounds for the first $n$ eigenvalues, we present Proposition~\ref{Prop04}, which controls the eigenvalue error using the subspace spanned by the first $n$ eigenvectors.

To state the perturbation result, we take the following notations.
Let $\mathcal{S}$ be a self‑adjoint operator with eigenpairs $\{(\lambda_n, u_n)\}_{n \geq 1}$.  The space spanned by the first $n$ eigenvectors of $\mathcal{S}$ is denoted by
    $
    W_n(\mathcal{S}) := \operatorname{span}\{ u_1,u_2,\dots ,u_n\}.
    $
The orthogonal projection onto the eigenspace $\ker\big(\mathcal{S}-\lambda_n\big)$ is written as $\mathcal{P}_{n,\mathcal{S}}$.
The $n$-th spectral gap of $\mathcal{S}$ is
    $
    \gamma_{n,\mathcal{S}} := \min_{\lambda_j \neq \lambda_n} |\lambda_j-\lambda_n|.
    $

\begin{proposition}
    \label{Prop04}
    Let $\mathcal{S}_1, \mathcal{S}_2$ be two self-adjoint operators in $\mathscr{X}_N$, with respective eigenpairs denoted by $\{(\lambda_{1,n},u_{1,n})\}_{n\geq 1},\ \{(\lambda_{2,n},u_{2,n})\}_{n\geq1}$ listed in non-decreasing order. 

    Then the following estimate for eigenvalues holds,
    \begin{equation}\label{Eq-prop04-01}
        |\lambda_{1,n}-\lambda_{2,n}|\leq \sup_{u\in W_n(\mathcal{S}_1)\cup W_n(\mathcal{S}_2)} \left|\frac{\langle u,(\mathcal{S}_2-\mathcal{S}_1)u\rangle}{\langle u,u\rangle}\right|.
    \end{equation}
   Let $M > 0$ be a constant such that $\mathcal{S}_1 + M$ is positive definite, and define the $\mathcal{S}_1$ norm as 
    \[
    \|\psi\|_{\mathcal{S}_1}^2 := \langle \psi, (\mathcal{S}_1 + M)\psi \rangle.
    \]
    If $ |\lambda_{2,n} - \lambda_{1,n}| \leq \frac{\gamma_{n,\mathcal{S}_1}}{3} $,
    then the following estimate for the eigenvectors holds:
    \begin{equation}\label{Eq-prop04-02}
        \| u_{2,n}-\mathcal{P}_{n,\mathcal{S}_1}u_{2,n}\| _{\mathcal{S}_1}\leq 6(1+\frac{\lambda_{1,n}+M}{\gamma_{n,\mathcal{S}_1}})\sup_{v\in \mathscr{X}_N\setminus\{0\}}\left|\frac{\langle v,(\mathcal{S}_2-\mathcal{S}_1)u_{2,n}\rangle}{\|v\|_{\mathcal{S}_1}}\right|.
    \end{equation}
\end{proposition}
The proof of this proposition follows from standard techniques in the perturbation theory of eigenvalues and is deferred to Appendix~\ref{proofProp04}.

\subsection{Proof of \texorpdfstring{Lemma~\ref{Lem1.1}}{lemma1.1} and \texorpdfstring{Theorem~\ref{Thm01}}{theorem1.2}}

\begin{proof}[Proof of Lemma~\ref{Lem1.1}]

Fix $n \in \mathbb{N}^+$. Using the same notation in Proposition~\ref{Prop04}. Let $\mathcal{S}_1 = \mathcal{P}_N \bigl[\mathcal{H} - \mathcal{F}_N(\lambda_n)\bigr] \mathcal{P}_N$ and $\mathcal{S}_2 = \mathcal{P}_N \mathcal{H} \mathcal{P}_N$, then
\begin{equation} \label{Eq-Lem1.1-01}
    \begin{aligned}
    \lambda_{1,n} &= \sigma_n^N(\lambda_n), \qquad &\lambda_{2,n} &= \lambda_n^N, \\
    u_{1,n} &= \psi_n^N(\lambda_n), \qquad &u_{2,n} &= \psi_n^N.
    \end{aligned}
\end{equation}
By Lemma~\ref{Lem5.6}, for $N \geq N_0$ and $\pi^2 N^2 \geq \lambda_n$, we have
\begin{equation}\label{Eq-Lem1.1-02}
    \sigma_n^N(\lambda_n) = \lambda_n, \qquad \psi_n^N(\lambda_n) \in \mathcal{P}_N \operatorname{ker}(\lambda_n - \mathcal{H}).
\end{equation}
By Lemma~\ref{Lem5.3}, there exists $M > 0$ such that 
\begin{equation}\label{Eq-Lem1.1-03}
    \frac{1}{2} \|\phi\|_{H^1}^2 \leq \langle \phi, (\mathcal{S}_i + M) \phi \rangle \lesssim \|\phi\|_{H^1}^2, \qquad \forall\, \phi \in \mathscr{X}_N, \quad i = 1, 2.
\end{equation}
This implies that $\mathcal{S}_1 + M $ is positive definite, and establishes the equivalence between the $\mathcal{S}_1$ norm and the $H^1$ norm. Substituting \eqref{Eq-Lem1.1-01} and \eqref{Eq-Lem1.1-02} into Proposition~\ref{Prop04}, we obtain
\begin{equation}\label{Eq-Lem1.1-04}
    |\lambda_n^N - \lambda_n| \leq 
    \sup_{u \in W_n(\mathcal{S}_1) \cup W_n(\mathcal{S}_2)}
    \left|
    \frac{\langle u, \mathcal{F}(\lambda_n) u \rangle}{\langle u, u \rangle}
    \right|,
\end{equation}
and if $ |\lambda_n - \lambda_n^N| \leq \frac{1}{3} \gamma_{n, \mathcal{S}_1} $, then
\begin{equation}\label{Eq-Lem1.1-05}
    \|\psi_n^N - \mathcal{P}_{n,\mathcal{S}_1}\psi_n^N\|_{\mathcal{S}_1}
    \leq 6\left(1 + \frac{\lambda_n + M}{\gamma_{n, \mathcal{S}_1}}\right)
    \sup_{v \in \mathscr{X}_N \setminus \{0\}}
    \left|
    \frac{\langle v, \mathcal{F}_N(\lambda_n) \psi_n^N \rangle}{\|v\|_{\mathcal{S}_1}}
    \right|.
\end{equation}
To adapt \eqref{Eq-Lem1.1-05}, we verify that $|\lambda_n-\lambda_n^N|\leq \tfrac{1}{3}\gamma_{n,\mathcal{S}_1}$ for sufficiently large $N$. 

We first show that $\gamma_{n, \mathcal{S}_1}$ has a uniform lower bound $\gamma_{n, \mathcal{S}_1} \geq \gamma_{n, \mathcal{H}}$. In fact, by Lemma~\ref{Lem5.6},
for $\lambda_j > \lambda_n$, we have
$
\sigma_j^N(\lambda_n) - \lambda_n \geq \sigma_j^N(\lambda_j) - \lambda_n = \lambda_j - \lambda_n > 0.
$
For $\lambda_j < \lambda_n$, we have
$
\sigma_j^N(\lambda_n) - \lambda_n \leq \sigma_j^N(\lambda_j) - \lambda_n = \lambda_j - \lambda_n < 0.
$
In both cases, we find that
\[
|\sigma_j^N(\lambda_n) - \lambda_n| \geq |\lambda_j - \lambda_n|, \quad \forall\, j
\]
which implies $\gamma_{n, \mathcal{S}_1}\geq \gamma_{n,\mathcal{H}}$ since $\lambda_n=\sigma^N_n(\lambda_n)$. 

Then we establish the convergence of eigenvalue. The numerator in \eqref{Eq-Lem1.1-04} is bounded by
\begin{equation}\label{Eq-Lem1.1-06}
\begin{aligned}    
\bigl|\langle u, \mathcal{F}_N(\lambda_n) u \rangle\bigr|
&= \bigl|\langle \mathcal{V} u, \mathcal{G}_N(\lambda_n) \mathcal{V} u \rangle \bigr| \\
&\leq \|\mathcal{P}_N^\perp \mathcal{V} u\|_{H^{-1}} \|\mathcal{G}_N(\lambda_n) \mathcal{V} u\|_{H^{1}} \\
&\lesssim \|\mathcal{P}_N^\perp \mathcal{V} u\|_{H^{-1}}^2,
\end{aligned}
\end{equation}
where the last inequality follows from Lemma~\ref{Lem5.7} with $t = -1$.
Applying the estimate from Lemmas~\ref{Lem5.5} and \ref{Lem5.2}, we obtain
\begin{equation}\label{eq-thm3-07}
\|\mathcal{P}_N^\perp \mathcal{V} u\|_{H^{-1}} \lesssim N^{-(2-l)} \|\mathcal{V} u\|_{H^{1-l}} \lesssim N^{-(2-l)} \|u\|_{H^1}.
\end{equation}
For $u \in W_n(\mathcal{S}_i)$, with $i = 1,2$, \eqref{Eq-Lem1.1-03} implies
\begin{equation}\label{eq-thm3-08}
\|u\|_{H^1}^2 \leq 2\langle u, (\mathcal{S}_i + M) u \rangle \leq 2 (\lambda_{i,n} + M) \|u\|_{L^2}^2.
\end{equation}
Combining \eqref{Eq-Lem1.1-04}, \eqref{Eq-Lem1.1-06} and \eqref{eq-thm3-07} together, we obtain
\begin{equation}\label{eq-thm3-09}
|\lambda_n^N - \lambda_n| \lesssim N^{-2(2-l)} (\max\{\lambda_n, \lambda_n^N\} + M) \lesssim N^{-2(2-l)}.
\end{equation}
The last inequality follows from the uniform boundedness of $\lambda_n^N$ with respect to $N$, as suggested by the first inequality. Since $l < 2$, \eqref{eq-thm3-09} implies the convergence of the eigenvalue.

Provided $\gamma_{n,\mathcal{S}_1}\geq \gamma_{n,\mathcal{H}}$ and \eqref{eq-thm3-09}, there exists $N_0$ such that for all $N\geq N_0$, the condition of \eqref{Eq-Lem1.1-05}$|\lambda_n - \lambda_n^N| \leq \gamma_{n,\mathcal{H}}/3$. Combining \eqref{Eq-Lem1.1-05} and \eqref{Eq-Lem1.1-03} yields
\begin{equation}\label{eq-thm3-10}
\|\psi_n^N - \mathcal{P}_{n,\mathcal{S}_1}\psi_n^N\|_{H^1}
\lesssim 
\sup_{v\in \mathscr{X}_N\setminus\{0\}}
\left|
\frac{\langle v,\mathcal{F}_N(\lambda_n)\psi_n^N\rangle}{\|v\|_{H^1}}
\right|.
\end{equation}
Following the same line of \eqref{Eq-Lem1.1-06}, the bilinear form is estimated by
\begin{equation}\label{eq-thm3-11}
\left|\langle v,\mathcal{F}_N(\lambda_n)\psi_n^N\rangle\right|
\lesssim \|\mathcal{P}_N^\perp \mathcal{V}v\|_{H^{-1}}
\|\mathcal{P}_N^\perp \mathcal{V}\psi_n^N\|_{H^{-1}}.
\end{equation}
For the first factor, using Lemmas~\ref{Lem5.5} and \ref{Lem5.2} to obtain
\begin{equation}\label{eq-thm3-12}
\|\mathcal{P}_N^\perp \mathcal{V}v\|_{H^{-1}}
\lesssim 
N^{-(2-l)} \|\mathcal{V}v\|_{H^{1-l}}
\lesssim 
N^{-(2-l)} \|v\|_{H^1}.
\end{equation}
For the second factor, using Lemmas~\ref{Lem5.5},\ref{Lem5.2} and \ref{Lem5.4}, we obtain
\begin{equation}\label{eq-thm3-13}
\|\mathcal{P}_N^\perp \mathcal{V}\psi_n^N\|_{H^{-1}}
\lesssim 
N^{-(s+1)} \|\mathcal{V}\psi_n^N\|_{H^s}
\lesssim 
N^{-(s+1)} \|\psi_n^N\|_{H^{s+l}}
\lesssim 
N^{-(s+1)} \|\psi_n^N\|_{L^2}.
\end{equation}
Substituting \eqref{eq-thm3-11}--\eqref{eq-thm3-13} into \eqref{eq-thm3-10} yields the desired estimate (note that $b=2-l$)
\[
\|\psi_n^N - \mathcal{P}_{n,\mathcal{S}_1}\psi_n^N\|_{H^1}
\lesssim 
N^{-(s+1+b)} \|\psi_n^N\|_{L^2}.
\]
Since $(\lambda_n,\mathcal P_{n,\mathcal{S}_1} \psi_n^N)$ is an eigenpair of $\mathcal{S}_1$, \eqref{Eq-Lem1.1-02} implies that there exists $\psi \in \operatorname{ker}(\lambda_n-\mathcal{H})$ such that $\mathcal P_{n,\mathcal{S}_1 } \psi_n^N =\mathcal{P}_N \psi$. This completes the proof of Lemma~\ref{Lem1.1}.
\end{proof}

\begin{proof}[Proof of Theorem~\ref{Thm01}]
    Applying the triangle inequality
    \begin{equation}\label{Eq-thm1-01}
        \|\psi_n^N-\psi\|_{H^1} \lesssim \|\psi^N_n-\mathcal{P}_N\psi\|_{H^1} + \|\psi - \mathcal{P}_N \psi\|_{H^1}.
    \end{equation}
    According to Lemma~\ref{Lem1.1}, There exist $\psi \in \operatorname{ker}(\lambda_n-\mathcal{H})$ such that 
    \begin{equation}\label{Eq-thm1-02}
         \|\psi^N_n-\mathcal{P}_N\psi\|_{H^1}\lesssim N^{-(s+1+b)} \|\psi_n^N\|_{L^2}.
    \end{equation}
    By Lemmas~\ref{Lem5.4} and \ref{Lem5.5},
    \begin{equation}\label{Eq-thm1-03}
        \|\psi-\mathcal{P}_N\psi\|_{H^1} = \|\mathcal{P}_N^\perp\psi\|_{H^1} \lesssim N^{-(s+1)}\|\psi\|_{H^{s+2}}\lesssim N^{-(s+1)}\|\psi\|_{L^2}.
    \end{equation}    
    Inserting \eqref{Eq-thm1-02} and \eqref{Eq-thm1-03} into \eqref{Eq-thm1-01} yields
    \begin{equation}\label{Eq-thm1-04}
        \|\psi_n^N-\psi\|_{H^1} \lesssim N^{-(s+1)}\bigl(\|\psi\|_{L^2}+\|\psi_n^N\|_{L^2}\bigr)\lesssim N^{-(s+1)}\|\psi\|_{L^2}.
    \end{equation}
    This proves the estimate for eigenvector. 
    
    The eigenvalue is bounded using the Rayleigh quotient $\lambda_n^N = \langle \psi_n^N,\mathcal{H}\psi_n^N\rangle/\|\psi_n^N\|_{L^2}^2$:
    \begin{equation}
    \begin{aligned}
        |\lambda_n^N-\lambda_n| \|\psi_n^N\|^2_{L^2}
        &= 
        \left|\langle \psi_n^N,(\mathcal{H}-\lambda_n)\psi_n^N\rangle\right|\\
        &=
        \left| \langle \psi_n^N-\psi,(\mathcal{H}-\lambda_n)(\psi_n^N-\psi)\rangle  \right|\\
        &\lesssim 
        \|\psi_n^N-\psi\|_{H^1}^2 + (\lambda_n+M)\|\psi_n^N-\psi\|_{L^2}^2\\
        &\lesssim N^{-(2s+2)}\|\psi\|_{L^2}^2\\
        &\lesssim N^{-(2s+2)}\|\psi_n^N\|_{L^2}^2,
    \end{aligned}
    \end{equation}
    where the first inequality is obtained using Lemma~\ref{Lem5.3}, and last two inequalities are obtained using \eqref{Eq-thm1-04}. Dividing both side $\|\psi_n^N\|_{L^2}^2$ yields the desired estimate for eigenvalue.
\end{proof}

\section{Analysis of AR and AR-FSM}
\label{sec:proofAR}

This section is devoted to the proofs of Lemma~\ref{Lem2.1} and Theorem~\ref{Thm03}. 
The convergence result for AR-FSM stated in Theorem~\ref{Thm02} was established in Section~\ref{sec:postproc} using these results.
We first analyze the basic properties of the space $PSH^{t,r}$ in Section~\ref{subsec:AR1}, then study the existence and regularity of the asymptotic functions in Section~\ref{subsec:AR2}. Finally, Sections~\ref{subsec:AR3} and \ref{subsec:AR4} are devoted to the proofs of Lemma~\ref{Lem2.1} and Theorem~\ref{Thm03}, respectively.

\subsection{Properties of function space $PSH^{t,r}$} 
\label{subsec:AR1} 
Our previous definition Definition~\ref{Def01} for $PSH^{t,r}$ is convenient for verifying whether a function belongs to $PSH^{t,r}$, but is less suitable for theoretical analysis since the choice of $V_{\mathrm{lc}}$ is not unique. We therefore first establish Lemma~\ref{Lem6.1}, which provides some equivalent characterizations of $PSH^{t,r}$ better suited for analysis.
To state this lemma, let $h_c:\mathbb{R}^d\to[0,1]$ be a $C^\infty$ cutoff function such that
\begin{equation}\label{eq-sec6.1-01}
\sum_{\bm n\in\mathbb Z^d} h_c(\bm x+\bm n)=1,
\quad
\forall\,\bm x\in\mathbb R^d,
\qquad
h_c(\bm x)=0,
\quad
\forall\,\bm x\in\mathbb{R}\setminus\Omega_{\mathrm{lc}},
\end{equation}
where $\Omega_{\mathrm{lc}}:= [-3/4,3/4]^d$ is chosen so as to exclude all periodic images of the singular point $\bm x=\bm 0$.
One can choose $h_c$ as the tensor product of one-dimensional cutoff functions. For each multi-index $\bm \alpha\in\mathbb N^d$, define the periodic analogue of $\bm x^{\bm \alpha}$ by
\begin{equation}\label{Eq-6.1-02}
    m_{\bm \alpha}(\bm x)
    :=
    \prod_{j=1}^{d}
    \bigl(e^{2\pi i x_j}-1\bigr)^{\alpha_j}.
\end{equation}
The function $m_{\bm \alpha}$ is periodic in $[0,1]^d$ and $C^\infty$ smooth, and satisfies
\[
m_{\bm \alpha}/\bm x^{\bm \alpha} \in C^{\infty}(\mathbb{R}^d),
\qquad
\bm x^{\bm \alpha}/m_{\bm \alpha} \in C^{\infty}(\Omega_{\mathrm{lc}}).
\]

\begin{lemma}\label{Lem6.1}
Let $t\in\mathbb R$ and $r\in\mathbb N$. The following are equivalent: $(\mathrm{i})$ $ f\in PSH^{t,r}$; $(\mathrm{ii})$  $\bm x^{\bm \alpha} h f \in H^{t+|\bm \alpha|}(\mathbb R^d)$, $\forall\, |\bm \alpha|\le r$; $(\mathrm{iii})$ $m_{\bm \alpha} f \in H^{t+|\bm \alpha|}$, $\forall\, |\bm \alpha|\le r$.
\end{lemma}
\begin{proof}
We prove the equivalence by showing
$\mathrm{(i)} \Rightarrow \mathrm{(iii)} \Rightarrow \mathrm{(ii)} \Rightarrow \mathrm{(i)}$.

$\mathrm{(i)} \Rightarrow \mathrm{(iii)}$. Let $f \in PSH^{t,r}$. By Definition~\ref{Def01}, there exists a compactly supported $f_{\mathrm{lc}}$ such that 
    $ f = \sum_{\bm n\in \mathbb{Z}^d}f_{\mathrm{lc}}(\cdot+\bm n) $
    with 
    $\bm x^{\bm \alpha}f_{\mathrm{lc}}\in H^{t+|\bm \alpha|}(\mathbb{R}^d)$, 
    for every $|\bm \alpha|\leq r$. Owing to the compact support of $f_{\mathrm{lc}}$, only finitely many $\bm n \in \mathbb{Z}^d$ contribute to
    $f(\bm x)=\sum_{\bm n} f_{\mathrm{lc}}(\bm x+\bm n)$ for $\bm x \in [0,1]^d$. This allows us to bound
    $\| m_{\bm \alpha}f\|_{H^{t+|\bm \alpha|}}$ by $\| m_{\bm \alpha}f_{\mathrm{lc}}\|_{H^{t+|\bm \alpha|}(\mathbb{R}^d)}$. Therefore, for every $|\bm \alpha|\leq r$
    \[
    \|m_{\bm \alpha}f\|_{H^{t+|\bm \alpha|}} 
    \lesssim \| m_{\bm \alpha}f_{\mathrm{lc}}\|_{H^{t+|\bm \alpha|}(\mathbb{R}^d)} 
    \lesssim \| \bm x^{\bm \alpha}f_{\mathrm{lc}}\|_{H^{t+|\bm \alpha|}(\mathbb{R}^d)}<+\infty,
    \]
    where the second follows from $m_{\bm \alpha}/\bm x^{\bm \alpha}\in C^{\infty}(\mathbb{R}^d)$.

$\mathrm{(iii)} \Rightarrow \mathrm{(ii)}$. 
   Since $h_c$ is compactly supported in $\Omega_{\mathrm{lc}}$ and
    $\bm x^{\bm \alpha}h_c / m_{\bm \alpha} \in C^{\infty}(\Omega_{\mathrm{lc}})$, we conclude $\mathrm{(ii)}$ by
    \[
    \|\bm x^{\bm \alpha} h_cf\|_{H^{t+|\bm \alpha|}(\mathbb{R}^d)}
    =
    \|\bm x^{\bm \alpha} h_cf\|_{H^{t+|\bm \alpha|}(\Omega_{\mathrm{lc}})}
    \lesssim
    \|m_{\bm \alpha} f\|_{H^{t+|\bm \alpha|}(\Omega_{\mathrm{lc}})}
    \lesssim
    \|m_{\bm \alpha} f\|_{H^{t+|\bm \alpha|}}.
    \]
    
$\mathrm{(ii)} \Rightarrow \mathrm{(i)}$.
    Let $f_{\mathrm{lc}}(\bm x):= h_c(\bm x)f(\bm x)$ for all $\bm x\in \mathbb{R}^d$. Then $\mathrm{(ii)}$ is exactly $\bm x^{\bm \alpha}f_{\mathrm{lc}}\in H^{t+|\bm \alpha|}(\mathbb{R}^d),\ \forall\, |\bm \alpha|\leq r$. Therefore we conclude $f\in PSH^{t,r}$ by
    \[
    \sum_{\bm n\in \mathbb{Z}^d} f_{\mathrm{lc}}(\bm x+\bm n)
    =
    \sum_{\bm n\in \mathbb{Z}^d} h_c(\bm x+\bm n) f(\bm x)
    =
    f(\bm x).
    \]
\end{proof}

The function space $PSH^{t,r}$ is generated by two operators $\mathcal{D}$ and $\mathcal{M}$, where $\mathcal{D}$ denotes the first-order derivative operator and $\mathcal{M}$ denotes the multiplication operator by $\bm x$ (or $m_{\bm \alpha}$ in the periodic setting). 
In accordance with Lemma~\ref{Lem6.1} (ii), we define the $PSH^{t,r}$ norm by
\begin{equation}\label{eq-sec6.1-03.1}
    \|f\|_{PSH^{t,r}}:= \sum_{|\bm \alpha|\leq r}\bnm{\bm x^{\bm \alpha }h_cf}_{H^{t+|\bm \alpha|}(\mathbb{R}^d)} \asymp \sum_{|\bm \beta|\leq |\bm \alpha|\leq r}\bnm{\mathcal{D}^{\bm \beta}\mathcal{M}^{\bm \alpha}[h_c f]}_{H^{t}(\mathbb{R}^d)} .
\end{equation}
This definition directly implies the embedding
\begin{equation}\label{eq-sec6.1-03}
    PSH^{t,r+1}\hookrightarrow PSH^{t,r}\hookrightarrow \cdots\hookrightarrow PSH^{t,0}=H^{t},
\end{equation}
and the induced commutative diagram structure 
\begin{equation}\label{eq-sec6.1-diag}
\begin{tikzcd}
PSH^{t,r}
  \arrow[r,"\mathcal{M}"]
  \arrow[d,"\mathcal{D}"]
  \arrow[rd,dotted,"\mathrm{Id}"]
&
PSH^{t+1,r-1}
  \arrow[d,"\mathcal{D}"']
&
PSH^{t,0} = H^{t}
    \arrow[d,"\mathcal{D}"]
\\
PSH^{t-1,r}
  \arrow[r,"\mathcal{M}"']
&
PSH^{t,r-1},
&
PSH^{t-1,0} = H^{t-1}.
\end{tikzcd}
\end{equation}
As a consequence, applying this diagram iteratively, we obtain
\begin{equation}\label{eq-sec6.1-04}
    \bnm{\bm x^{\bm \alpha}\mathcal{D}^{\bm \beta}[h_cf]}_{H^{t}(\mathbb{R}^d)}\lesssim \|f\|_{PSH^{t+|\bm \beta|-|\bm \alpha|,|\bm \alpha|}}.
\end{equation}

Next, we study the boundedness property of  $V \in PSH^{s,r}$ as a multiplication operator. Lemma~\ref{Lem6.2} shows that, for $t\leq s$, multiplication by $V$ defines a bounded map from $H^{t+r+l}$ to $H^{t+r}$, extending the admissible index range by $r$ compared with Lemma~\ref{Lem5.2}.

\begin{lemma}\label{Lem6.2}
    Let $V \in PSH^{s,r}$ with $r\in\mathbb{N}$. If $\phi \in H^{t+r+l}$ satisfies $ \mathcal{D}^{\bm \alpha}_{\bm 0}[\phi]=0$ for every $\bm |\bm \alpha| <r$, then for $-1\leq t+r\leq s+r$, we have
        \begin{equation}\label{eq-Lem6.2-01}
        \|V\phi\|_{H^{t+r}}\lesssim \|V\|_{PSH^{s,r}}\| \phi \|_{H^{t+r+l}}.
        \end{equation}
\end{lemma}

The proof of this lemma relies on an auxiliary lemma stated in Lemma~\ref{Lem6.3}, whose proof is deferred to Appendix~\ref{prooflem2}.

\begin{lemma}\label{Lem6.3}
    Let $m> \max\{0,d/2-1\}$, $r \in \mathbb{N}$. If $\phi \in H^{m+r}$ satisfies ${D}^{\bm \alpha}_{\bm 0}[\phi] = 0$ for every $|\bm \alpha| <r$, then $\phi$ admits the decomposition
    \begin{equation}\label{eq-Lem6.3-01}
        \phi(\bm x) = \sum_{|\bm \alpha|=r}\bm x^{\bm \alpha}\phi_{\bm \alpha}(\bm x),\quad \|\phi_{\bm \alpha}\|_{H^{m}(B_R)}\lesssim \|\phi\|_{H^{m+r}},
    \end{equation}
    where $B_R$ denotes the ball centered at $\bm x=\bm 0$ with radius $R>0$.
\end{lemma}

\begin{proof}[Proof of Lemma~\ref{Lem6.2}]
    We prove by induction on $r$. The case $r = 0$ follows from Lemma~\ref{Lem5.2}. Assume that, for some $r\ge 1$, the statement holds for every $0\le r'<r$. We show that it also holds for $r$.
    The proof is divided into three cases.
    
    \textbf{Case 1.} If $t\leq s-1$, set $r'=r-1, t'=t+1$. Then $-1\leq t'+r'\leq s+r'$, the induction hypothesis with $r'$ applied to $t'$ and $\phi$ implies
    \begin{equation}\label{eq-Lem6.2-02}
        \|V\phi\|_{H^{t'+r'}}\lesssim \|V\|_{PSH^{s,r'}}\|\phi \|_{H^{t'+r'}}\leq \|V\|_{PSH^{s,r}}\|\phi\|_{H^{t'+r'}},
    \end{equation}
    which is the desired \eqref{eq-Lem6.2-01} since $t'+r'=t+r$.
    
    \textbf{Case 2.} If $s-1<t\leq s$ and $s\geq 0$, then $t\geq -1 $. We claim that for all $0\leq |\bm \beta|\leq r$, if a function $f$ satisfies $\mathcal{D}_{\bm 0}^{\bm \alpha}f = 0$ for every $|\bm \alpha|< |\bm \beta|$, then
    \begin{equation}\label{eq-Lem6.2-03}
        \|\mathcal{D}^{\bm \beta}[V]f\|_{H^{t}}\lesssim \|V\|_{PSH^{s,|\bm \beta|}}\|f\|_{H^{t+l+|\bm \beta|}}.
    \end{equation}
    To prove this claim, choose $R>0$ such that $\Omega_{\mathrm{lc}}\subset B_{R}$. Since $t+l>s+l-1 \geq \max\{0,d/2-1\}$, we decompose $f = \sum_{|\bm \alpha|=|\bm \beta|}\bm x^{\bm \beta}f_{\bm \alpha}$ via Lemma~\ref{Lem6.3}, with estimate $\|f_{\bm \alpha}\|_{H^{t+l}(B_R)}\lesssim \|f\|_{H^{t+l+|\bm \beta|}}$. Let $\mathcal{I}$ be the extension operator from $B_{R}$ to $\mathbb{R}^d$ that satisfies $\mathcal{I}f_{\bm \alpha}|_{B_R}=f_{\bm \alpha}$ and $\|\mathcal{I}f_{\bm \alpha}\|_{H^{t+l}(\mathbb{R}^d)}\lesssim \|f\|_{H^{t+l}(B_R)}$. 
    Then we prove the claim \eqref{eq-Lem6.2-03} by 
    \begin{equation}\label{eq-Lem6.2-04}
    \begin{aligned}
        \bnm{\mathcal{D}^{\bm \beta}[V]f}_{H^{t}}
        &\lesssim
        \bnm{\mathcal{D}^{\bm \beta}[h_cV]f}_{H^{t}(\mathbb{R}^d)}\\
        &\leq 
        \sum_{|\bm \alpha|=|\bm \beta|}
        \bnm{\bm x^{\bm \alpha}\mathcal{D}^{\bm \beta}[h_cV]\mathcal{I}f_{\bm \alpha}}_{H^{t}(\mathbb{R}^d)}\\
        &\lesssim 
        \sum_{|\bm \alpha|=|\bm \beta|}
        \bnm{\bm x^{\bm \alpha}\mathcal{D}^{\bm \beta}[h_cV]}_{H^{s}(\mathbb{R}^d)}
        \nm{\mathcal{I}f_{\bm \alpha}}_{H^{t+l}(\mathbb{R}^d)}\\
        &\lesssim \nm{V}_{PSH^{s,|\bm \beta|}}\|f\|_{H^{t+r+l}}.
    \end{aligned}
    \end{equation}
    Here the third inequality follows from the embedding $H^{s}(\mathbb{R}^d)\times H^{t+l}(\mathbb{R}^d)\hookrightarrow H^{t}(\mathbb{R}^d)$ implied by Lemma~\ref{Lem5.1}, and the last inequality follows  from \eqref{eq-sec6.1-04}. Applying the claim \eqref{eq-Lem6.2-03} to $f = \mathcal{D}^{\bm \alpha}\phi$ for every $|\bm \alpha|\leq r-|\bm \beta|$, we obtain the desired \eqref{eq-Lem6.2-01} by
    \begin{equation}\label{eq-Lem6.2-05}
    \begin{aligned}
    \|V\phi\|_{H^{t+r}}
    &\lesssim
    \sum_{|\bm \beta|\leq r}\bigl\|\mathcal{D}^{\bm \beta}[V\phi]\bigr\|_{H^{t}}
    \\&\lesssim
    \sum_{|\bm \beta|\leq r}\sum_{|\bm \alpha|\leq r-|\bm \beta|}\bnm{\mathcal{D}^{\bm \beta}[V]\mathcal{D}^{\bm \alpha}[\phi]}_{H^t}.
    \\&\lesssim
    \sum_{|\bm \beta|\leq r}\sum_{|\bm \alpha|\leq r-|\bm \beta|}
    \nm{V}_{PSH^{s,|\bm \beta|}}\bnm{\mathcal{D}^{\bm \alpha}\phi}_{H^{t+|\bm \beta|+l}}
    \\&\lesssim
    \nm{V}_{PSH^{s,r}}\nm{\phi}_{H^{t+l+|\bm \beta|}}.
    \end{aligned}
    \end{equation}  
   
    \textbf{Case 3.} If $s-1<t\leq s$ and $-1<s<0$, then $l\geq1-s>1$. 
    We modify \eqref{eq-Lem6.2-03} into the following claim. For $0\leq |\bm \beta|\leq r-1$, if a function $f$ satisfies $\mathcal{D}_{\bm 0}^{\bm \alpha}f = 0$ for every $|\bm \alpha|<|\bm \beta|+1$, then
    \begin{equation}\label{eq-Lem6.2-06}
        \|\mathcal{D}^{\bm \beta}[V]f\|_{H^{t+1}}\lesssim \|V\|_{PSH^{s,|\bm \beta|+1}}\|f\|_{H^{t+l+|\bm \beta|+1}}.
    \end{equation}
    The proof of this claim follows from the same argument in \eqref{eq-Lem6.2-04} 
    \begin{equation}\label{eq-Lem6.2-07}
    \begin{aligned}
        \bnm{\mathcal{D}^{\bm \beta}[V]f}_{H^{t+1}}
        &\lesssim 
        \sum_{|\bm \alpha|=|\bm \beta|+1}
        \bnm{\bm x^{\bm \alpha}\mathcal{D}^{\bm \beta}[h_cV]\mathcal{I}f_{\bm \alpha}}_{H^{t+1}(\mathbb{R}^d)}
        \\&\lesssim 
        \sum_{|\bm \alpha|=|\bm \beta|+1}
        \bnm{\bm x^{\bm \alpha}\mathcal{D}^{\bm \beta}[h_cV]}_{H^{s+1}(\mathbb{R}^d)}
        \nm{\mathcal{I}f_{\bm \alpha}}_{H^{t+l}(\mathbb{R}^d)}\\
        &\lesssim \nm{V}_{PSH^{s,|\bm \beta|+1}}\|f\|_{H^{t+l+|\bm \beta|+1}}.
    \end{aligned}
    \end{equation}
    The second inequality follows from the embedding $H^{s'}(\mathbb{R}^d)\times H^{t'+l'}(\mathbb{R}^d)\hookrightarrow H^{t'}(\mathbb{R}^d)$ implied by Lemma~\ref{Lem5.1} with $t'=t+1,\ s'=s+1$ and $l'=\min\{0,l-1\}=l-1$, and the last inequality follows  from \eqref{eq-sec6.1-04}. Applying the claim \eqref{eq-Lem6.2-06} to $f = \mathcal{D}^{\bm \alpha}\phi$ for every $|\bm \alpha|\leq r-1-|\bm \beta|$, we obtain the desired \eqref{eq-Lem6.2-01} by
    \begin{equation}\label{eq-Lem6.2-08}
    \begin{aligned}
    \|V\phi\|_{H^{t+r}}
    &\lesssim
    \sum_{|\bm \beta|\leq r-1}\bigl\|\mathcal{D}^{\bm \beta}[V\phi]\bigr\|_{H^{t+1}}
    \\&\lesssim
    \sum_{|\bm \beta|\leq r-1}\sum_{|\bm \alpha|\leq r-1-|\bm \beta|}\bnm{\mathcal{D}^{\bm \beta}[V]\mathcal{D}^{\bm \alpha}[\phi]}_{H^{t+1}}.
    \\&\lesssim
    \sum_{|\bm \beta|\leq r-1}\sum_{|\bm \alpha|\leq r-1-|\bm \beta|}
    \nm{V}_{PSH^{s,|\bm \beta|+1}}\bnm{\mathcal{D}^{\bm \alpha}\phi}_{H^{t+l+|\bm \beta|+1}}
    \\&\lesssim
    \nm{V}_{PSH^{s,r}}\nm{\phi}_{H^{t+r+l}}.
    \end{aligned}
    \end{equation}  
\end{proof}

At the end of this section, we establish Lemma~\ref{Lem6.4}, showing that $f\in PSH^{t,r}$ has $H^{t+r}$ regularity away from the singular point $\bm x=\bm 0$.

\begin{lemma}\label{Lem6.4} Let $f\in PSH^{t,r}$ with $t\in \mathbb{R}$ and $r\in \mathbb{N}$. For any $\delta>0$, $f$ admits a decomposition $f= f^{\mathrm{L}}+f^{\mathrm{G}}$ such that $f^{\mathrm{L}}=0$ on $[0,1]^{d}\setminus B_{\mathrm{per}}(\delta)$ and
    \begin{equation}
    \|f^{\mathrm{L}}\|_{PSH^{t,r}}
    +\|f^{\mathrm{G}}\|_{H^{t+r}}
    \lesssim \| f \|_{PSH^{t,r}},
    \end{equation}
\end{lemma}
where the $B_{\mathrm{per}}(\delta)$ denotes the periodic ball of radius $\delta$
\[
B_{\mathrm{per}}(\delta)
:=
\left\{
\bm x\in [0,1]^d :
|\bm x+\bm n|<\delta
\text{ for some }
\bm n\in\mathbb Z^d
\right\}.
\]

\begin{proof}
  Define $f_{\mathrm{lc}} = h_cf$. Let $h'_c: \mathbb{R}^d \rightarrow [0,1]$ be a $C^\infty$ smooth cutoff function that is $1$ inside the ball of radius $\delta/2$, $0$ outside the ball of radius $\delta$. Define
$
f^{\mathrm{L}}(\bm x):= \sum_{\bm n \in \mathbb{Z}^d} h_c'\bigl(\bm x+\bm n\bigr)f_{\mathrm{lc}}(\bm x+\bm n),
$
which is 0 on $[0,1]^{d}\setminus B_{\mathrm{per}}(\delta)$. Since $h\in C^{\infty}(\mathbb{R}^d)$, we have $ \|f^{\mathrm{L}}\|_{PSH^{t,r}}\lesssim \| f\|_{PSH^{t,r}}$. Note that $f^{\mathrm{G}}:=f-f^{\mathrm{L}}$ is the periodic sum of $(1-h_c')f_{\mathrm{lc}}$, where
\[
(1-h_c'(\bm x))f_{\mathrm{lc}}(\bm x) = \frac{1-h_c'(\bm x)}{\sum_{j=1}^dx_{j}^{2r}}\sum_{j=1}^{d}x_j^{r}\bigl( x_j^{r} f_{\mathrm{lc}}(\bm x)\bigr).
\]
Since $(1-h_c'(\bm x))x_j^r/(\sum_{j=1}^{d} x_j^{2r}) \in C^{\infty}(\Omega_{\mathrm{lc}})$, we obtain
\begin{equation*}
\| f^{\mathrm{G}}\|_{H^{t+r}}
\lesssim \|(1-h_c')f_{\mathrm{lc}}\|_{H^{t+r}(\mathbb{R}^d)} 
\lesssim \sum_{j=1}^{d} \bigl\|x_j^rf_{\mathrm{lc}}\bigr\|_{H^{t+r}(\mathbb{R}^d)} \leq \| f\|_{PSH^{t,r}}.
\end{equation*}
\end{proof}

\subsection{Existence and regularity of asymptotic function}\label{subsec:AR2}
This section studies the asymptotic functions 
$\{\Phi_{\bm \alpha} : |\bm \alpha| \leq 1\}$, which satisfy
\begin{equation}\label{eq-sec6.2-01}
(\mathcal{T}+\mathcal{V}) \Phi_{\bm \alpha} \in H^{s+2}, 
\qquad 
\mathcal{D}_{\bm 0}^{\bm \alpha'}[\Phi_{\bm \alpha}] = \delta_{\bm \alpha,\bm \alpha'}, 
\quad \forall\, |\bm \alpha'|\leq |\bm \alpha|,
\end{equation}
for potentials $V \in PSH^{s,2}$. Lemma~\ref{Lem6.5} guarantees the existence of $\Phi_{\bm \alpha}$. Its proof, based on an iterative regularity-bootstrapping procedure for the residual, is deferred to Appendix~\ref{proofLem6.5}.
Provided that $\Phi_{\bm \alpha}$ satisfies \eqref{eq-sec6.2-01}, Lemma~\ref{Lem6.6} implies
$
\Phi_{\bm \alpha} \in PSH^{s+2+|\bm \alpha|,\,2-|\bm \alpha|}.
$

For convenience in the proof, the statements of Lemmas~\ref{Lem6.5} and \ref{Lem6.6} are formulated for the operator $\mathcal{T}+1+\mathcal{V}$.
Since $V'=V-1\in PSH^{s,2}$, applying  Lemmas~\ref{Lem6.5} and \ref{Lem6.6} to $V'$ recovers \eqref{eq-sec6.2-01}.

\begin{lemma}\label{Lem6.5}
Let $V \in PSH^{s,2}$. 
Then for every multi-index $\bm \alpha \in \mathbb{N}^d$ with $|\bm \alpha| \leq 1$, there exists a function $\Phi_{\bm \alpha} \in H^1$ satisfying
\begin{equation}\label{eq-Lem6.5-01}
(\mathcal{T}+1+\mathcal{V}) \Phi_{\bm \alpha} = R_{\bm \alpha}\in H^{s+2}, 
\qquad 
\mathcal{D}^{\bm \alpha'}_{\bm 0}[\Phi_{\bm \alpha}] = \delta_{\bm \alpha',\bm \alpha}, 
\quad \forall\, |\bm \alpha'|\leq |\bm \alpha|.
\end{equation}
\end{lemma}

\begin{lemma}\label{Lem6.6}
Let $V \in PSH^{s,2}$ and $r \in \{0,1,2\}$. 
Suppose $\Phi \in H^1$ satisfies
\begin{equation}
(\mathcal{T}+1+\mathcal{V})\Phi = R \in H^{s+2}, 
\qquad 
\mathcal{D}_{\bm 0}^{\bm \alpha}[\Phi] = 0, \quad \forall\, |\bm \alpha| < r.
\end{equation}
Then $\Phi \in PSH^{s+2+r,\,2-r}$.
\end{lemma}
\begin{proof}[Proof of Lemma~\ref{Lem6.6}]
    Define $\mathcal{T}_{\star}=\mathcal{T}+1$. The proof proceeds in two steps.
    
    \textbf{Step 1.} $\Phi \in H^{s+2+r}$.
        We argue by a bootstrap argument. By Lemma~\ref{Lem6.2}, for $-1\leq t+r\leq s+r$, taking the $H^{t+r}$ norm on equation 
        $
        \mathcal{T}_{\star}\Phi = R - \mathcal{V}\Phi
        $
        gives
        \begin{equation}\label{eq-Lem6.6-01}
        \begin{aligned}
          \|\Phi\|_{H^{t+r+2}} \lesssim \|R\|_{H^{t+r}}+\|V\Phi\|_{H^{t+r}}
          \lesssim \|R\|_{H^{t+r}} + \|\Phi\|_{H^{t+r+l}}.
        \end{aligned}
        \end{equation}
        This estimate implies the bootstrapping implication
        $
        \Phi \in H^{t+r+l} \Longrightarrow \Phi\in H^{t+r+2}
        $.
        Starting from $\psi \in H^{1}$, a finite iteration of this argument lifts the regularity to $\Phi \in H^{s+2+r}$.
        
    \textbf{Step 2.} $\Phi \in PSH^{s+2+r,2-r}$. We first claim that for every $r'\leq 2-r$
        \begin{equation}\label{eq-Lem6.6-02}
        \bigl(\mathcal{T}_{\star}+\mathcal{V}\bigr)(m_{\bm \alpha}\Phi) \in H^{s+r+|\bm \alpha|},\ \forall\, |\bm \alpha|\leq r' 
        \quad \Longrightarrow \quad \Phi\in PSH^{s+2+r,r'}.
        \end{equation}To prove this claim, we observe from the diagram \eqref{eq-sec6.1-diag} that $m_{\bm \alpha}V \in PSH^{s+|\bm \alpha|,\,2-|\bm \alpha|}$. Then, by taking $s' = s+|\bm \alpha|$, $t' = s'$ and $l' = \max\{0,l-|\bm \alpha|\}$ in Lemma~\ref{Lem6.2}, we have
        \begin{equation}\label{eq-Lem6.6-03}
        \|m_{\bm \alpha}V\Phi\|_{H^{s+r+|\bm \alpha|}}
        \lesssim
        \|m_{\bm \alpha}V\|_{PSH^{s+|\bm \alpha|,\,2-|\bm \alpha|}}
        \|\Phi\|_{H^{s+r+|\bm \alpha|+l'}}.
        \end{equation}
        Since
        \[
        s+r+|\bm \alpha|+l'
        =
        s+r+\max\{|\bm \alpha|,l\}
        \le
        s+r+2,
        \]
        and $\Phi\in H^{s+r+2}$, we conclude from \eqref{eq-Lem6.6-03} that $m_{\bm \alpha}V\Phi \in H^{s+r+|\bm \alpha|}$.
        Then, by the left-hand side of \eqref{eq-Lem6.6-02}, we obtain
        $m_{\bm \alpha}\Phi \in H^{s+r+|\bm \alpha|+2}$ for every $|\bm \alpha|\leq r'$, which means $\Phi \in PSH^{s+2+r,r'}$ by Lemma~\ref{Lem6.1} (iii).
    
      Now we turn to estimate the left-hand side of \eqref{eq-Lem6.6-02}. A direct calculation by expanding $\Delta (m_{\bm \alpha}\Phi)$ yields
        \begin{equation}\label{eq-Lem6.6-04}
            \bigl(\mathcal{T}_{\star}+\mathcal{V}\bigr)(m_{\bm \alpha}\Phi) =m_{\bm \alpha}R + \Phi \mathcal{T}m_{\bm \alpha} -            2\sum_{|\bm \beta|=1}\mathcal{D}^{\bm \beta}[m_{\bm \alpha}]\mathcal{D}^{\bm \beta}[\Phi].
        \end{equation}
        Since $m_{\bm \alpha}$ is $C^{\infty}$ smooth and $\Phi\in H^{s+2+r}$, we have
        \begin{equation}\label{eq-Lem6.6-05}
            m_{\bm \alpha} R \in H^{s+2},\quad \Phi \mathcal{T}_{\star}m_{\bm \alpha} \in H^{s+2+r},\quad \mathcal{D}^{\bm \beta}[m_{\bm \alpha}]\mathcal{D}^{\bm \beta}[\Phi]\in H^{s+1+r}.
        \end{equation}
       Combining \eqref{eq-Lem6.6-04} and \eqref{eq-Lem6.6-05}, we establish the left-hand side of \eqref{eq-Lem6.6-02} for $r,r'$ satisfying
        \[
        s+r+r' \le \min\{s+2,\; s+2+r,\; s+1+r\}
        \quad \Longleftrightarrow \quad
        r' \le \min\{2-r,\,1\}.
        \]
        This implies that $\Phi \in PSH^{s+2+r,r'}$ for $r' \le \min\{2-r,1\}$. The only remaining case for proving $\Phi \in PSH^{s+2+r,2-r}$ is $r=0$.
        In this case, it suffices to show the cross term
        $\mathcal{D}^{\bm \beta}[m_{\bm \alpha}]\mathcal{D}^{\bm \beta}[\Phi] \in H^{s+2}.$
        Write $\bm \alpha = \bm \alpha_1 +\bm \alpha_2$ with $|\bm \alpha_1|=|\bm \alpha_2|=1$. Then 
        \[
        \mathcal{D}^{\bm \beta}[m_{\bm \alpha}]
        \mathcal{D}^{\bm \beta}[\Phi]
        = 
        \mathcal{D}^{\bm \beta}[m_{\bm \alpha_2}]
        \mathcal{D}^{\bm \beta}[m_{\bm \alpha_1}\Phi]
        +
        \mathcal{D}^{\bm \beta}[m_{\bm \alpha_1}]
        \mathcal{D}^{\bm \beta}[m_{\bm \alpha_2}\Phi]
        -2
        \Phi\mathcal{D}^{\bm \beta}[m_{\bm \alpha_2}]
        \mathcal{D}^{\bm \beta}[m_{\bm \alpha_1}].\]
        Each term on the right‑hand side lies in  $H^{s+2}$ owing to $m_{\bm \alpha_i}\Phi\in H^{s+3}$ (by the previously established $\Phi \in PSH^{s+2,1}$ in the case $r=0,r'=1$). This completes the proof of  Lemma~\ref{Lem6.6}.
\end{proof}

\subsection{Proof of \texorpdfstring{Lemma~\ref{Lem2.1}}{Theorem X}}
\label{subsec:AR3}
This section is devoted to the proof of Lemma~\ref{Lem2.1}.
We denote $\mathcal{D}_{p}^{\bm \alpha} := \mathcal{D}^{\bm \alpha}_{\bm x_p}$ for the $p$-th singular point $\bm x_p$ in Assumption~\ref{Asp::ARFSM}.  By  Lemmas~\ref{Lem6.5} and \ref{Lem6.6}, There exists 
\begin{equation}\label{Eq-6.3-01}
\Phi_{p,\bm \alpha}\in PSH^{s+2+|\bm \alpha|,2-|\bm \alpha|}
\quad \text{such that} \quad
\left\{
\begin{aligned}
(\mathcal{T}+\mathcal{V}_p)\Phi_{p,\bm \alpha}
&=
R_{p,\bm \alpha}\in H^{s+2},\\
\mathcal{D}_{p}^{\bm \alpha'}[\Phi_{p,\bm \alpha}]
&=
\delta_{\bm \alpha',\bm \alpha},
\quad
\forall\, |\bm \alpha'|\le |\bm \alpha|.
\end{aligned}
\right.
\end{equation}

\begin{proof}[Proof of Lemma~\ref{Lem2.1}] 
We define $\psi^{\geq 1}: = \psi^{(1)}+\psi^{\geq 2}.$
For $r=1,2$, substituting the expansion $\psi = \sum_{0\leq j<r}^{}\psi^{(j)} +\psi^{\geq r},$ and \eqref{Eq-2.2-01} into the eigenvalue equation $(H-\lambda)\psi = 0$ yields
\begin{equation}\label{Eq-lem2.1-03}
    0 = \mathcal{T}\psi^{\geq r} + \sum_{p=1}^{n_p} \mathcal{V}_p\Bigl(\psi - \sum_{|\bm \alpha|<r}\beta_{p,\bm \alpha}\Phi_{p,\bm \alpha}\Bigr) + R(\lambda),
\end{equation}
with the residual
\begin{equation}\label{Eq-lem2.1-04}
    R(\lambda) = \sum_{p=1}^{n_p}\sum_{|\bm \alpha|<r}\beta_{p,\bm \alpha}R_{p,\bm \alpha} + (\mathcal{V}_0 - \lambda)\psi.
\end{equation}

    By Lemma~\ref{Lem5.4}, we have $\|\psi\|_{H^{s+2}}\lesssim\nm{\psi}_{L^2}$. 
    The proof  proceeds in two steps of bootstrapping via \eqref{Eq-lem2.1-03}.
    
     \textbf{Step 1.} $r=1\Rightarrow \psi^{\geq 1}\in H^{s+3}$. By the Sobolev embedding $H^{s+2}\hookrightarrow L^{\infty}$, the coefficients $\beta_{p,\bm 0}= \mathcal{D}_{p}^{\bm 0}[\psi]$ are bounded by
    \begin{equation}\label{Eq-lem2.1-05}
        \quad |\beta_{p,\bm 0}| \leq \nm{\psi}_{L^{\infty}}\lesssim\|\psi\|_{H^{s+2}}\lesssim \|\psi\|_{L^2}.
    \end{equation}
    Since $\Phi_{p,0}\in PSH^{s+2,2}\hookrightarrow H^{s+2}$, we use \eqref{Eq-lem2.1-05} to obtain 
    \begin{equation}\label{Eq-lem2.1-06}
       \| \psi^{(0)} \|_{H^{s+2}} = \Bigl\| \sum_{p=1}^{n_p}\beta_{p,\bm 0}\Phi_{p,\bm 0}\Bigr\|_{H^{s+2}} \lesssim \|\psi\|_{L^2},
    \end{equation}
    where the last inequality follows from $\Phi_{p,\bm \alpha}\in PSH^{s+2+|\bm \alpha|,2-|\bm \alpha|}$ that only depends on $\mathcal{V}_p$. Consequently,
    \[
    \|\psi^{\geq 1}\|_{H^{s+2}}=\|\psi-\psi^{(0)}\|_{H^{s+2}}\lesssim \|\psi\|_{L^2},
    \]
    which serves as the starting point of the bootstrap argument. 
    
    Let $t$ satisfies $s\leq t+1 \leq s+1$. Taking $H^{t+1}$ norm on both sides of \eqref{Eq-lem2.1-03} gives
    \begin{equation}\label{Eq-lem2.1-07}
        \| \psi^{\geq 1} \|_{H^{t+3}} \lesssim \sum_{p=1}^{n_p}\Bigl\|\mathcal{V}_p\bigl(\psi-\beta_{p,\bm 0}\Phi_{p,\bm 0}\bigr)\Bigr\|_{H^{t+1}}+\|R(\lambda)\|_{H^{t+1}}.
    \end{equation}
    The residual \eqref{Eq-lem2.1-04} is bounded using \eqref{Eq-lem2.1-05} and $H^{s+2}\times H^{t+1}\hookrightarrow H^{t+1}$,
    \begin{equation}\label{Eq-lem2.1-08}
        \|R(\lambda)\|_{H^{t+1}} \lesssim \bigl(\lambda+\|V_0\|_{H^{s+2}}\bigr)\|\psi\|_{H^{t+1}} +\sum_{p=1}^{n_p}|\beta_{p,\bm 0}|\|R_{p,\bm 0}\|_{H^{t+1}}\lesssim \|\psi\|_{L^2}.
    \end{equation}
    To bound the first term, we localize the problem around each singular point. Define the periodic neighborhood of $\bm x_p$ by
    \begin{equation}\label{Eq-lem2.1-09}
    B_{p}:=\bigl\{\bm x \in \Omega :|\bm x-\bm x_p+\bm n|<r_m/2 \text{ for some } \bm n\in \mathbb{Z}^d\bigr\},
    \end{equation}
    where $r_m = \min_{p\neq q}|\bm x_p - \bm x_q|$ is the minimal distance between distinct singular points. Taking $\delta = r_m/2$ in Lemma~\ref{Lem6.4} , we obtain
    \[
        V_p = V_p^{\mathrm{L}}+V^{\mathrm{G}}_p,\quad 
        \Phi_{p,\bm 0} = \Phi_{p,\bm 0}^{\mathrm{L}}+\Phi^{\mathrm{G}}_{p,\bm 0},
    \]
    where $V^{\mathrm{L}}_p$, $\Phi^{\mathrm{L}}_{p,\bm 0}$ are supported in $B_p$, and $V_p^{\mathrm{G}}\in H^{s+2}$, $\Phi_{p,\bm 0}^{\mathrm{G}}\in H^{s+4}$. Note that $B_p\cap B_q = \emptyset, \forall\, p\neq q$, we have  
    \[
        \psi(\bm x)-\beta_{p,\bm 0}\Phi_{p,\bm 0}(\bm x) = \psi^{\geq 1}(\bm x) +\sum_{q\neq p} \beta_{q,\bm 0}\Phi_{q,\bm 0}^{\mathrm{G}}(\bm x), \quad \forall\, \bm x\in B_p.
    \]
    Note that $\mathcal{D}_{p}^{\bm 0}[\psi - \beta_{p,\bm 0}\Phi_{p,\bm 0}]=0$, we apply Lemma~\ref{Lem6.2} with $r=1$ for $V^{\mathrm{L}}_p$ term 
    \begin{equation}\label{Eq-lem2.1-10}
    \begin{aligned}
        \Bnm{\mathcal{V}_p^{\mathrm{L}}(\psi - \beta_{p,\bm 0}\Phi_{p,\bm 0})}_{H^{t+1}} 
        &=\Bnm{\mathcal{V}_p^{\mathrm{L}}(\psi^{\ge 1} +\sum_{q\neq p} \beta_{q,\bm 0}\Phi^{G}_{q,\bm 0})}_{H^{t+1}}\\
        &\lesssim \|V^{\mathrm{L}}_p\|_{PSH^{s,1}}\bnm{\psi^{\ge 1} +\sum_{q\neq p} \beta_{q,\bm 0}\Phi^{G}_{q,\bm 0}}_{H^{t+1+l}}\\
        &\lesssim 
        \|\psi^{\ge 1}\|_{H^{t+1+l}} + \|\psi\|_{L^2}.
    \end{aligned}
    \end{equation}
    The $V^{G}_p$ term is bounded using Lemma~\ref{Lem5.2} with $s' = s+1$, $t'=t+1$ and $l' = \max\{0,l-1\}$
    \begin{equation}\label{Eq-lem2.1-11}
        \Bigl\|\mathcal{V}_p^{G}\bigl(\psi-\beta_{p,\bm 0}\Phi_{p,\bm 0}\bigr)\Bigr\|_{H^{t+1}}
        \lesssim \bnm{V_{p}^{\mathrm{G}}}_{H^{s+1}}\bnm{\psi-\beta_{p,\bm 0}\Phi_{p,0}}_{H^{t+1+l'}}\lesssim \|\psi\|_{L^2},
    \end{equation}
    where the last inequality follows from $t+1+l' = t+\max\{1,l\}\leq s+2$. Combining \eqref{Eq-lem2.1-10} and \eqref{Eq-lem2.1-11} gives
   \begin{equation}\label{Eq-lem2.1-12}  
        \Bigl\|\mathcal{V}_p\bigl(\psi-\beta_{p,\bm 0}\Phi_{p,\bm 0}\bigr)\Bigr\|_{H^{t+1}}
        \lesssim\|\psi\|_{L^2}+\|\psi^{\geq 1}\|_{H^{t+l+1}}.
    \end{equation}
    Substituting \eqref{Eq-lem2.1-08} and \eqref{Eq-lem2.1-12} into \eqref{Eq-lem2.1-07} yields
    \begin{equation}
        \|\psi^{\geq 1}\|_{H^{t+3}} \lesssim \|\psi^{\geq 1}\|_{H^{t+l+1}} + \|\psi\|_{L^2}.
    \end{equation}
    This gives the bootstrapping implication 
    $
        \psi^{\geq 1} \in H^{t+1+l} \Longrightarrow \psi^{\geq 1} \in H^{t+3}.
    $
    Starting from $\psi^{\geq 1}\in H^{s+2}$, a finite iteration of this argument yields
    \begin{equation}
        \|\psi^{\geq 1}\|_{H^{s+3}} \lesssim  \|\psi^{\geq 1}\|_{H^{s+2}}+\|\psi\|_{L^2}\lesssim
        \|\psi\|_{L^2}.
    \end{equation}
   
   \textbf{Step 2.} $r=2\Rightarrow \psi^{\geq 2}\in H^{s+3}$. 
    The proof follows the same strategy as in \textbf{Step 1}; we therefore only outline the main steps below and defer the details to Appendix~\ref{proofthmAR1},
    \[
    \psi^{\ge 1} \in H^{s+3} 
    \overset{\mathrm{embedding}}{\xLongrightarrow{\qquad\quad}}
    \beta_{p,\bm \alpha}<+\infty \Longrightarrow\psi^{\ge 2}\in H^{s+3} 
    \overset{\mathrm{bootstrap}}{\xLongrightarrow{\qquad\quad}}
    \psi^{\ge 2} \in H^{s+4}.
    \]
\end{proof}

\subsection{Proof of \texorpdfstring{Theorem~\ref{Thm03}}{Theorem Y}} 
\label{subsec:AR4}

\begin{proof}[Proof of Theorem~\ref{Thm03}]
    Applying the triangle inequality
    \begin{equation}\label{Eq-thm03-03}
         \|\mathcal{A}_N \psi^N-\psi\|_{H^{1}}\lesssim  
         \big\|\mathcal{A}_N \big(\psi^N-\mathcal{P}_N\psi\big)\big\|_{H^{1}}+ \|\mathcal{A}_N \mathcal{P}_N\psi-\psi\|_{H^{1}}.
    \end{equation}
    The proof process in three steps.
    
    \textbf{Step 1. Stability of the linear system.}
   In the AR post-processing, one needs to solve a linear system to recover the coefficients $\beta_{p,\bm \alpha}$ in \eqref{Eq-3.0-02} from the given $\psi^N$ and $\Phi_{p,\bm \alpha}$. It is therefore necessary to analyze the stability of this system.

    For each singular point $p$ and multi-index $|\bm \alpha|\leq 1$, define the scaled quantities
    \begin{equation}\label{eq-thm2-SCALE}
    \hat{d}_{p,\bm \alpha}
    :=
    N^{-|\bm \alpha|}
    \mathcal{D}^{\bm \alpha}_{p}[\psi^N],
    \qquad
    \hat{\beta}_{p,\bm \alpha}
    :=
    N^{-|\bm \alpha|}
    \beta_{p,\bm \alpha}.
    \end{equation}
    Under this notation, the linear system \eqref{Eq-3.0-02} becomes
    \[
    \hat\Phi \bm{\hat{\beta}}
    =
    \bm{\hat{d}},
    \]
    where $\bm{\hat{\beta}},\bm{\hat{d}}$ are vectors with entries $\hat{\beta}_{p,\bm \alpha}$ and $\hat{d}_{p,\bm \alpha}$, respectively. The coefficient matrix $\Phi$ has entries (where we write $\mathcal{D}_p^{\bm \alpha }:= D_{\bm x_p}^{\bm \alpha}$)
    \begin{equation}
        \hat\Phi[p,\bm \alpha';q,\bm \alpha]:= \left\{
        \begin{aligned}
           &-N^{|\bm \alpha|-|\bm \alpha'|}\mathcal{D}^{\bm \alpha'}_p\bigl[\mathcal{P}_N^\perp\Phi_{q,\bm \alpha}\bigr],&& p\neq q,\\
           &N^{|\bm \alpha|-|\bm \alpha'|}\mathcal{D}^{\bm \alpha'}_p\bigl[\mathcal{P}_N\Phi_{p,\bm \alpha}\bigr],&& p= q.\\
        \end{aligned}
        \right.
    \end{equation}
    Our goal is to show that the matrix $\hat\Phi$ is diagonally dominant for all sufficiently large $N$. Specifically, its off‑diagonal entries are $o(1)$ while the diagonal entries are $1+o(1)$ as $N\to \infty$. To this end, We first collect three estimates that will be used repeatedly.
    Remind that by \eqref{Eq-6.3-01}, we have $\Phi_{q,\bm \alpha}\in PSH^{s+2+|\bm \alpha|,2-|\bm \alpha|}$.
    
    \textbf{Estimate 1.}($|\bm \alpha'|\leq|\bm \alpha|$, arbitrary $p,q$).
    Since the index $s+2+|\bm \alpha| > d/2 + |\bm \alpha'|$, applying  Lemma~\ref{Lem5.5} (ii) yields
    \begin{equation}
       N^{|\bm \alpha|-|\bm \alpha'|}\Bigl| \mathcal{D}^{\bm \alpha'}_p\big[\mathcal{P}_N^\perp\Phi_{q,\bm \alpha}\big]\Bigr|\lesssim N^{d/2-s-2}\|\Phi_{q,\bm \alpha}\|_{H^{s+2+|\bm \alpha|}}=o(1);
    \end{equation}

    \textbf{Estimate 2.}($|\bm \alpha'| > |\bm \alpha|$, arbitrary $p,q$). Applying Lemma~\ref{Lem5.5}  (iii)  yields
    \begin{equation}\label{Eq-thm03-07}
    N^{|\bm \alpha|-|\bm \alpha'|}\Bigl|\mathcal{D}^{\bm \alpha'}_p[\mathcal{P}_N\Phi_{q,\bm \alpha}]\Bigr|\lesssim N^{\max\{-(|\bm \alpha'|-|\bm \alpha|)^{\m},d/2-s-2\}} \|\Phi_{q,\bm \alpha}\|_{H^{s+2+|\bm \alpha|}} =o(1);
    \end{equation}

    \textbf{Estimate 3.}($|\bm \alpha'| > |\bm \alpha|$, $p\neq q$). 
    By Lemma~\ref{Lem6.4}, we have the decomposition $\Phi_{q,\bm \alpha}=  \Phi_{q,\bm \alpha}^{\mathrm{L}}+ \Phi_{q,\bm \alpha}^{\mathrm{G}}$, where 
    $\Phi_{q,\bm \alpha}^{\mathrm{L}}$ is supported in $B_q$ (see \eqref{Eq-lem2.1-09}), and  $\Phi_{q,\bm \alpha}^{\mathrm{G}}\in H^{s+4}$. By $H^{s+4}\hookrightarrow W^{|\bm \alpha'|,\infty}$, we have 
    \begin{equation}\label{Eq-thm03-08}
    N^{|\bm \alpha|-|\bm \alpha'|}\Bigl|\mathcal{D}^{\bm \alpha'}_p[\Phi_{q,\bm \alpha}]\Bigr|\lesssim N^{|\bm \alpha|-|\bm \alpha'|}\|\Phi^{\mathrm{G}}_{q,\bm \alpha}\|_{H^{s+4}}= o(1).
    \end{equation}

    With the above estimates at our disposal, we now examine the four possible index configurations.
    
    \textbf{Case 1.} $|\bm \alpha'| \le |\bm \alpha|$ and $p\neq q$. By  \textbf{Estimate 1}, we have
    \[
        \hat\Phi[p,\bm \alpha';q,\bm \alpha]=  -N^{|\bm \alpha|-|\bm \alpha'|}\mathcal{D}^{\bm \alpha'}_p[\mathcal{P}^{\perp}_N\Phi_{q,\bm \alpha}] = o(1);
    \]
    
    \textbf{Case 2.} $|\bm \alpha'| > |\bm \alpha|$ and $p\neq q$. Combining  \textbf{Estimates 2} and \textbf{3} yields
    \[
        \hat\Phi[p,\bm \alpha';q,\bm \alpha]\leq  N^{|\bm \alpha|-|\bm \alpha'|}\left(\Bigl|\mathcal{D}^{\bm \alpha'}_p[\mathcal{P}_N\Phi_{q,\bm \alpha}]\Bigr|+\Bigl|\mathcal{D}^{\bm \alpha'}_p[\Phi_{q,\bm \alpha}]\Bigr|\right) = o(1);
    \]
    
    \textbf{Case 3.} $|\bm \alpha'| > |\bm \alpha|$ and $p = q$. From  \textbf{Estimate 2}, we obtain
    \[
        \hat\Phi[p,\bm \alpha';q,\bm \alpha] =  N^{|\bm \alpha|-|\bm \alpha'|}\mathcal{D}^{\bm \alpha'}_p[\mathcal{P}_N\Phi_{q,\bm \alpha}] = o(1);
    \]
    
    \textbf{Case 4.} $|\bm \alpha'| \leq |\bm \alpha|$ and $p = q$. 
    Using $\mathcal{D}^{\bm \alpha'}_{\bm p}[\Phi_{p,\bm \alpha}]=\delta_{\bm \alpha',\bm \alpha}$ and  \textbf{Estimate 1} gives
    \[
        \hat\Phi[p,\bm \alpha';q,\bm \alpha] =  N^{|\bm \alpha|-|\bm \alpha'|}\left(\delta_{\bm \alpha',\bm \alpha}-\mathcal{D}^{\bm \alpha'}_p[\mathcal{P}^\perp_N\Phi_{q,\bm \alpha}]\right) =\delta_{\bm \alpha',\bm \alpha} + o(1).
    \]

    In summary, we have estimated that
    \[
        \hat\Phi[p,\bm \alpha';q,\bm \alpha] = \delta_{p,q}\delta_{\bm \alpha,\bm \alpha'}+o(1).
    \]
    Consequently, there exists $N_0$ such that for all $N\geq N_0$,  the solution of $\Phi \bm{\hat{x}} =\bm{\hat{b}} $ satisfies the uniform bound $\| \bm{\hat{x}}\|_{\infty} \leq 2\bigl\| \bm{\hat{b}}\|_{\infty}$ in vector norm. With this stability result in hand, we now return to estimate the two terms in \eqref{Eq-thm03-03}.
    
    \textbf{Step 2. Stability of AR operator.} Let $\epsilon^N:=\psi^N-\mathcal{P}_N\psi$. We shall show that $\mathcal{A}_N$ is bounded uniformly in $N$, i.e. for any $\epsilon^N\in \mathscr{X}_N$ 
    \begin{equation}\label{Eq-thm03-09}
        \big\|\mathcal{A}_N \epsilon^N\big\|_{H^{1}}\lesssim \big\| \epsilon^N\big\|_{H^{1}}.
    \end{equation}
    From the definition \eqref{Eq-3.0-01} of $\mathcal{A}_N$, 
    \begin{equation}\label{Eq-thm03-10}
        \mathcal{A}_N \epsilon^N  = \epsilon^N + \mathcal{P}_N^\perp\sum_{p=1}^{n_p}\sum_{|\alpha|\leq 1}\epsilon^N_{p,\bm \alpha}\Phi_{p,\bm \alpha},
    \end{equation}
    where the coefficients $\epsilon_{p,\bm \alpha}^N$ are determined by the scaled linear system
    \[
    \hat{\Phi} \bm{\hat{x}}= \bm{\hat{b}},\quad \bm{\hat{x}}[p,\bm \alpha]=N^{-|\bm \alpha|}\epsilon_{p,\bm\alpha},\quad \bm{\hat{b}}[p,\bm \alpha]=N^{-|\bm \alpha|}\mathcal{D}_{p}^{\bm \alpha}[\epsilon^N].
    \]
    Owing to the stability of scaled linear system, for any for all sufficiently large $N$, we have $\| \bm{\hat{x}}\|_{\infty} \leq 2\bigl\| \bm{\hat{b}}\|_{\infty}$. Hence  
    \begin{equation}\label{Eq-thm03-11}
        \epsilon_{p,\bm \alpha}^N\leq 2N^{|\bm \alpha|}\max_{q,\bm \alpha'} N^{-|\bm \alpha'|}\mathcal{D}_q^{\bm \alpha'}[\epsilon^N].
    \end{equation}
    Now $\epsilon^N = \mathcal{P}_N\epsilon^N$. Choose $m = d/2-s-1$, applying  Lemma~\ref{Lem5.5} (iii) gives
    \begin{equation}\label{Eq-thm03-12}
        N^{-|\bm \alpha'|}\max_{q} \mathcal{D}_q^{\bm \alpha'}[\epsilon^N] \leq N^{-|\bm \alpha'|}\bigl\|\epsilon^N\bigr\|_{W^{|\bm \alpha'|,\infty}}\lesssim N^{s+1}\bigl\|\epsilon^N\bigr\|_{H^m}.
    \end{equation}
    Since $m < 1$ we have $\bigl\|\epsilon^N\bigr\|_{H^m}\le\bigl\|\epsilon^N\bigr\|_{H^1}$. Substituting \eqref{Eq-thm03-12} into \eqref{Eq-thm03-11} yields
    \begin{equation}\label{Eq-thm03-13}
        |\epsilon_{p,\bm \alpha}| \lesssim  N^{|\bm \alpha|+s+1}\bigl\|\epsilon^N\bigr\|_{H^m}\leq N^{|\bm \alpha|+s+1}\bigl\|\epsilon^N\bigr\|_{H^1}.
    \end{equation}
    On the other hand, applying Lemma~\ref{Lem5.5} (ii),
    \begin{equation}\label{Eq-thm03-14}
        \bigl\|\mathcal{P}^\perp_N \Phi_{p,\bm \alpha}\bigr\|_{H^1} \lesssim N^{-s-1-|\bm \alpha|}\bigl\| \Phi_{p,\bm \alpha}\bigr\|_{H^{s+2+|\bm \alpha|}}.
    \end{equation}
    Inserting \eqref{Eq-thm03-13} and \eqref{Eq-thm03-14} into \eqref{Eq-thm03-10}, we have 
    \[
    \begin{aligned}
        \bigl\|\mathcal{A}_N\epsilon^N\bigr\|_{H^1} 
        &\leq \bigl\|\epsilon^N\bigr\|_{H^1} + \sum_{p=1}^{n_p}\sum_{|\alpha|\leq  1}|\epsilon^N_{p,\bm \alpha}|\bigl\|\mathcal{P}_N^\perp\Phi_{p,\bm \alpha}\bigr\|_{H^1}\\
        &\lesssim \bigl\|\epsilon^N\bigr\|_{H^1}\left(1 + \sum_{p=1}^{n_p}\sum_{|\alpha|\leq 1} \bigl\|\Phi_{p,\bm \alpha}\bigr\|_{H^{s+2+|\bm \alpha|}}\right).
    \end{aligned}
    \]
    This proves the uniform bound \eqref{Eq-thm03-09}.
    
    \textbf{Step 3. Approximation estimate.} We now prove that for the exact eigenfunction $\psi$, the approximation error of AR operator satisfies 
    \begin{equation}\label{Eq-thm03-15}
        \|\mathcal{A}_N \mathcal{P}_N\psi-\psi\|_{H^{1}}\lesssim N^{-(s+3)}\|\psi\|_{L^2}.
    \end{equation}
    From the regularity result in Lemma~\ref{Lem2.1}, we have 
    \begin{equation}\label{Eq-thm03-16}
    \psi = \psi^{\geq 2}+ \sum_{p=1}^{n_p}\sum_{|\bm \alpha|\leq 1}\beta_{p,\bm \alpha}\Phi_{p,\bm \alpha},    
    \qquad
    \| \psi^{\geq 2}\|_{H^{s+4}} \lesssim \|\psi\|_{L^2},
    \end{equation}
    where the coefficients $\beta_{p,\bm \alpha}$ satisfy
    \begin{equation}\label{Eq-thm03-17}
        \mathcal{D}_p^{\bm \alpha'}\bigl[\psi - \sum_{|\bm \alpha'|\leq 1} \beta_{q,\bm \alpha}\Phi_{p,\bm \alpha}\bigr] = 0.
    \end{equation}
    By construction \eqref{Eq-3.0-01} of $\mathcal{A}_N$, we have
    \begin{equation}\label{Eq-thm03-18}
        \mathcal{A}_N \mathcal{P}_N\psi  = \mathcal{P}_N\psi + \mathcal{P}_N^\perp\sum_{p=1}^{n_p}\sum_{|\alpha|\leq1}\beta^N_{p,\bm \alpha}\Phi_{p,\bm \alpha},
    \end{equation}
    where the coefficients $\beta_{p,\bm \alpha}^N$ satisfy
    \begin{equation}\label{Eq-thm03-19}
        \mathcal{D}_p^{\bm \alpha'}\bigl[\mathcal{A}_N\mathcal{P}_N\psi - \sum_{|\bm \alpha'|\leq 1} \beta_{q,\bm \alpha}\Phi_{p,\bm \alpha}\bigr] = 0.
    \end{equation}
    Subtracting \eqref{Eq-thm03-16} from \eqref{Eq-thm03-18} gives
    \begin{equation}\label{Eq-thm03-20}
        \mathcal{A}_N \mathcal{P}_N\psi-\psi = \mathcal{P}_N^\perp\psi^{\geq 2}+\sum_{p=1}^{n_p}\sum_{|\bm \alpha|\leq 1}\epsilon^N_{p,\bm \alpha}\mathcal{P}_N^\perp\Phi_{p,\bm \alpha},
    \end{equation}
    where $\epsilon_{p,\bm \alpha}^N = \beta_{p,\bm \alpha}^N -\beta_{p,\bm \alpha}$. 
    Subtracting \eqref{Eq-thm03-17} from \eqref{Eq-thm03-19} gives 
    \begin{equation}\label{Eq-thm03-21}
        {D}^{\bm \alpha'}_{p}\Bigl[\mathcal{P}_N^\perp\psi^{\geq 2}-\sum_{|\bm \alpha|\le 1} \epsilon^N_{p,\bm \alpha}\mathcal{P}_N\Phi_{p,\bm\alpha}+\sum_{q\neq p}\sum_{|\bm \alpha|\le 1}\epsilon^N_{q,\bm \alpha}\mathcal{P}_N^\perp\Phi_{q,\bm\alpha}\Bigr] = 0,
    \end{equation}
    which is a linear system for the vector $\hat{x}$ with entries $N^{-|\bm \alpha|}\epsilon^N_{p,\bm \alpha}$, the right-hand side vector $\hat{b}$ with entries $N^{-|\bm \alpha'|}\mathcal{D}^{\bm \alpha'}_p[\mathcal{P}_N^\perp \psi^{\geq 2}]$, and the coefficient matrix $\hat{\Phi}$. Due to the stability of this linear system, we have for sufficiently large $N$ that
    \begin{equation}
        |\epsilon_{p,\bm \alpha}^N|\lesssim N^{|\bm \alpha|}\max_{q,\bm \alpha'} N^{-|\bm \alpha'|}\mathcal{D}_q^{\bm \alpha'}[\mathcal{P}^\perp_N\psi^{\geq 2}].
    \end{equation}
    Since $s+4>d/2+|\bm \alpha'|$,  we obtain by  Lemma~\ref{Lem5.5} (ii),
    \begin{equation}\label{Eq-thm03-23}
        |\epsilon_{p,\bm \alpha}^N|
        \lesssim N^{|\bm \alpha|-|\bm \alpha'|}\bigl\|\mathcal{P}^\perp_N\psi^{\geq 2}\bigr\|_{W^{|\bm \alpha'|,\infty}}
        \lesssim N^{|\bm \alpha|+d/2}N^{-s-4}\bigl\|\psi^{\geq 2}\bigr\|_{H^{s+4}}.
    \end{equation}
    Now we return to the expression \eqref{Eq-thm03-20}. By \eqref{Eq-thm03-23} and Lemma~\ref{Lem5.5} (i), we have 
    \begin{equation}
    \begin{aligned}
        \bigl\|\mathcal{A}_N \mathcal{P}_N\psi-\psi\bigr\|_{H^1} 
        &\leq \bigl\|\mathcal{P}_N^\perp\psi^{\geq 2}\bigr\|_{H^1} + \sum_{p=1}^{n_p}\sum_{|\bm \alpha|\leq 1}|\epsilon^N_{p,\bm \alpha}| \bigl\|\mathcal{P}_N^\perp\Phi_{p,\bm \alpha}\bigr\|_{H^1}\\
        &\lesssim N^{-s-3}\bigl\|\psi^{\geq 2}\bigr\|_{H^{s+4}}\\
        &\quad \times\Bigl(1 + N^{d/2-s-2}\sum_{p=1}^{n_p}\sum_{|\alpha|\leq 1} \bigl\|\Phi_{p,\bm \alpha}\bigr\|_{H^{s+2+|\bm \alpha|}}\Bigr).
    \end{aligned}
    \end{equation}
    Since $s>d/2-2$, the sum over $p,\bm \alpha$ is bounded by a finite constant independent of $N$. Combining this estimate with the regularity bound \eqref{Eq-thm03-16} gives precisely \eqref{Eq-thm03-15}.

    Finally, applying the triangle inequality \eqref{Eq-thm03-03} together with the stability estimate \eqref{Eq-thm03-09} and the approximation estimate \eqref{Eq-thm03-15} yields the desired convergence bound \eqref{Eq-thm03-01}. The proof is thus completed.
\end{proof}

\section{Conclusions and discussions}
\label{sec:conclusions}
In this paper we studied FSM for Schr\"{o}dinger eigenvalue problems with potentials containing isolated point singularities. We derived an asymptotic expansions of eigenfunctions that separate a regular component from a finite number of explicitly characterized asymptotic functions associated with each singular point.
Based on this expansion, we proposed an asymptotic-recovery (AR) post-processing technique that recover the  high-frequency components from the truncated Fourier coefficients. The resulting AR-FSM significantly improves the convergence behavior of the standard Fourier spectral method. 

Rigorous error estimates are established for FSM and AR-FSM for singular potentials $V \in H^{s}$ with $s > \max\{d/2-2,-1\}$, where $d$ denotes the spatial dimension. We first prove the optimal convergence orders of  FSM, showing that it achieves order $2s+2$ for eigenvalues, order $s+1$ for eigenfunctions in the $H^1$ norm together with a super-convergence gain of order $b=\min\{(s+2-d/2)^{\m},s+1,2\}$. Building on this result, we show that the AR-FSM fully exploits this super-convergence, yielding convergence rates of order $2s+2+2b$ for eigenvalues and $s+1+b$ for eigenfunctions. Numerical experiments confirm the theoretical predictions and demonstrate that the proposed methods substantially improve the convergence rates of FSM while retaining its computational efficiency.

In the analysis, we introduce a rigorous definition of point singularities and develops a foundational framework for their study. It further establishes an asymptotic expansion of eigenfunctions consisting of a
regular component in $H^{s+4}$ together with $d+1$ asymptotic functions associated with each singular point. These contributions are of independent interest beyond the AR-FSM framework.

We do not further investigate the dependence of the error estimates on the eigenvalue index $n$. Intuitively, one should replace the regularity results used throughout the analysis by estimates of the form
$\|\psi_n\|_{H^{t}}\lesssim (\lambda_n+M)^{t/2}\|\psi_n\|_{L^2}$. 
Such estimates are of interest for two reasons. First, they provide quantitative criteria for determining how many computed eigenpairs are reliable at a given truncation parameter $N$. Second, they constitute a necessary step toward the analysis of time-dependent Schr\"odinger equation, where the solution typically involves contributions from infinitely many eigenpairs.

An interesting observation from the convergence results is that the lower bound of convergence order improves as $d$ increases. Indeed, the eigenvalue convergence order of FSM is $2s+2 \ge d-2$, since $s \ge d/2-2$. On the other hand, the computational complexity of FSM scales as $M = N^d \log N$. Consequently, the eigenvalue error measured with respect to the computational complexity scales like $ M^{-1+2/d}$,
which approaches the rate $M^{-1}$ as $d\to\infty$.
This observation raises two questions related to the curse of dimensionality: What is the minimum computational complexity $M$ required to observe the asymptotic $M^{-1}$ convergence regime, and how can the tensor-product grid be avoided in order to reduce minimum computational cost.

One direction of advance is to incorporate the asymptotic expansions Lemma~\ref{Lem2.1} directly into the discretization, leading to an enhanced Fourier spectral methods that further improve the convergence order. 
Such developments may provide a systematic framework for constructing high-accuracy spectral algorithms for Schr\"odinger equation with point-singular potential. 

Another direction is to extend the analysis of point singularities to broader classes of partial differential equations. Such extensions could provide a unified framework for treating singularities in a wide range of applications



\bmhead{Acknowledgements}
This research was supported by the National Natural Science Foundation of China (Nos. 12325112, 12288101) and  the High-performance Computing Platform of Peking University.

\begin{appendices}

\section{Proof of Lemma~\ref{Lem5.5}} \label{proofLem5.5}
\begin{proof} 
    Write $\psi$ in its Fourier series
    \begin{equation}
        \phi = \sum_{\bm k \in \mathbb{Z}^d} \hat{\phi}_{\bm k} e_{\bm k},\quad \hat{\phi}_{\bm k} := \langle e_{\bm k},\phi \rangle.
    \end{equation}

    \noindent\textbf{(i)} For $t<m$,
    \[
    \begin{aligned}
        \|\mathcal{P}_N^\perp \phi\|_{H^t}^2 
        &= \sum_{\|\bm k\|_{\infty}\ge N} \bigl(4\pi^2|\bm k|^2+1\bigr)^{t}|\hat{\phi}_{\bm k}|^2\\
        &\leq (2\pi N)^{2t-2m}\sum_{\|\bm k\|_{\infty}\ge N} \bigl(4\pi^2|\bm k|^2+1\bigr)^{m}|\hat{\phi}_{\bm k}|^2\leq\|\phi\|_{H^m}^2.
    \end{aligned}
    \]

    \noindent\textbf{(ii)}  $t\in \mathbb{N}$ and $t+d/2<m$.  Using the Cauchy–Schwarz inequality,
    \[
    \begin{aligned}
        \|\mathcal{P}_N^\perp \phi\|_{W^{t,\infty}}^2 
        &\lesssim \left(\sum_{\|\bm k\|_{\infty}\ge N} \big(4\pi|\bm k|^2+1\big)^{t/2}|\hat{\phi}_{\bm k}|\right)^2\\
        &\lesssim \sum_{\|\bm k\|_{\infty}\ge N}\bigl(4\pi^2|\bm k|^2+1\bigr)^{t-m}\sum_{\|\bm k\|_{\infty}\ge N} \bigl(4\pi^2|\bm k|^2+1\bigr)^{m}|\hat{\phi}_{\bm k}|^2 .
    \end{aligned}
    \]
    Because $2t-2m+d<0$, the first sum is bounded by $CN^{2t-2m+d}$. Therefore
    \[
        \|\mathcal{P}_N^\perp \phi\|_{W^{t,\infty}}\lesssim N^{t-m+d/2}\|\phi\|_{H^m}.
    \]

    \noindent\textbf{(iii)}  $t\in \mathbb{N}$ and $t+d/2>m$.  The same argument applied to the low‑frequency part gives
    \[
    \begin{aligned}
        \|\mathcal{P}_N \phi\|_{W^{t,\infty}} \lesssim \Bigl(\sum_{\|\bm k\|_{\infty}\le N}\bigl(4\pi^2|\bm k|^2+1\bigr)^{t-m}\Bigr)^{1/2}\|\phi\|_{H^m}.
    \end{aligned}
    \]
    For $2t-2m+d>0$, the first sum is bounded by $CN^{2t-2m+d}$. For $2t-2m+d\leq 0$, the first sum is bounded by $N^{\epsilon}$ for any $\epsilon>0$. Therefore
    \[
        \|\mathcal{P}_N \phi\|_{W^{t,\infty}}\lesssim N^{\max\{-0^{\m},t-m+d/2\}}\|\phi\|_{H^m}.
    \]
    This completes the proof of Lemma~\ref{Lem5.5}.
\end{proof}

\section{Proof of Lemma~\ref{Lem5.6}}\label{proofLem5.6}
\begin{proof}
 (i) By the Lax-Milgram theorem, it suffices to establish the coercivity of the operator $\mathcal{H}-\lambda^\star$ on $\mathscr{X}_N^\perp$. 
By Lemma~\ref{Lem5.3}, there exists $M>0$ such that
\begin{equation}\label{eq-Lem5.6-01}
0.5 \|\psi^\perp\|_{H^1}^2
\leq 
\langle \psi^\perp, (\mathcal{H}+M)\psi^\perp\rangle
=
\langle \psi^\perp, (\mathcal{H}-\lambda^\star)\psi^\perp\rangle 
+ (\lambda^\star+M)\|\psi^\perp\|_{L^2}^2.
\end{equation}
Since $\psi^\perp = \mathcal{P}_N^\perp \psi^\perp$, Lemma~\ref{Lem5.5} implies
\begin{equation}\label{eq-Lem5.6-02}
(\lambda^\star+M)\|\psi^\perp\|_{L^2}^2 
\leq 
\tfrac14 \bigl(1 + M(\pi N)^{-2}\bigr)\|\psi^\perp\|_{H^1}^2.
\end{equation}
Substituting \eqref{eq-Lem5.6-02} into \eqref{eq-Lem5.6-01}, we obtain
\begin{equation}
\langle \psi^\perp, (\mathcal{H}-\lambda^\star)\psi^\perp\rangle
\geq  \tfrac14 \bigl(1 - M(\pi N)^{-2}\bigr)\|\psi^\perp\|_{H^1}^2.
\end{equation}
This proves the coercivity of $\mathcal{H}-\lambda^\star$ on $\mathscr{X}_N^\perp$ for $N\geq N_0:= \sqrt{M}$.

    (ii) Note that the derivative of $\mathcal{H} - \mathcal{F}_N(\lambda^\star)$ is given by
    \[
    \frac{\mathrm{d}}{\mathrm{d} \lambda^\star} \bigl[\mathcal{H} - \mathcal{F}_N(\lambda^\star)\bigr] = -\mathcal{V} \left[\frac{\mathrm{d}}{\mathrm{d} \lambda^\star} \mathcal{G}_N(\lambda^\star)\right] \mathcal{V} = -\mathcal{V} \left[\mathcal{G}_N(\lambda^\star)\right]^2 \mathcal{V},
    \]
    which is bounded and non-positive. Hence, the eigenvalue $\sigma_m^N(\lambda^\star)$ is continuous and non-increasing.

    (iii) Set $\lambda^\star = \lambda_n\leq \pi^2N^2$ and let $\alpha_n=\operatorname{dim}\operatorname{ker}(\lambda-\mathcal{H})\geq 1$ denote the multiplicity of $\lambda_n$ for $\mathcal{H}$. Since $\mathcal{G}_N(\lambda^\star)$ is uni-solvent,  the construction yields a one‑to‑one correspondence between the eigenfunctions of the original problem \eqref{Eq-1.1-02} and those of the reduced problem \eqref{Eq-5.2-02}
    \begin{equation} \label{Eq-lem5.6-04}
    \mathcal{H}\psi = \lambda^\star\psi \Longleftrightarrow \mathcal{P}_N\bigl(\mathcal{H}-\mathcal{F}_N(\lambda^\star)\bigr)\psi^N = \lambda^\star\psi^N,
    \end{equation}
    where $\psi^N = \mathcal{P}_N \psi$ and $\psi  = \bigl(1- \mathcal{G}_N(\lambda^\star) \mathcal{V}\bigr)\psi^N$.  \eqref{Eq-lem5.6-04} shows that the reduced problem \eqref{Eq-5.2-03} also possesses $\lambda_n$  as an eigenvalue with the same multiplicity $\alpha_n$. Thus, it only remains to prove that $\lambda_n$ is exactly the $n$-th eigenvalue $\mathcal{H}-\mathcal{F}_N(\lambda)$, or equivalently,  \begin{equation}\label{Eq-lem5.6-05}
    \bigl|\{m\in \mathbb{N}^+:\lambda_m<\lambda_n\}\bigr| = 
    \bigl|\{m\in \mathbb{N}^+:\sigma_m^N(\lambda_n)<\lambda_n\}\bigr|.
    \end{equation}
    For each fixed $m \in \mathbb{N}^+$, consider the fixed-point equation $x-\sigma_m^N(x)=0$ on $(-\infty,\lambda_n)$. By (ii),  the function $x-\sigma_m^N(x)$ is continuous and strictly increasing in $x$, with $x-\sigma_m^N(x)\leq x-\sigma^N_m(\lambda_n) \to -\infty $ as $x \to -\infty$. Consequently, there exists a unique $x\in (-\infty,\lambda_n)$ satisfying $x = \sigma_m^N(x) $ if and only if $\sigma_m^N(\lambda_n)<\lambda_n$. Hence,
    \begin{equation}
        \bigl|\{ m\in \mathbb N^{+} :\ \exists\, x\in(-\infty,\lambda_n),\ \sigma_m^N(x)=x\}\bigr| =
        \bigl|\{ m\in \mathbb N^{+} :\ \sigma_m^N(\lambda_n)<\lambda_n\}\bigr|.
    \end{equation}
    On the other hand, the correspondence in \eqref{Eq-lem5.6-04} identifies each $x$ such that $\sigma_m^N(x) = x$ with an eigenvalue $\lambda = x$ of $\mathcal{H}$ of the same multiplicity. Looping through every $x \in (-\infty, \lambda_n)$ gives
    \begin{equation}
       \bigl|\{ m\in \mathbb N^{+} :\ \exists\, x\in(-\infty,\lambda_n),\ \sigma_m^N(x)=x\}\bigr| = \bigl|\{m\in \mathbb{N}^+:\lambda_m<\lambda_n\}\bigr|.
    \end{equation}
    Combining the last two equalities yields precisely \eqref{Eq-lem5.6-05}.
\end{proof}

\section{Proof of Proposition \ref{Prop04}}\label{proofProp04}
\begin{proof}
    By the min–max theorem,
    \begin{equation}\label{Eq-prop04-03}
    \begin{aligned}
    \lambda_{1,n}
    &= \min_{\dim W = n} \max_{u \in W} \frac{\langle u, \mathcal{S}_1 u \rangle}{\langle u, u \rangle} \\
    &\leq \max_{u \in W_n(\mathcal{S}_2)} \frac{\langle u, (\mathcal{S}_2 + \mathcal{S}_1 - \mathcal{S}_2) u \rangle}{\langle u, u \rangle} \\
    &= \lambda_{2,n} + \max_{u \in W_n(\mathcal{S}_2)} \frac{\langle u, (\mathcal{S}_1 - \mathcal{S}_2) u \rangle}{\langle u, u \rangle}.
    \end{aligned}
    \end{equation}
    By symmetry between $\mathcal{S}_1$ and $\mathcal{S}_2$, we similarly obtain
    \begin{equation}\label{Eq-prop04-04}
    \lambda_{2,n} \leq \lambda_{1,n} + \max_{u \in W_n(\mathcal{S}_1)} \frac{\langle u, (\mathcal{S}_2 - \mathcal{S}_1) u \rangle}{\langle u, u \rangle}.
    \end{equation}
    Combining \eqref{Eq-prop04-03} and \eqref{Eq-prop04-04} yields the desired \eqref{Eq-prop04-01}.

    We now proceed to prove \eqref{Eq-prop04-02}. Without loss of generality, we set $M=0$; the case $M\neq 0$ reduces to this setting by introducing $\mathcal{{S}}_1'=\mathcal{S}_1+M$ and $\mathcal{{S}}_2'=\mathcal{S}_2+M$.  Let $\Gamma$ be a circle centered at $\lambda_{1,n}$ with radius $\frac{1}{2}\gamma_{1,n}$. Then the projection operator $\mathcal{P}_{n,\mathcal{S}_1}$ can be expressed as
    \begin{equation}\label{Eq-prop04-05}
    \mathcal{P}_{n,\mathcal{S}_1} = \frac{1}{2\pi i} \int_{\Gamma} (z - \mathcal{S}_1)^{-1} \, dz.
    \end{equation}
    Since $\mathcal{S}_{2}u_{2,n} = \lambda_{2,n}u_{2,n}$ and $\lambda_{2,n}$ lies inside circle $\Gamma$, we have
    \begin{equation}\label{Eq-prop04-06}
        u_{2,n} = \frac{1}{2\pi i }\int_{\Gamma}(z-\mathcal{S}_2)^{-1}u_{2,n}\, dz.
    \end{equation}
    Combining \eqref{Eq-prop04-05} and \eqref{Eq-prop04-06}, we have
    \begin{equation}
        \begin{aligned}
            u_{2,n}-\mathcal{P}_{n,\mathcal{S}_1}u_{2,n} 
            &= \frac{1}{2\pi i}\int_{\Gamma}(z-\mathcal{S}_1)^{-1}(\mathcal{S}_{2}-\mathcal{S}_1)(z-\mathcal{S}_2)^{-1}u_{2,n}\, dz\\
            &= \frac{1}{2\pi i}\int_{\Gamma}\frac{dz}{z-\lambda_{2,n}}(z-\mathcal{S}_1)^{-1}(\mathcal{S}_2-\mathcal{S}_1)u_{2,n}.\\
        \end{aligned}
    \end{equation}
    Let $\phi = u_{2,n}-\mathcal{P}_{n,\mathcal{S}_1}u_{2,n},$ we can estimate the left-hand side of \eqref{Eq-prop04-02} as follows:
    \begin{equation}
        \begin{aligned}\label{Eq-prop04-07}
             \|\phi\| _{\mathcal{S}_1}^2
             &\leq \left(\frac{1}{2\pi}\int_{\Gamma} \frac{|dz|}{|z-\lambda_{2,n}|}\right)\sup_{z\in \Gamma}\left|\langle \frac{\mathcal{S}_1}{z-\mathcal{S}_1}\phi,\big(\mathcal{S}_2-\mathcal{S}_1\big)u_{2,n}\rangle\right|\\
             &\leq \left(\frac{1}{2\pi}\int_{\Gamma} \frac{|dz|}{|z-\lambda_{2,n}|}\right)
             \sup_{z\in \Gamma}\left\|\frac{\mathcal{S}_1}{z-\mathcal{S}_1}\phi\right\|_{\mathcal{S}_1}
             \sup_{v\in \mathscr{X}_N\setminus\{0\}}\left|\frac{\langle v,(\mathcal{S}_2-\mathcal{S}_1)u_{2,n}\rangle}{\|v\|_{\mathcal{S}_1}}\right|.
        \end{aligned}
    \end{equation}
    The first factor is bounded by
    \begin{equation}\label{Eq-prop05-09}
        \frac{1}{2\pi}\int_{\Gamma}\frac{|d z|}{|z-\lambda_{2,n}|}
        \leq \frac{1}{2\pi}\sup_{z\in \Gamma}\frac{1}{|z-\lambda_{2,n}|}\int_{\Gamma}|dz| 
        \leq \frac{2\pi \frac{1}{2}\gamma_{n,\mathcal{S}_1}}{2\pi \frac{1}{6}\gamma_{n,\mathcal{S}_1}} = 3.
    \end{equation}
    For the second factor, we have for any $z \in \Gamma$,
    \begin{equation}\label{Eq-prop04-09}
    \begin{aligned}
        \left\|\frac{\mathcal{S}_1}{z-\mathcal{S}_1}\phi\right\|_{\mathcal{S}_1} 
        &\leq  \|\phi\|_{\mathcal{S}_1}\max_{j \geq 1}\frac{\lambda_{1,j}}{|z-\lambda_{1,j}|}\\
        &\leq  \|\phi\|_{\mathcal{S}_1} \max\left\{ \frac{\lambda_{1,n}}{|z-\lambda_{1,n}|},  \max_{\lambda_{1,j}\neq \lambda_{1,n}} \left| \frac{z}{z-\lambda_{1,j}}\right|+1 \right\}\\
        &\leq  \|\phi\|_{\mathcal{S}_1} \max\left\{ \frac{\lambda_{1,n}}{\frac{1}{2}\gamma_{n,\mathcal{S}_1}},  \frac{\lambda_{1,n}+\frac{1}{2}\gamma_{n,\mathcal{S}_1}}{\frac{1}{2}\gamma_{n,\mathcal{S}_1}}+1 \right\}\\
        &\leq  \|\phi\|_{\mathcal{S}_1}(2+2\frac{\lambda_{1,n}}{\gamma_{n,\mathcal{S}_1}}). 
    \end{aligned}
    \end{equation}
    Combining \eqref{Eq-prop04-07}-\eqref{Eq-prop04-09} yields the desired \eqref{Eq-prop04-02}.
\end{proof}

\section{Proof of Lemma~\ref{Lem6.3}}\label{prooflem2}
\begin{proof}[Proof of Lemma~\ref{Lem6.3}]
Since $\mathcal{D}^{\bm \alpha}_{\bm 0}[\phi]=0$ for all $|\bm \alpha|<r$, the Taylor expansion of $\phi$ at $\bm x=0$ contains no terms of order lower than $r$. Hence, by Taylor's theorem with integral remainder,
\begin{equation}\label{eq-Lem6.3-02}
\phi(\bm x)
=
\sum_{|\bm \alpha|=r}
\bm x^{\bm \alpha}\frac{r}{\bm \alpha!}
\int_0^1
(1-t)^{r-1}
\mathcal{D}^{\bm \alpha}[\phi](t\bm x)\dt
\end{equation}
where $\bm \alpha!:= \prod_{j=1}^{d}\alpha_j!$.
This yields the decomposition
\begin{equation}\label{eq-Lem6.3-03}
    \phi(\bm x)= \sum_{|\bm \alpha|=r} \bm x^{\bm \alpha}\phi_{\bm \alpha}(\bm x), \quad \phi_{\bm \alpha}= \frac{r}{\bm \alpha!} \int_0^1 (1-t)^{r-1}\mathcal{D}^{\bm \alpha}[\phi](t\bm x)\dt.
\end{equation}
Since $\bigl\|\mathcal{D}^{\bm \alpha}[\phi]\bigr\|_{H^{m}(B_R)}\lesssim \bigl\|\phi\bigr\|_{H^{m+r}(B_R)}\lesssim \|\phi\|_{H^{m+r}}$, the problem reduces to showing that for every $f \in H^{m}$ with $m>\max \{0,d/2-1\}$,
\begin{equation}\label{eq-Lem6.3-04}
    \Bigl\|\int_{0}^1(1-t)^{r-1}f(t\bm x)\dt\Bigr\|_{H^{m}(B_R)}\lesssim \|f\|_{H^{m}(B_R)}.
\end{equation}
Define 
\begin{equation}
 g(\bm x) = \int_{0}^1(1-t)^{r-1}f(t\bm x)\dt.   
\end{equation}
To bound $\|g\|_{H^{m}(B_R)}$, we use the Gagliardo–Nirenberg semi-norm representation for the fractional Sobolev space
\begin{equation}\label{eq-Lem6.3-05}
    \|g\|_{H^{m}(B_R)} \lesssim \|g\|_{L^{2}(B_R)}+\sum_{|\bm \alpha|=\lfloor m\rfloor}|\mathcal{D}^{\alpha}g|_{H^{\theta}(B_R)},
\end{equation}
where $\theta = m-\lfloor m\rfloor$ and the semi-norm is defined by
\[
    |f|_{H^{\theta}(B_R)}^2 := 
    \int_{B_R\times B_R}\frac{\bigl|f(\bm x)-f(\bm y)\bigr|^2}{|\bm x -\bm y|^{d+2\theta}}\mathrm{d} \bm x\mathrm{d}\bm y.
\]
First, the $L^2$ norm in \eqref{eq-Lem6.3-05} is bounded by Minkowski's integral inequality
\begin{equation}\label{eq-Lem6.3-06}
    \|g\|_{L^2(B_R)}\leq\int_0^1 \|f(t\cdot)\|_{L^{2}(B_R)}\dt = \int_{0}^{1}t^{-d/2}\|f\|_{L^{2}(B_{tR})}\dt.
\end{equation}
For any $p>1$, the H\"older inequality gives
\begin{equation}\label{eq-Lem6.3-07}
    \nm{f}_{L^{2}(B_{tR})}^2 \leq \left(\int_{B_{tR}} |f|^{2p}\right)^{1/p}\bigl|B_{tR}\bigr|^{1-\frac{1}{p}}\lesssim t^{d(1-\frac{1}{p})}\|f\|_{L^{2p}(B_R)}^2,
\end{equation}
where the constant in the last inequality is independent of $t$.
Since $m>\max\{0,d/2-1\}$, one can choose a constant $p$ such that 
$d/2-m<d/(2p)<\min\{d/2,1\}$. This choice guarantees $p>1,\ -d/2p>-1$, and $H^m(B_R)\hookrightarrow L^{2p}(B_R)$. Combining these with \eqref{eq-Lem6.3-06} and \eqref{eq-Lem6.3-07} then yields
\begin{equation}\label{eq-Lem6.3-08}
\begin{aligned}
    \|g\|_{L^{2}(B_R)}\lesssim \int_0^1 t^{-d/2p}\|f\|_{L^{2p}(B_R)}\dt \lesssim \|f\|_{H^{m}(B_R)}.
\end{aligned}
\end{equation}
Next, the semi-norm in \eqref{eq-Lem6.3-05} is bounded by
the scaling property for semi-norms
\[
    |h(t\cdot)|_{H^{\theta}(B_R)} = t^{\theta-d/2}| h(\cdot)|_{H^{\theta}(B_{tR})},
\]
and the Minkowski’s integral inequality
\begin{equation}\label{eq-Lem6.3-09}
\begin{aligned}
     |\mathcal{D}^{\bm \alpha}g|_{H^{\theta}(B_R)}
    &\lesssim
    \int_0^1 t^{|\bm \alpha|}\bigl|\mathcal{D}^{\bm \alpha}[f](t \cdot)\bigr|_{H^{\theta}(B_R)}\dt
    \\&\lesssim 
    \int_{0}^{1} t^{|\bm \alpha|+\theta-d/2}|\mathcal{D}^{\bm \alpha}f|_{H^{\theta}(B_{tR})}\dt.
\end{aligned}
\end{equation}
For $|\bm \alpha| =\lfloor m  \rfloor$, the factor $|\bm \alpha|+\theta-d/2= m-d/2>-1$, therefore \eqref{eq-Lem6.3-09} gives
\begin{equation}\label{eq-Lem6.3-10}
    \sum_{|\bm \alpha|=\lfloor m\rfloor}|\mathcal{D}^{\alpha}g|_{H^{\theta}(B_R)}\lesssim \sum_{|\bm \alpha|=\lfloor m\rfloor}|\mathcal{D}^{\alpha}f|_{H^{\theta}(B_R)}\lesssim \|f\|_{H^{m}(B_R)}.
\end{equation}
Combining \eqref{eq-Lem6.3-08} and \eqref{eq-Lem6.3-10} together completes the proof of \eqref{eq-Lem6.3-04}.
\end{proof}

\section{Proof of Lemma~\ref{Lem6.5}}\label{proofLem6.5}
\begin{proof}[Proof of Lemma~\ref{Lem6.5}]
First, define the expansion operator $\mathcal{Q}_r$ for $r=1,2$ by
\begin{equation}\label{eq-Lem7.5-01}
\mathcal{Q}_{r}[f](\bm x)
:=
\sum_{|\bm \alpha|< r}
\frac{1}{(2\pi i)^{|\bm \alpha|}}
\mathcal{D}^{\bm \alpha}_{\bm 0}[f]\,
m_{\bm \alpha}(\bm x),
\end{equation}
where $m_{\bm \alpha}$ is defined in \eqref{Eq-6.1-02}. The residual of the expansion \eqref{eq-Lem7.5-01} satisfies 
\begin{equation}\label{eq-Lem7.5-02}
\mathcal{R}_r[f]
:=
f-\mathcal{Q}_r[f],
\qquad
\mathcal{D}^{\bm \alpha}_{\bm 0}\!\left[\mathcal{R}_r[f]\right]=0,
\quad
|\bm \alpha|<r.
\end{equation}
Then, by Lemma~\ref{Lem6.2}, for $-1\leq t+r\leq s+r$,
\begin{equation}\label{eq-Lem7.5-03}
\mathcal{R}_r[f]\in H^{t+r+l}
\quad\Longrightarrow\quad
\mathcal{V}\mathcal{R}_r[f]\in H^{t+r}.
\end{equation}
The definition of $\mathcal{Q}_r$ requires the existence of
$\mathcal{D}^{\bm \alpha}_{\bm 0}[f]$ for all $|\bm \alpha|<r$, it is well defined for
$f\in H^{d/2-1+r}\hookrightarrow W^{r-1,\infty}$. Moreover, $\mathcal{Q}_r[f]\in C^{\infty}$ since $m_{\bm \alpha}\in C^{\infty}$. 
Therefore, we obtain for $s-1\leq t\leq s$ (where $t+l+r\geq s+l-1+r>d/2-1+r$)
\begin{equation}\label{eq-Lem7.5-04}
    f \in H^{t+r+l} 
    \quad \Longrightarrow\quad 
    \left\{
    \begin{aligned}
    &\mathcal{Q}_r[f] \in C^{\infty}\\
    &\mathcal{R}_r[f] \in H^{t+r+l}   
    \end{aligned}
    \right.
    \quad \Longrightarrow\quad 
    \mathcal{V}\mathcal{R}_r[f] \in H^{t+r}
\end{equation}

Next, for every $|\bm \alpha|\leq 1$, we construct an iterative sequence
$(\Phi_{\bm \alpha}^{(k)},\Psi_{\bm \alpha}^{(k)},R_{\bm \alpha}^{(k)})$
satisfying
\begin{equation}\label{eq-Lem7.5-05}
    \bigl(\mathcal{T}_{\star}+\mathcal{V}\bigr)
    \Phi_{\bm \alpha}^{(k)}
    =
    \mathcal{V}\Psi_{\bm \alpha}^{(k)}
    +R_{\bm \alpha}^{(k)},
    \qquad k\in \mathbb{N},
\end{equation}
where $\mathcal{T}_{\star}:=\mathcal{T}+1$. We initialize the iteration by
\begin{equation}\label{eq-Lem7.5-05.1}
    \Phi_{\bm \alpha}^{(0)}
    =
    \Psi_{\bm \alpha}^{(0)}
    =
    \frac{m_{\bm \alpha}}{(2\pi i)^{|\bm \alpha|}}.
\end{equation}
Then $R_{\bm \alpha}^{(0)}=\mathcal{T}_{\star}\Phi_{\bm \alpha}^{(0)}\in C^{\infty}$, and
$\mathcal{D}^{\bm \alpha'}_{\bm 0}
\bigl[\Phi_{\bm \alpha}^{(0)}\bigr]
=
\delta_{\bm \alpha',\bm \alpha}$ for every
$|\bm \alpha'|\leq |\bm \alpha|.$
To successively eliminate
$\mathcal{V}\Psi_{\bm \alpha}^{(k)}$ term on the right-hand side of \eqref{eq-Lem7.5-05} while preserving $\mathcal{D}_{\bm 0}^{\bm \alpha'}[\Phi_{\bm \alpha}^{(k)}]$ for $|\bm \alpha'|\leq |\bm \alpha|$, we update $\Phi_{\bm \alpha}^{(k)}$ using the residual operator $\mathcal{R}_{r_k}$,
\begin{equation}\label{eq-Lem7.5-06}
    \Phi^{(k+1)}_{\bm \alpha}
    :=
    \Phi_{\bm \alpha}^{(k)}
    -
    \mathcal{R}_{r_k}
    \mathcal{T}_{\star}^{-1}
    \mathcal{V}
    \Psi_{\bm \alpha}^{(k)},
\end{equation}
where $r_k \in \{1,2\}$ is determined later.
If $r_j\geq |\bm \alpha|+1$ for every $0\leq j\leq k$, then by \eqref{eq-Lem7.5-02},
\begin{equation}\label{eq-Lem7.5-07}
\mathcal{D}^{\bm \alpha'}_{\bm 0}\bigl[\Phi_{\bm \alpha}^{(k+1)}\bigr] = \mathcal{D}^{\bm \alpha'}_{\bm 0}\bigl[\Phi_{\bm \alpha}^{(k)}\bigr]=\cdots=\mathcal{D}^{\bm \alpha'}_{\bm 0}\bigl[\Phi_{\bm \alpha}^{(0)}\bigr]=\delta_{\bm \alpha',\bm \alpha},\quad \forall\, |\bm \alpha'|\leq |\bm \alpha|.
\end{equation}
Applying $\mathcal{T}_{\star}+\mathcal{V}$ on both sides of \eqref{eq-Lem7.5-06} and requiring the resulting equation to retain the form of \eqref{eq-Lem7.5-05}, we update
\begin{equation}\label{eq-Lem7.5-08}
    \Psi_{\bm \alpha}^{(k+1)} = -\mathcal{R}_{r_k}\mathcal{T}_\star^{-1}\mathcal{V}\Psi_{\bm \alpha}^{(k)},\quad R_{\bm \alpha}^{(k+1)}= R^{(k)}_{\bm \alpha}+\mathcal{T}_{\star}\mathcal{Q}_{r_k}\mathcal{T}_{\star}^{-1}\mathcal{V}\Psi_{\bm \alpha}^{(k)}.
\end{equation}

Lastly, we claim that, for a suitable choice of $r_k$, there exists $k_0$ such that
\begin{equation}\label{eq-Lem7.5-clm}
\mathcal{V}\Psi_{\bm \alpha}^{(k_0)} 
+
R_{\bm \alpha}^{(k_0)}
\in H^{s+2},\qquad 
\mathcal{D}_{\bm 0}^{\bm \alpha'}[\Phi_{\bm \alpha}^{(k_0)}]
=
\delta_{\bm \alpha',\bm \alpha}, 
\quad
|\bm \alpha'|\leq |\bm \alpha|.
\end{equation}
Then the proof is completed using this claim by setting $\Phi_{\bm \alpha}:=\Phi_{\bm \alpha}^{(k_0)}$. Note that the second part of the claim is satisfied using \eqref{eq-Lem7.5-07} if $r_k \geq |\bm \alpha|+1$ for every $k\in\mathbb{N}$. To establish the first part of the claim, we use \eqref{eq-Lem7.5-08} to obtain
\begin{equation}\label{eq-Lem7.5-09}
    \Psi^{(k+1)}_{\bm \alpha}
    =
    -\mathcal{R}_{r_k}\mathcal{T}_{\star}^{-1}\mathcal{V}\Psi_{\bm \alpha}^{(k)}
    =
    (-1)^{k+1}
    \mathcal{R}_{r_k}
    \Bigl(
    \prod_{j=0}^{k-1}
    \mathcal{T}_{\star}^{-1}\mathcal{V}\mathcal{R}_{r_j}
    \Bigr)
    \mathcal{T}_{\star}^{-1}\mathcal{V}\Psi_{\bm \alpha}^{(0)}.
\end{equation}
By \eqref{eq-Lem7.5-04}, the operator
$\mathcal{T}_{\star}^{-1}\mathcal{V}\mathcal{R}_r$
maps $H^{t+r+l}$ into $H^{t+r+2}$ for
$s-1\leq t\leq s$.
Since $r_j\in \mathbb{N}$ is chosen from $|\bm \alpha|+1\leq r_j\leq 2$ in  \eqref{eq-Lem7.5-09}, 
$\mathcal{T}_{\star}^{-1}\mathcal{V}\mathcal{R}_{r_j}$
maps $H^{t'+l}$ into $H^{t'+2}$ for
$t'\in \cup_{r_j}[s+r_j-1,\,s+r_j]=[s+|\bm \alpha|,\,s+2]$. Iterating this bootstrap finitely many times, we obtain a $k_0\in\mathbb N$ and $r_j$ satisfies $|\bm \alpha|+1\leq r_j\leq 2$ such that the operator 
$\prod_{j=0}^{k_0-2}   \mathcal{T}_{\star}^{-1}\mathcal{V}\mathcal{R}_{r_j}$
maps $H^{s+|\bm \alpha|+l}$ to $H^{s+4}$. Note that by Lemma~\ref{Lem6.1} (iii), the initialization \eqref{eq-Lem7.5-05} satisfies
    \[
        \mathcal{V}\Phi_{\bm \alpha}^{(0)} = m_{\bm \alpha}V/(2\pi i)^{|\bm \alpha|} \in H^{s+|\bm \alpha|} 
        \quad \Longrightarrow \quad 
        \mathcal{T}_{\star}^{-1}\mathcal{V}\Phi_{\bm \alpha}^{(0)}\in H^{s+|\bm \alpha|+2}\hookrightarrow H^{s+|\bm \alpha|+l},
    \]
we have
\begin{equation}\label{eq-Lem7.5-10}
    \Bigl(
    \prod_{j=0}^{k_0-2}
    \mathcal{T}_{\star}^{-1}\mathcal{V}\mathcal{R}_{r_j}
    \Bigr)
    \mathcal{T}_{\star}^{-1}\mathcal{V}\Psi_{\bm \alpha}^{(0)} \in H^{s+4}.
\end{equation}
Since $H^{s+4}\hookrightarrow H^{s+2+l}$, choosing
$r_{k_0-1}=2$ and applying \eqref{eq-Lem7.5-04} once more yields $\mathcal{V}\Psi^{(k_0)}_{\bm \alpha}\in H^{s+2}$. Since each $r_j$ is chosen so that \eqref{eq-Lem7.5-04} applies, we have
\[
\mathcal{Q}_{r_j}\mathcal{T}_{\star}^{-1}\mathcal{V}\Psi^{(j)}_{\bm \alpha}
\in C^{\infty}.
\]
Hence \eqref{eq-Lem7.5-08} implies $R^{(k_0)}_{\bm \alpha}\in C^{\infty}$. Combined with $\mathcal{V}\Psi_{\bm \alpha}^{(k_0)}\in H^{s+2}$, this establishes the claim \eqref{eq-Lem7.5-clm}.
\end{proof}

\section{Step 2 in the proof of Lemma~\ref{Lem2.1}}
\label{proofthmAR1}
\begin{proof}[proof of \textbf{Step 2}]
$r=2\Rightarrow \psi^{\geq 2} \in H^{s+4}$.
    In \textbf{Step 1}, we derive that 
    \[ \|\psi^{\geq 1}\|_{H^{s+3}}\lesssim\|\psi\|_{L^2},\quad | \beta_{p,\bm 0}|\lesssim \|\psi\|_{L^2}.
    \]
    For $|\bm \alpha| = 1$, the coefficients $\beta_{p,\bm \alpha}$ are bounded by the  embedding $H^{s+3}\hookrightarrow W^{1,\infty}$
    \begin{equation}\label{Eq-lem2.1-15}
    \begin{aligned}
        |\beta_{p,\bm \alpha}|
        &=\bigl|\mathcal{D}^{\bm \alpha}_p[\psi^{\geq 1}+\sum_{q\neq p}\beta_{q,\bm 0}\Phi_{q,\bm 0}]\bigr|
        \leq \bigl\|\psi^{\geq 1}+\sum_{q\neq p}\beta_{q,\bm 0}\Phi_{q,\bm 0}^{\mathrm{G}}\bigr\|_{H^{s+3}}
        \lesssim \|\psi\|_{L^2}.
    \end{aligned}
    \end{equation}
     Since $\Phi_{p,\bm \alpha}\in PSH^{s+3,1}\hookrightarrow H^{s+3}$ for $|\bm \alpha|=1$. we use \eqref{Eq-lem2.1-15} to obtain
     \begin{equation}\label{Eq-lem2.1-16}
        \|\psi^{(1)}\|_{H^{s+3}} =\Bigl\| \sum_{p=1}^{n_p}\sum_{\bm |\alpha|  =1}\beta_{p,\bm \alpha}\Phi_{p,\bm \alpha}\Bigr\|_{H^{s+3}}\lesssim \|\psi\|_{L^2}.
    \end{equation}
    Consequently,
    \begin{equation}
        \|\psi^{\geq 2}\|_{H^{s+3}} =\|\psi^{\geq 1} - \psi^{(1)}\|_{H^{s+3}}\lesssim \|\psi\|_{L^2},
    \end{equation}
    which serves as the starting point of the bootstrap argument. 
    
    Let $t$ satisfies $s+1\leq t+2 \leq s+2$. Taking $H^{t+2}$ norm on both sides of \eqref{Eq-lem2.1-03} gives
    \begin{equation}\label{Eq-lem2.1-17}
        \| \psi^{\geq 2} \|_{H^{t+4}} \lesssim \sum_{p=1}^{n_p}\Bigl\|\mathcal{V}_p\bigl(\psi-\sum_{|\bm \alpha|\leq 1}\beta_{p,\bm \alpha}\Phi_{p,\bm \alpha}\bigr)\Bigr\|_{H^{t+2}}+\|R(\lambda)\|_{H^{t+2}}.
    \end{equation}
    The second term is bounded using \eqref{Eq-lem2.1-05} and \eqref{Eq-lem2.1-15} by
    \begin{equation}\label{Eq-lem2.1-18}
        \|R(\lambda)\|_{H^{t+2}} \lesssim \bigl(\lambda+\|V_0\|_{H^{s+2}}\bigr)\|\psi\|_{H^{t+2}} +\sum_{p=1}^{n_p}\sum_{|\bm \alpha|\leq 1} |\beta_{p,\bm \alpha}|\|R_{p,\bm \alpha}\|_{H^{t+2}}\lesssim \|\psi\|_{L^2}.
    \end{equation}
    To bound the first term, we localize the problem as in \textbf{Step 1}. By Lemma~\ref{Lem6.4}  with $\delta = r_m/2$, we have the decomposition for $|\bm \alpha|\leq 1$
    \[
    \Phi_{p,\bm \alpha} = \Phi_{p,\bm \alpha}^\mathrm{L} +\Phi_{p,\bm \alpha}^\mathrm{G},\quad\Phi_{p,\bm \alpha}^\mathrm{L}\in PSH^{s+2+|\bm \alpha|,2-|\bm \alpha|},\quad\Phi_{p,\bm \alpha}^\mathrm{G}\in H^{s+4}.
    \]
    Note that $B_p\cap B_q = \emptyset, \forall\, p\neq q$, we have  
    \[
        \psi(\bm x)-\sum_{|\bm \alpha|\leq 1}\beta_{p,\bm \alpha}\Phi_{p,\bm \alpha}(\bm x) = \psi^{\geq 2}(\bm x) +\sum_{q\neq p} \sum_{|\bm \alpha|\leq 1}\beta_{q,\bm \alpha}\Phi_{q,\bm \alpha}^{\mathrm{G}}(\bm x), \quad \forall\, \bm x\in B_p.
    \]
     Applying Lemma~\ref{Lem6.2} with $r=2$ to bound the $V_p^{L}$ term, we have 
     \begin{equation}\label{Eq-lem2.1-19}
       \begin{aligned}
            \Bigl\|\mathcal{V}_p\bigl(\psi-\sum_{|\bm \alpha|\leq 1}\beta_{p,\bm \alpha}\Phi_{p,\bm \alpha}\bigr)\Bigr\|_{H^{t+2}}
            &\lesssim  \|V_{p}^{\mathrm{G}}\|_{H^{s+2}}\|\psi-\sum_{|\bm \alpha|\leq 1}\beta_{p,\bm \alpha}\Phi_{p,\bm \alpha}\|_{H^{t+2}}\\
            &+ \|V_p^{\mathrm{L}}\|_{PSH^{s,2}}\Bigl\|\psi^{\geq 2}+\sum_{q\neq p}\sum_{|\bm \alpha|\leq 1}\beta_{q,\bm 0}\Phi_{q,\bm 0}^{\mathrm{G}}\Bigr\|_{H^{t+l+2}}\\
            &\lesssim \|\psi\|_{L^2}+\|\psi^{\geq 2}\|_{H^{t+l+2}}.
       \end{aligned}
    \end{equation}
    Combining \eqref{Eq-lem2.1-19},\eqref{Eq-lem2.1-18}, and \eqref{Eq-lem2.1-17} together, using same bootstrap argument used in \textbf{Step 1} yields $\|\psi^{\geq 2}\|_{H^{s+4}}\lesssim \|\psi\|_{L^2}$. Combining this with \eqref{Eq-lem2.1-06} and \eqref{Eq-lem2.1-16} completes the proof of Lemma~\ref{Lem2.1}.
\end{proof}

\end{appendices}



\end{document}